\documentclass[10pt]{amsart}

\usepackage{xr-hyper}

\usepackage[utf8]{inputenc}         %
\usepackage[T1]{fontenc}            %

\usepackage{a4wide}                 %
\usepackage{amsmath, amssymb}       %
\usepackage{amsthm}                 %
\usepackage{appendix}               %

\usepackage{tikz}                   %
\usepackage{tikz-cd}                %
\tikzcdset{ %
arrow style=tikz
}
\usetikzlibrary{arrows}             %

\usepackage{graphicx}               %
\usepackage{stackrel}               %

\usepackage{color}                  %
\usepackage[normalem]{ulem}         %
\usepackage{enumitem}               %
\setlist[itemize]{leftmargin=*}
\setlist[enumerate]{leftmargin=*,label=\roman*),ref=\roman*)}
\newlist{subenumerate}{enumerate}{2}
\setlist[subenumerate]{leftmargin=*,label=\alph*),ref=\alph*)}

\usepackage{blkarray}               %

\definecolor{darkblue}{rgb}{0,0,0.6} %
\usepackage[ocgcolorlinks,colorlinks=true, citecolor=darkblue, filecolor=darkblue, linkcolor=darkblue, urlcolor=darkblue]{hyperref}         %

\usepackage[alphabetic]{amsrefs}

\usepackage{stix}                   %
\usepackage{bbm}                    %
\usepackage{eucal}                  %

\DeclareFontFamily{T1}{cbgreek}{}
\DeclareFontShape{T1}{cbgreek}{m}{n}{<-6>  grmn0500 <6-7> grmn0600 <7-8> grmn0700 <8-9> grmn0800 <9-10> grmn0900 <10-12> grmn1000 <12-17> grmn1200 <17-> grmn1728}{}
\DeclareSymbolFont{quadratics}{T1}{cbgreek}{m}{n}
\DeclareMathSymbol{\qoppa}{\mathord}{quadratics}{19}
\DeclareMathSymbol{\Qoppa}{\mathord}{quadratics}{21}

\makeatletter
\let\old@font@info\@font@info
\def\@font@info#1{%
\expandafter\ifx\csname\detokenize{#1}\endcsname\relax
  \old@font@info{#1}%
\fi
\expandafter\xdef\csname\detokenize{#1}\endcsname{}%
}
\makeatother

\usepackage{comment}                %
\usepackage{xspace}                 %

\newtheoremstyle{thms}
	{}{}{\itshape}{}{\bfseries }{.}{ }
	{\thmname{#1} \thmnumber{#2}. \thmnote{\bfseries{[#3]}}}
\newtheoremstyle{thms2}
	{}{}{\itshape}{}{\bfseries }{.}{ }
	{}
\newtheoremstyle{ithreethms}
	{}{}{\itshape}{}{\bfseries }{}{ }
	{\thmname{#1} \thmnumber{#2}. \thmnote{\bfseries{[#3]}}}

\newtheoremstyle{name}
	{}{}{\itshape}{}{\bfseries }{.}{ }
	{\thmname{#1}\thmnumber{#2}\thmnote{\bfseries{[#3]}}}

\newtheoremstyle{defs}
	{}{}{\normalfont}{}{\bfseries }{.}{ }
	{\thmname{#1} \thmnumber{#2}. \thmnote{\bfseries{(#3)}}}
\newtheoremstyle{defs2}
	{}{12pt}{\normalfont}{}{\bfseries }{.}{ }
	{\thmname{#1}\thmnumber{#2}. \thmnote{\bfseries{(#3)}}}
\newtheoremstyle{ithreedefs}
	{}{}{\normalfont}{}{\bfseries }{}{ }
	{\thmname{#1} \thmnumber{#2}. \thmnote{\bfseries{(#3)}}}

\newtheoremstyle{rmk}
	{}{}{\normalfont}{}{\itshape }{.}{ }
        {}

\newtheoremstyle{claim}
	{}{}{\normalfont}{}{\itshape}{.}{ }
        {\thmname{#1} \thmnumber{#2}. \thmnote{#3}}

\theoremstyle{thms2}

\theoremstyle{thms2}

\newtheorem*{namedthm}{\namedthmname}
\newcounter{namedthm}

\makeatletter
{\endnamedthm}
\makeatother

\theoremstyle{rmk}

\theoremstyle{ithreethms}
\newtheorem{ithm}{Theorem}
\newtheorem{icor}[ithm]{Corollary}

\theoremstyle{ithreedefs}

\theoremstyle{thms2}

\swapnumbers         %

\newcounter{rthree}
\theoremstyle{defs}
\newtheorem{definition-r-three}{Definition}[rthree]
\newtheorem{notation-r-three}[definition-r-three]{Notation}
\newtheorem{remark-r-three}[definition-r-three]{Remark}
\newtheorem{example-r-three}[definition-r-three]{Example}

\theoremstyle{thms}
\newtheorem{proposition-r-three}[definition-r-three]{Proposition}
\newtheorem{corollary-r-three}[definition-r-three]{Corollary}
\newtheorem{theorem-r-three}[definition-r-three]{Theorem}

\newcounter{rfour}
\theoremstyle{defs}
\newtheorem{definition-r-four}{Definition}[rthree]
\newtheorem{notation-r-four}[definition-r-four]{Notation}
\newtheorem{remark-r-four}[definition-r-four]{Remark}
\newtheorem{example-r-four}[definition-r-four]{Example}

\theoremstyle{thms}
\newtheorem{proposition-r-four}[definition-r-four]{Proposition}
\newtheorem{corollary-r-four}[definition-r-four]{Corollary}
\newtheorem{theorem-r-four}[definition-r-four]{Theorem}

\theoremstyle{thms}
\newtheorem{proposition}{Proposition}[subsection]
\newtheorem{theorem}[proposition]{Theorem}
\newtheorem{lemma}[proposition]{Lemma}
\newtheorem{corollary}[proposition]{Corollary}

\theoremstyle{defs}
\newtheorem{definition}[proposition]{Definition}%
\newtheorem{notation}[proposition]{Notation}

\newtheorem{construction}[proposition]{Construction}
\newtheorem{example}[proposition]{Example}

\newtheorem{remark}[proposition]{Remark}

\theoremstyle{defs2}

\theoremstyle{rmk}
\newtheorem*{Rmk}{Remark}

\theoremstyle{claim}

\newcommand{\defi}[1]{\emph{#1}}                 %

\newcommand{\Del}{\Delta}

\newcommand{\Om}{\Omega}

\newcommand{\Sig}{\Sigma}

\newcommand{\vphi}{\varphi}
\newcommand{\eps}{\epsilon}

\newcommand{\lrar}{\longrightarrow}
\newcommand{\lto}{\longrightarrow}

\newcommand{\hrar}{\hookrightarrow}

\newcommand{\st}{\stackrel}

\newcommand{\adj}{\mathbin{%
\begin{tikzpicture}[baseline,thick] 
\coordinate (source) at (0ex,.5ex);
\coordinate (target) at (3ex,.5ex);
\draw[->] ([yshift=1ex]source) -- ([yshift=1ex]target); 
\draw[->] ([yshift=-.5ex]target) -- ([yshift=-.5ex]source);
\node at (1.5ex,.8ex) {$\scriptscriptstyle \perp$};
\end{tikzpicture}%
}}

\newcommand{\cocolon}{\nobreak \mskip6mu plus1mu \mathpunct{}\nonscript\mkern-\thinmuskip {:}\mskip2mu \relax}

\makeatletter
\newcommand{\dovl}[1]{\overline{\dbl@overline{#1}}}
\newcommand{\dbl@overline}[1]{\mathpalette\dbl@@overline{#1}}
\newcommand{\dbl@@overline}[2]{%
  \begingroup
  \sbox\z@{$\m@th#1\overline{#2}$}%
  \ht\z@=\dimexpr\ht\z@-2\dbl@adjust{#1}\relax
  \box\z@
  \ifx#1\scriptstyle\kern-\scriptspace\else
  \ifx#1\scriptscriptstyle\kern-\scriptspace\fi\fi
  \endgroup
}
\newcommand{\dbl@adjust}[1]{%
  \fontdimen8
  \ifx#1\displaystyle\textfont\else
  \ifx#1\textstyle\textfont\else
  \ifx#1\scriptstyle\scriptfont\else
  \scriptscriptfont\fi\fi\fi 3
}
\makeatother

\newcommand{\NN}{\mathbb{N}}               %
\newcommand{\RR}{\mathbb{R}}               %
\newcommand{\CC}{\mathbb{C}}               %
\newcommand{\ZZ}{\mathbb{Z}}               %
\newcommand{\QQ}{\mathbb{Q}}               %
\newcommand{\FF}{\mathbb{F}}               %

\newcommand{\nO}{\mathcal{O}}               %

\newcommand{\im}{\mathrm{im}}              %
\newcommand{\coker}{\mathrm{coker}}        %

\newcommand{\odd}{\mathrm{odd}}            %

\newcommand{\spec}{\mathrm{spec}}          %

\newcommand{\fp}{\mathfrak{p}}             %

\newcommand{\Div}{\mathrm{Div}}            %
\newcommand{\val}{\mathrm{val}}            %

\newcommand{\cd}{\mathrm{cd}}              %
\newcommand{\vcd}{\mathrm{vcd}}            %

\newcommand{\Hom}{\operatorname{Hom}}            %
\newcommand{\Aut}{\operatorname{Aut}}            %

\newcommand{\fib}{\operatorname{fib}}            %
\newcommand{\cof}{\operatorname{cof}}            %

\newcommand{\colim}{\mathop{\mathrm{colim}}}     %
\newcommand{\core}{\mathrm{Cr}}                  %

\newcommand{\op}{^\mathrm{op}}                   %

\newcommand{\cp}{\omega}                         %
\newcommand{\perf}{\mathrm{p}}                   %

\newcommand{\pr}{\mathrm{pr}}                %
\newcommand{\id}{\mathrm{id}}                %

\newcommand{\Spa}{{\mathcal Sp}}             %
\newcommand{\cwedge}{{\scriptscriptstyle\wedge}} %
\renewcommand{\hom}{\operatorname{hom}}        %
\newcommand{\map}{\operatorname{hom}}          %

\renewcommand{\SS}{\mathbb{S}}                 %

\newcommand{\h}{\mathrm{h}}                    %
\newcommand{\Ct}{\mathrm{C_2}}                 %
\newcommand{\hC}{{\h\Ct}}                      %
\newcommand{\tC}{{\mathrm{t}\Ct}}              %

\newcommand{\Ab}{\mathcal{A}b}                   %

\newcommand{\Vect}{\mathrm{Vect}}                %

\newcommand{\Ring}{\mathrm{Ring}}                %

\newcommand{\Mod}{\operatorname{Mod}}            %

\newcommand{\fg}{{\mathrm{f}}}                   %

\newcommand{\Proj}{\mathrm{Proj}}                %
\newcommand{\Free}{\mathrm{Free}}                %

\newcommand{\Modcp}{\Mod^{\cp}}                  %

\newcommand{\Modp}[1]{\Modcp({#1})}              %

\newcommand{\Torf}{\mathrm{Tor}^{\mathrm{f}}}    %

\newcommand{\Unimod}{\mathrm{Unimod}}            %

\newcommand{\Ering}{\mathrm{E}}			 %
\newcommand{\Ek}[1]{{\Ering_{#1}}}               %
\newcommand{\Einf}{\Ek{\infty}}          	 %
\newcommand{\Eone}{\Ek{1}}               	 %

\newcommand{\Ch}{\operatorname{Ch^b}}            %

\newcommand{\DHom}{\mathcal{H}om}                %

\newcommand{\GEM}{\mathrm{H}}                    %

\newcommand{\Der}{\mathcal D}                    %
\newcommand{\Dfree}{{\mathcal D}^\fg}            %
\newcommand{\Dperf}{{\mathcal D}^\perf}          %

\newcommand{\rH}{\mathrm{H}}                     %

\newcommand{\et}{\text{ét}}                      %

\newcommand{\Catp}{\mathrm{Cat}^{\mathrm p}_\infty}     %

\newcommand{\Fun}{\operatorname{Fun}}            %

\newcommand{\Motp}[1][]{{\if\relax\detokenize{#1}\mathrm{Mot^p}\relax\else\mathrm{Mot}_{#1}^{\mathrm{p}}\fi}}  %
\newcommand{\Motpun}[1][]{\mathrm{Mot^p_{un\if\relax\detokenize{#1}\relax\else{,}#1\fi}}}           %
\newcommand{\Motkar}[1][]{{\if\relax\detokenize{#1}\mathrm{Mot^{k}}\relax\else\mathrm{Mot^k_{#1}}\fi}}  %
\newcommand{\Motpbord}[1][]{{\if\relax\detokenize{#1}\mathrm{Mot^{pw}}\relax\else\mathrm{Mot}_{#1}^{\mathrm{pw}}\fi}} %
\newcommand{\Mot}[1][]{{\if\relax\detokenize{#1}\mathrm{Mot}\relax\else\mathrm{Mot_{#1}}\fi}}            %

\newcommand{\MotpE}[2]{\mathrm{Mot^p_{{#2}\if\relax\detokenize{#1}\relax\else{,}#1\fi}}}           %

\newcommand{\MotkwE}[2]{\mathrm{Mot^{kw}_{{#2}\if\relax\detokenize{#1}\relax\else{,}#1\fi}}}           %

\newcommand{\Motwun}[1][]{\mathrm{Mot^w_{un\if\relax\detokenize{#1}\relax\else{,}#1\fi}}}           %
\newcommand{\MotwE}[2]{\mathrm{Mot^w_{{#2}\if\relax\detokenize{#1}\relax\else{,}#1\fi}}}           %

\newcommand{\MotpwE}[2]{\mathrm{Mot^{pw}_{{#2}\if\relax\detokenize{#1}\relax\else{,}#1\fi}}}           %

\newcommand{\Motw}[1][]{{\if\relax\detokenize{#1}\mathrm{Mot^w}\relax\else\mathrm{Mot^w_{#1}}\fi}}  %

\newcommand{\Motpw}[1][]{{\if\relax\detokenize{#1}\mathrm{Mot^{pw}}\relax\else\mathrm{Mot^{pw}_{#1}}\fi}}  %

\newcommand{\Motkw}[1][]{{\if\relax\detokenize{#1}\mathrm{Mot^{kw}}\relax\else\mathrm{Mot^{kw}_{#1}}\fi}}  %

\newcommand{\MotkarE}[2]{\mathrm{Mot^k_{{#2}\if\relax\detokenize{#1}\relax\else{,}#1\fi}}}           %

\newcommand{\GL}{\mathrm{GL}}                    %
\newcommand{\BGL}{\mathrm{BGL}}                  %
\newcommand{\rO}{\mathrm{O}}                     %
\newcommand{\BU}{\mathrm{BU}}                    %
\newcommand{\BO}{\mathrm{BO}}                    %
\newcommand{\Sp}{\mathrm{Sp}}                    %
\newcommand{\BSp}{\mathrm{BSp}}                  %

\newcommand{\Pic}{\operatorname{Pic}}            %

\newcommand{\ko}{\mathrm{ko}}                    %
\newcommand{\K}{\operatorname K}                 %

\newcommand{\topo}{\mathrm{top}}                 %
\renewcommand{\L}{\operatorname L}               %
\newcommand{\short}{\mathrm{short}}              %
\newcommand{\W}{\operatorname W}                 %
\newcommand{\GW}{\operatorname{GW}}              %

\newcommand{\U}{\operatorname{U}}                %
\newcommand{\V}{\operatorname{V}}                %

\newcommand{\fgt}{\mathrm{fgt}}                  %

\newcommand{\fpm}{\lambda}			 %
\newcommand{\gfpm}{{\g\fpm}}

\newcommand{\B}{\mathrm{B}}                      %

\newcommand{\Bil}{\mathrm{B}}               %
\newcommand{\Lin}{\Lambda}               %

\newcommand{\Dual}{\mathrm{D}}              %

\newcommand{\ev}{\mathrm{ev}}                    %

\newcommand{\inv}[1]{\overline{#1}}              %

\newcommand{\HqZ}{\mathrm{H}_{\mathrm{q}}}       %
\newcommand{\Hms}{\mathrm{H}_{-\mathrm{s}}}      %

\newcommand{\qshift}[1]{^{[#1]}}                 %
\newcommand{\Met}{\operatorname{Met}}            %
\newcommand{\met}{\mathrm{met}}                  %

\newcommand{\hyp}{\mathrm{hyp}}                  %

\newcommand{\cl}{\mathrm{cl}}                    %

\newcommand{\sym}{\mathrm{s}}                    %
\newcommand{\s}{\mathrm{s}}                      %
\newcommand{\qdr}{\mathrm{q}}                     %
\newcommand{\g}{\mathrm{g}}                      %
\newcommand{\gs}{\mathrm{gs}}                    %
\newcommand{\gq}{\mathrm{gq}}                    %
\newcommand{\gev}{\mathrm{ge}}                   %
\newcommand{\tate}{\mathrm{t}}                   %

\newcommand{\burn}{\mathrm{b}}                   %

\newcommand{\uni}{\mathrm{u}}                    %

\newcommand{\Poinc}{\mathrm{Pn}}                 %
\newcommand{\pM}{\mathcal{M}}                    %

\newcommand{\I}{\mathcal{I}}               %
\newcommand{\J}{\mathcal{J}}               %
\newcommand{\E}{\mathcal{E}}               %

\newcommand{\M}{\mathcal{M}}               %

\newcommand{\C}{\mathcal C}                %
\newcommand{\Ctwo}{{{\mathcal C}'}}        %
\newcommand{\D}{\mathcal{D}}               %

\newcommand{\F}{\mathcal{F}}               %

\newcommand{\QF}{\Qoppa}                   %
\newcommand{\QFtwo}{{\QF'}}                %
\newcommand{\QFD}{\Phi}                    %

\newcommand{\Qgen}[2]{\QF^{\geq #1}_{#2}}  %

\newcommand{\qone}{q}                      %

\makeatletter

\renewcommand{\tocsection}[3]{%
\indentlabel{\@ifnotempty{#2}{\parbox[b]{3ex}{\bfseries\ignorespaces#1 #2}}}\bfseries#3} 

\renewcommand{\tocsubsection}[3]{%
\indentlabel{\@ifnotempty{#2}{\hspace{1.6em}\parbox[b]{5ex}{\ignorespaces#1 #2}}}#3}

\renewcommand{\tocsubsubsection}[3]{%
\indentlabel{\@ifnotempty{#2}{\hspace{3.9em}\parbox[b]{5ex}{\ignorespaces#1 #2}}}#3}

\makeatother

\DeclareRobustCommand{\SkipTocEntry}[5]{} 

\newcommand{\introsubsection}[1]{\addtocontents{toc}{\SkipTocEntry}\subsection*{#1}}
\newcommand{\bibsubsubsection}[1]{\addtocontents{toc}{\SkipTocEntry}\subsubsection*{#1}}

\BibSpec{collection.article}{%
    +{}  {\PrintAuthors}                {author}
    +{,} { \textit}                     {title}
    +{.} { }                            {part}
    +{:} { \textit}                     {subtitle}
    +{,} { \PrintContributions}         {contribution}
    +{,} { \PrintConference}            {conference}
    +{}  {\PrintBook}                   {book}
    +{,} { }                            {booktitle}
    +{,} { }                            {series}
    +{,} { \voltext}                    {volume}
    +{,} { }                            {publisher}
    +{,} { }                            {organization}
    +{,} { }                            {address}
    +{,} { \PrintDateB}                 {date}
    +{,} { \eprintpages}                {pages}
    +{,} { }                            {status}
    +{,} { available at \eprint}        {eprint}
    +{}  { \parenthesize}               {language}
    +{}  { \PrintTranslation}           {translation}
    +{;} { \PrintReprint}               {reprint}
    +{.} { }                            {note}
    +{.} {}                             {transition}
    +{}  {\SentenceSpace \PrintReviews} {review}
}

\BibSpec{article}{%
    +{}  {\PrintAuthors}                {author}
    +{.} { \textit}                     {title}
    +{.} { }                            {part}
    +{:} { \textit}                     {subtitle}
    +{,} { \PrintContributions}         {contribution}
    +{.} { \PrintPartials}              {partial}
    +{,} { }                            {journal}
    +{}  { \textbf}                     {volume}
    +{} { \PrintDate}                  {date}
    +{,} { \issuetext}                  {number}
    +{,} { \eprintpages}                {pages}
    +{,} { }                            {status}
    +{,} { available at \eprint}        {eprint}
    +{}  { \parenthesize}               {language}
    +{}  { \PrintTranslation}           {translation}
    +{;} { \PrintReprint}               {reprint}
    +{.} { }                            {note}
    +{.} {}                             {transition}
    +{}  {\SentenceSpace \PrintReviews} {review}
}

\newcommand{\refthree}[1]{\ref{#1}}
\newcommand{\eqrefthree}[1]{\eqref{#1}}
\newcommand{\refthreeitem}[1]{\ref{#1}}

\newcommand{\refone}[1]{\cite{Part-one}.\ref{I-#1}}

\newcommand{\reftwo}[1]{\cite{Part-two}.\ref{II-#1}}

\newcommand{\reftwoitem}[1]{\ref{II-#1}}

\newcommand{\paperone}{Paper~\cite{Part-one}\xspace}
\newcommand{\papertwo}{Paper~\cite{Part-two}\xspace}

\newcommand{\paperfour}{Paper~\cite{Part-four}\xspace}

\newcommand{\refundefined}[1]{\textcolor{red}{Undefined ref}}

\title[Hermitian K-theory for stable $\infty$-categories III: Grothendieck-Witt groups of rings]{Hermitian K-theory for stable $\infty$-categories III:\\
Grothendieck-Witt groups of rings}

\author[Calmès]{Baptiste Calmès}
\address{Université d'Artois, Laboratoire de Mathématiques de Lens, Lens, France}
\email{baptiste.calmes@univ-artois.fr}

\author[Dotto]{Emanuele Dotto}
\address{University of Warwick, Mathematics Institute, Coventry, United Kingdom}
\email{emanuele.dotto@warwick.ac.uk}

\author[Harpaz]{Yonatan Harpaz}
\address{Université Paris Cité, Sorbonne Université, Paris, France}
\email{harpaz@imj-prg.fr}

\author[Hebestreit]{Fabian Hebestreit}
\address{Universität Bielefeld, Fakultät für Mathematik, Bielefeld, Germany}
\email{hebestreit@math.uni-bielefeld.de}

\author[Land]{Markus Land}
\address{Johannes Gutenberg University Mainz, Institute of Mathematics, Mainz, Germany}
\email{mland@uni-mainz.de}

\author[Moi]{Kristian Moi}
\address{KTH, Institutionen för matematik, Stockholm, Sweden}
\email{kristian.moi@gmail.com}

\author[Nardin]{Denis Nardin}
\address{Universität Regensburg, Mathematisches Institut, Regensburg, Germany}

\author[Nikolaus]{Thomas Nikolaus}
\address{WWU Münster, Mathematisches Institut, Münster, Germany}
\email{nikolaus@uni-muenster.de}

\author[Steimle]{Wolfgang Steimle}
\address{Universität Augsburg, Institut für Mathematik, Augsburg, Germany}
\email{wolfgang.steimle@math.uni-augsburg.de}

\dedicatory{To Andrew Ranicki.}

\date{\today}

\begin{document}

\begin{abstract}
We establish a fibre sequence relating the classical Grothendieck-Witt theory of a ring $R$ to the homotopy $\Ct$-orbits of its K-theory and Ranicki's original (non-periodic) symmetric L-theory.
We use this fibre sequence to remove the assumption that $2$ is a unit in $R$ from various results about Grothendieck-Witt groups. For instance, we solve the homotopy limit problem for Dedekind rings whose fraction field is a number field, 
calculate the various flavours of Grothendieck-Witt groups of $\ZZ$, show that the Grothendieck-Witt groups of rings of integers in number fields are finitely generated, and that the comparison map from quadratic to symmetric Grothendieck-Witt theory of coherent rings of global dimension $d$ is an equivalence in degrees $\geq d+3$. 
As an important tool, we establish the hermitian analogue of Quillen's localisation-dévissage sequence for Dedekind rings and use it to solve a conjecture of Berrick-Karoubi.

\end{abstract}

\maketitle
\tableofcontents

\section*{Introduction}
This paper investigates the hermitian K-theory spectra of non-degenerate symmetric and quadratic forms over a ring $R$ and their homotopy groups: the \emph{higher Grothendieck-Witt groups} of $R$.
Many structural and computational features of the higher Grothendieck-Witt groups of rings $R$ in which 2 is a unit are well understood, prevalently due to extensive work of Karoubi \cite{K1, K2, K3, K4, Karoubi-Le-theoreme-fondamental} and Schlichting \cite{schlichting-exact, schlichting-mv, schlichting-derived, SchlichtinghigherI}. 
Previously, in \papertwo we have used the categorical framework of Poincaré $\infty$-categories to establish some fundamental properties of the higher Grothendieck-Witt groups of rings in which $2$ is not necessarily a unit, most notably a form of Karoubi periodicity \reftwo{corollary:karoubi-fundamental} and the existence of a fibre sequence relating the Grothendieck-Witt theory of any Poincaré $\infty$-category with its algebraic $\K$-theory and  L-theory. This allows for the separation of $\K$-theoretic and $\L$-theoretic arguments, and the theme of this paper is to deduce results about Grothendieck-Witt theory from their counterparts in L-theory.
\medskip

\introsubsection{Main results}
Let $R$ be a unital and associative, but not necessarily commutative ring.
Let $\Dual$ be a duality on the category $\Proj(R)$ of finitely generated projective left $R$-modules.
Then $\Dual$ is necessarily of the form $\Dual P=\hom_R(P,M)$, where $M:=\Dual R$ is an invertible $\ZZ$-module with involution (see Definition~\refthree{definition:modinv}).
An $M$-valued unimodular symmetric form on $P$ is then a self-dual isomorphism $\varphi \colon P \to \Dual P$. Together with their isomorphisms, these form a groupoid $\Unimod(R;M)$, symmetric monoidal under orthogonal direct sum. The classical symmetric Grothendieck-Witt theory of $R$ is its group-completion:
\[
\GW_{\cl}^\sym(R;M)=(\Unimod(R;M),\oplus)^{gp}. 
\]
By construction, $\GW_{\cl}^\sym(R;M)$ is a group-like $\E_\infty$-space, which we view equivalently as a connective spectrum. Its homotopy groups are the higher symmetric Grothendieck-Witt groups $\GW_{\cl,\ast}^\sym(R;M)$ of $R$. Similarly, one can consider $\GW_{\cl}^{\qdr}(R;M)$, the variant of $\GW_{\cl}^\sym(R;M)$ where $M$-valued symmetric bilinear forms are replaced by $M$-valued quadratic forms, whose higher homotopy groups are the higher quadratic Grothendieck-Witt groups.

After inverting $2$, it turns out that $\GW_{\cl}^\qdr(R;M)[\tfrac{1}{2}] \simeq \GW_\cl^\sym(R;M)[\tfrac{1}{2}]$, and the study of the higher Grothendieck-Witt groups of $R$ reduces to the study of the K-groups and Witt groups, as by work of Karoubi, there is a natural splitting 
\[
\GW_{\cl,\ast}^\sym(R;M)[\tfrac12]\cong (\K_\ast(R;M)[\tfrac12])^{\Ct}\oplus (\W_{\ast}(R;M)[\tfrac12]),
\]
see also \cite{BF}. Here $\K(R;M)$ is the K-theory spectrum of $R$ with $\Ct$-action induced by sending  $P$ to its dual $\Dual P$, and the first summand is the subgroup of invariants of its homotopy groups with $2$ inverted. The second summand consists of the Witt groups of symmetric forms and formations, which are $4$-periodic by definition.
The first main result of the present paper combines the general fibre sequence of \papertwo with Ranicki's algebraic surgery to obtain an integral version of this result.

\begin{ithm}
\label{theorem:fiber-sequence-intro-three}%
For every ring $R$ and duality $\Dual=\hom_R(-,M)$ on $\Proj(R)$, there is a fibre sequence of spectra
\[
\K(R;M)_{\hC} \stackrel{\hyp}{\lto} \GW_{\cl}^{\sym}(R;M) \lto \L^{\short}(R;M)
\]
where $\L^{\short}(R;M)$ is a canonical connective spectrum whose homotopy groups are Ranicki's original (non-4-periodic) symmetric L-groups from \cite{RanickiATS1}.
\end{ithm}
After inverting 2, this fibre sequence recovers Karoubi's splitting of $\GW_\cl^\sym(R;M)$, but it also allows to efficiently treat the behaviour of Grothendieck-Witt theory at the prime $2$, as we will explain below. Without inverting 2 on the outside, but when $2$ is a unit in $R$, Ranicki's L-groups $\L^{\short}_\ast(R;M)$ are still $4$-periodic and isomorphic to the Witt groups $\W_{\ast}(R;M)$, and in this case the sequence of Theorem~\refthree{theorem:fiber-sequence-intro-three} is due to Schlichting \cite{schlichting-derived}*{\S 7}.
However, if 2 is not invertible in $R$, there are several variants of L-spectra in addition to $\L^{\short}(R;M)$, most notably the 4-periodic symmetric L-theory $\L^{\sym}(R;M)$ used by Ranicki in later work \cite{Ranickiblue}. Our insight is that it is the non-periodic classical symmetric L-theory of Ranicki \cite{RanickiATS1} which makes Theorem~\refthree{theorem:fiber-sequence-intro-three} true for all rings.

Coming back to the 2-local behaviour of Grothendieck-Witt theory, we note that sending a symmetric bilinear form to its underlying finitely generated projective module leads to a canonical map $\GW^{\sym}_\cl(R;M) \to \K(R;M)^{\hC}$. 
The question whether this map is a 2-adic equivalence in positive degrees is known as Thomason's homotopy limit problem \cite{thomason}, which admits a positive solution for many rings in which 2 is invertible, notably by work of Hu, Kriz and Ormsby \cite{HKO-homotopy-limit}, Bachmann and Hopkins \cite{bachmannperiodic}, and Berrick, Karoubi, Schlichting and \O stv\ae r \cite{BKSO-fixed-point}. In \S\refthree{subsection:homotopy-limit} we will show:

\begin{ithm}
\label{theorem:homotopy-limit-intro-three}%
Let $R$ be a Dedekind ring whose fraction field is a number field. Then the canonical map $\GW_{\cl}^{\sym}(R;M) \to \K(R;M)^{\hC}$ is a 2-adic equivalence in non-negative degrees.
\end{ithm}
To the best of our knowledge this is the first general result on the homotopy limit problem for a class of rings which are not fields and in which 2 is not assumed to be a unit. The strategy we adopt to prove Theorem~\refthree{theorem:homotopy-limit-intro-three} is to use Theorem~\refthree{theorem:fiber-sequence-intro-three} to reduce it to the case of  $R[\tfrac{1}{2}]$, where it holds by \cite{BKSO-fixed-point}.
For a general ring $R$, we further observe that the failure of 4-periodicity of $\L^{\short}(R;M)$ in high degrees provides a purely L-theoretic obstruction for the homotopy limit problem map $\GW^{\sym}_{\cl}(R;M) \to \K(R;M)^{\hC}$ to be a 2-adic equivalence in positive degrees; see Proposition~\refthree{proposition:obstruction-to-HLP}. 

The polarisation of a quadratic form induces a comparison map $\GW_{\cl}^{\qdr}(R;M) \to \GW_{\cl}^{\sym}(R;M)$, and in \S\refthree{subsection:surgery-symmetric} we show that for coherent rings of finite global dimension, this map is an equivalence in high degrees:
\begin{ithm}
\label{theorem:GW-classical-quad-sym}%
Suppose $R$ is a coherent ring of finite global dimension $d$. Then the map $\GW_{\cl,n}^{\qdr}(R;M) \to \GW_{\cl,n}^{\sym}(R;M)$ is injective for $n \geq d+2$ and an isomorphism for $n \geq d+3$. 
Moreover, if $R$ is 2-torsion free and $M=R$, the map is injective for $n \geq d$ and an isomorphism for $n \geq d+1$. 
\end{ithm}

Here, by slight abuse of terminology, by a coherent ring we mean a left-coherent ring, and likewise finite global dimension refers to left-global dimension. The same result is, however, true for right-coherent rings of finite right-global dimension, see Remark~\refthree{remark:right-coherent-variant}. In addition, similar statements hold for Grothendieck-Witt groups associated to any form parameter in the sense of Bak in place of quadratic forms in the above theorem, see Remark~\refthree{remark:form-paramters-and-symmetric-surgery} for details.

Theorem~\refthree{theorem:fiber-sequence-intro-three} does not only provide a conceptual description of symmetric Grothendieck-Witt spectra, but it can also be used for explicit calculations.
For instance,
when $R=\ZZ$ there are two dualities on $\Proj(\ZZ)$, leading to the symmetric  and symplectic Grothendieck-Witt groups of $\ZZ$, respectively. In \S\refthree{subsection:integers} we explicitly calculate these groups in a range of degrees $<20000$, and beyond that conditionally on the Kummer-Vandiver conjecture in the following sense: Of some Grothendieck-Witt groups, we can only determine the order, and the Kummer-Vandiver conjecture implies that these groups are cyclic. 
\begin{ithm}
\label{theorem:Z}%
The symmetric and symplectic Grothendieck-Witt groups of $\ZZ$ are given in the table of  Theorem~\refthree{theorem:gwsz}. 
\end{ithm}

Finally, using in addition Theorem~\refthree{theorem:GW-classical-quad-sym} and explicit low dimensional calculations, we also obtain the quadratic and skew-quadratic Grothendieck-Witt groups of $\ZZ$ in Theorems~\refthree{theorem:gwqz} and \refthree{theorem:gw-gqz}.

\introsubsection{Proof strategy and further results}
We approach $\GW_{\cl}^{\sym}$ by investigating Grothendieck-Witt theory in the general context of Poincaré $\infty$-categories, as defined by Lurie \cite{Lurie-L-theory} and further developed in \paperone, \papertwo.
We briefly recall that a Poincaré $\infty$-category consists of a small stable $\infty$-category $\C$ equipped with a Poincaré structure, that is a functor $\QF \colon \C\op \to \Spa$ which is quadratic and satisfies a non-degeneracy condition, which allows to extract an induced duality  $\Dual \colon \C\op \to\C$. 
We refer to \paperone for a general introduction to Poincaré $\infty$-categories, and to \papertwo for the construction of their Grothendieck-Witt and L-spectra and their universal properties.
In \reftwo{corollary:tate-square-L}, we showed that for any Poincaré $\infty$-category $(\C,\QF)$ there is a natural fibre sequence
\begin{equation}
\label{equation:quadsequence}%
 \K(\C,\QF)_{\hC} \lto \GW(\C,\QF) \lto \L(\C,\QF).
\end{equation}
where $\K(\C,\QF)$ is the K-theory spectrum of $\C$ with the $\Ct$-action induced by $\Dual$.
To connect this general fibre sequence to Theorem~\refthree{theorem:fiber-sequence-intro-three}, we will be concerned with studying appropriate Poincaré structures on the derived $\infty$-category of perfect complexes $\Dperf(R)$. 
Some immediate examples of Poincaré structures on  $\Dperf(R)$ are the quadratic and symmetric Poincaré structures given at 
a perfect complex $X$ by the formulae
\begin{equation}
\label{equation:homotopy-forms}%
\QF^{\qdr}_M(X) = \map_{R \otimes R}(X \otimes X,M)_{\hC} \quad \text{ and } \quad \QF^{\sym}_M(X) = \map_{R\otimes R}(X\otimes X,M)^{\hC},
\end{equation}
where $M$ is an invertible $\ZZ$-module with involution over $R$ (see Definition~\refthree{definition:modinv}), and the $\Ct$-action is given by conjugating the flip action on $X \otimes X$ and the $\Ct$-action on $M$. These two Poincaré structures are the homotopy theoretic analogues of quadratic and symmetric forms in algebra, which on a finitely generated projective $R$-module $P$ are respectively the groups of coinvariants and invariants
\begin{equation}
\label{equation:forms}%
\Hom_{R\otimes R}(P\otimes P,M)_{\Ct} \quad \text{ and } \quad \Hom_{R\otimes R}(P \otimes P,M)^{\Ct}
\end{equation}
for the same $\Ct$-action as above.
One insight in our series of papers is that the abstract framework of \paperone, \papertwo allows us to work with Poincaré structures on $\Dperf(R)$ which are more intimately related to algebra than the naive homotopy theoretic constructions of \eqrefthree{equation:homotopy-forms}. 
These are the non-abelian derived functors of the algebraic constructions of \eqrefthree{equation:forms}, which we call the \emph{genuine} quadratic and \emph{genuine}  symmetric Poincaré structures and that we denote respectively by $\QF^{\gq}_M$ and $\QF^{\gs}_M$. There are canonical comparison maps
\[
\QF^{\qdr}_M \lto \QF^{\gq}_M \lto \QF^{\gs}_M \lto \QF^{\sym}_M
\]
relating these Poincaré structures.
When 2 is a unit in $R$, all of them are equivalences, and we showed in Appendix~\reftwo{appendix:AppIIB} that the corresponding Grothendieck-Witt spectra coincide with previous constructions of Grothendieck-Witt spectra due to Schlichting and Spitzweck \cite{schlichting-derived, Spitzweck-GW}. 
In general, when 2 is not necessarily a unit, the fourth and ninth authors \cite{comparison} relate the genuine Grothendieck-Witt spectra $\GW^{\gs}(R;M) := \GW(\Dperf(R);\QF^{\gs}_M)$ and $\GW^{\gq}(R;M) := \GW(\Dperf(R);\QF^{\gq}_M)$ to the classical ones,
 by providing natural equivalences
\[
\GW^{\sym}_{\cl}(R;M) \stackrel{\simeq}{\lto} \tau_{\geq0}\GW^{\gs}(R;M)   \quad \text{ and } \quad \GW_{\cl}^{\qdr}(R;M) \stackrel{\simeq}{\lto} \tau_{\geq0}\GW^{\gq}(R;M),
\]
where $\GW^\qdr_{\cl}(R;M)$ denotes, similary to $\GW^\sym_{\cl}(R;M)$, the group completion of the category of unimodular quadratic forms, and $\tau_{\geq0}$ denotes the connective cover. Writing similarly $\L^{\gs}(R;M)$ for $\L(\Dperf(R);\QF^{\gs}_M)$, we therefore obtain a fibre sequence $\K(R;M)_{\hC} \to \GW^{\gs}(R;M)  \to \L^{\gs}(R;M)$, and
Theorem~\refthree{theorem:fiber-sequence-intro-three} is then implied by the following result, see Theorem~\refthree{theorem:main-theorem-L-theory}.
\begin{ithm}
\label{theorem:Ranicki-coincides}%
For any ring $R$ and non-negative integer $n$, the genuine symmetric L-groups $\L^\gs_n(R;M)$ are canonically isomorphic to Ranicki's original symmetric L-groups from \cite{RanickiATS1}. Thus, in the notation of Theorem~\refthree{theorem:fiber-sequence-intro-three}, we have $\L^{\short}(R;M) = \tau_{\geq 0} \L^\gs(R;M)$.
\end{ithm}
We recall that the original symmetric L-groups of Ranicki are defined so that elements of the $n$'th L-group are represented by Poincaré chain complexes of length at most $n$, for $n\geq 0$. 
Ranicki then defines negative symmetric L-groups in an ad hoc manner, and we show that these negative L-groups are also canonically isomorphic to the corresponding negative genuine symmetric L-groups: Concretely they are given by
\[
\L^{\gs}_{n}(R;M) = \begin{cases} \L^{\ev}_{n+2}(R;-M) & \text{ if } n=-2,-1 \\ \L^{\qdr}_n(R;M) & \text{ if } n\leq -3, \end{cases}
\]
where $\L^{\ev}_{\ast}$ and $\L^{\qdr}_\ast$ are respectively the even and quadratic L-groups of  \cite{RanickiATS1}.
In particular, Theorem~\refthree{theorem:Ranicki-coincides} and the described addendum show that the classical symmetric L-groups can be realised as the homotopy groups of the non-connective spectrum $\L^\gs(R;M)$. The general form of Karoubi periodicity of \papertwo, which we review in Theorem~\refthree{theorem:Karoubi} below, relates the Poincaré structures $\QF^{\gs}_M$ and $\QF^{\gq}_M$ and their $\GW$ and L-spectra, in particular showing that $\Sigma^4\L^{\gs}(R;M) \simeq \L^{\gq}(R;M)$; see Corollary~\refthree{corollary:periodicity-L}. From the fibre sequence for general Poincaré \(\infty\)-categories, we therefore also obtain a quadratic version of Theorem~\refthree{theorem:fiber-sequence-intro-three}, given by the fibre sequence
\[
\K(R;M)_{\hC}\stackrel{\hyp}{\lto}  \GW_\cl^{\qdr}(R;M)\lto\tau_{\geq 0}(\Sigma^4\L^{\gs}(R;M)).
\]

We prove Theorem~\refthree{theorem:Ranicki-coincides} in \S~\refthree{subsection:surgery-quadratic} using Ranicki's procedure of algebraic surgery, which allows us to compare the L-groups of various Poincaré structures in a range of degrees. We discuss this technique also for connective ring spectra in Corollary~\refthree{corollary:pi-pi}, and in \S~\refthree{subsection:surgery-symmetric}, to obtain the following comparison result. We will write $\GW^\sym(R;M) := \GW^\sym(\Dperf(R);\QF^\sym_M)$ for the homotopy symmetric Grothendieck-Witt theory, and write likewise $\L^\sym(R;M) := \L(\Dperf(R);\QF^\sym_M)$ for periodic symmetric L-theory. 
\begin{ithm}
\label{theorem:gs-and-s-agreement-range}%
Suppose $R$ is a coherent ring of finite global dimension $d$. Then:
\begin{enumerate}
\item
\label{item:range-gs-s}%
the map $\L^{\gs}_n(R;M) \to \L_n^\sym(R;M)$ is injective for $n \geq d-2$ and an isomorphism for $n \geq d-1$,
\item 
the map $\L^{\gq}_n(R;M) \to \L^{\gs}_n(R;M)$ is injective for $n \geq d+2$ and an isomorphism for $n \geq d+3$.
\end{enumerate}
\end{ithm}
Part~\refthreeitem{item:range-gs-s} of Theorem~\refthree{theorem:gs-and-s-agreement-range}, together with Theorem~\refthree{theorem:Ranicki-coincides}, improve a similar comparison result of Ranicki \cite{RanickiATS1}*{Proposition 4.5}, where he proves injectivity for non-negative $n \geq 2d-3$ and bijectivity for non-negative $n\geq 2d-2$ for Noetherian rings of finite global dimension $d$. Combining Theorem~\refthree{theorem:fiber-sequence-intro-three} and Theorem~\refthree{theorem:gs-and-s-agreement-range}, we obtain Theorem~\refthree{theorem:GW-classical-quad-sym} from above, see also Corollary~\refthree{corollary:improve} and Remark~\refthree{remark:improve-two}.

Furthermore, part \refthreeitem{item:range-gs-s} of Theorem~\refthree{theorem:gs-and-s-agreement-range} implies that the map $\GW^\sym_{\cl}(R;M) \to \tau_{\geq 0}\GW^\sym(R;M)$ is an equivalence if $R$ is a Dedekind domain.
Thus in order to study the classical Grothendieck-Witt groups of Dedekind rings, it suffices to study the homotopy symmetric Grothendieck-Witt theory $\GW^\sym$. This is an interesting invariant in its own right which enjoys pleasant properties not shared with the genuine variant $\GW^\gs$. Most notably, we prove in Theorem~\refthree{theorem:devissage} that 4-periodic symmetric L-theory $\L^\sym$, and hence also $\GW^\sym$, satisfies a dévissage theorem. In particular, we obtain the hermitian analogue of Quillen's famous localisation-dévissage fibre sequence \cite{quillen}, see Corollary~\refthree{corollary:decomposition-s}:
\begin{ithm}
\label{theorem:localisation-devissage-Dedekind}%
Let $R$ be a Dedekind ring, $T\subset R$ a multiplicative subset, and $\mathbb{F}_\mathfrak{p}$ the residue field $R/\mathfrak{p}$ at a maximal ideal $\mathfrak{p} \subseteq R$. Then restriction, localisation, and a choice of uniformiser for every $\mathfrak{p}$ with $\fp \cap T \neq \emptyset$ induce a fibre sequence of spectra
\[
\bigoplus_{\fp\cap T\neq\emptyset}\GW^\sym(\mathbb{F}_\mathfrak{p};(M/\fp)[-1]) \lto \GW^\sym(R;M) \lto \GW^\sym(R[T^{-1}];{R[T^{-1}]\otimes_R M}),
\]
where $R[T^{-1}]$ is obtained from $R$ by inverting the elements of $T$.
\end{ithm}
We in fact construct a more general fibre sequence for localisations of $R$ away from a set of non-empty prime ideals of $R$ in a formulation that does not depend on choices of uniformisers, see Corollary~\refthree{corollary:decomposition-s}.
This result establishes a conjecture of Berrick and Karoubi which asserts that the map $\ZZ \to \ZZ[\tfrac{1}{2}]$ induces an equivalence on the positive, 2-localised Grothendieck-Witt groups \cite{berrick-karoubi}. In fact, this result holds for general rings of integers in number fields as we observe in Proposition~\refthree{proposition:BK-conjecture}. 
In \S~\refthree{section:HLP} we then combine Theorem~\refthree{theorem:localisation-devissage-Dedekind} with work of Berrick, Karoubi, Schlichting and \O stv\ae r \cite{BKSO-fixed-point} to deduce Theorem~\refthree{theorem:homotopy-limit-intro-three}, as well as the calculations for the integers of Theorem~\refthree{theorem:Z}.

Finally, we also use Theorem~\refthree{theorem:fiber-sequence-intro-three}, together with a calculation of the symmetric and quadratic L-groups of Dedekind rings to deduce the following finiteness result; see Corollary~\refthree{corollary:finite-generation-GW}. When $M=R$ with the involution given by multiplication by $\eps=\pm1$, we write $\GW^\sym_{\cl,n}(R;\eps)$ for $\GW^\sym_{\cl,n}(R;M)$, and similarly for $\GW^\qdr_{\cl,n}(R;\eps)$. 
\begin{icor}
\label{corollary:finite-generation-GW-number-rings}%
Let $\mathcal{O}$ be a number ring, that is, a localisation of the ring of integers in a number field away from finitely many primes, and $\eps = \pm 1$. Then its classical $\eps$-symmetric and $\eps$-quadratic Grothendieck-Witt groups $\GW^\sym_{\cl,n}(\mathcal{O};\eps)$ and $\GW^\qdr_{\cl,n}(\mathcal{O};\eps)$ are finitely generated. 
\end{icor}

In the quadratic case, one can prove this result also through homological stability, but in the generality presented here the argument is not known to carry over to the symmetric case, as we explain in Remark~\refthree{remark:homological-stability}.

\begin{Rmk}
\label{remark:Schlichting-results}%
Some of the results presented above have also been announced in \cite{Schlichting-integers} and its erratum \cite{Schlichting-wrong}:
The calculations of the Grothendieck-Witt groups of the integers of Theorem~\refthree{theorem:Z} in the symmetric, symplectic and quadratic cases, the localisation-dévissage sequence of Theorem~\refthree{theorem:localisation-devissage-Dedekind} in non-negative degrees, and Theorem~\refthree{theorem:GW-classical-quad-sym} for the ring $R=\ZZ$ with the trivial involution.
\end{Rmk}

\introsubsection{Notation and Conventions}
When not stated explicitly otherwise, the symbol $\otimes$ denotes the derived tensor product over $\ZZ$.
We always denote by $\Dual=\hom_R(-,M)$ the dualities on $\Proj(R)$ and $\Dperf(R)$ determined by an invertible $\ZZ$-module with involution $M$.

\introsubsection{Acknowledgements}
For useful discussions about our project,
we heartily thank
Tobias Barthel,
Clark Barwick,
Lukas Brantner,
Mauricio Bustamante,
Denis-Charles Cisinski,
Dustin Clausen,
Diarmuid Crowley,
Uriya First,
Rune Haugseng,
André Henriques,
Lars Hesselholt,
Gijs Heuts,
Geoffroy Horel,
Marc Hoyois,
Max Karoubi,
Daniel Kasprowski,
Ben Knudsen,
Manuel Krannich,
Achim Krause,
Henning Krause,
Sander Kupers,
Wolfgang Lück,
Ib Madsen,
Cary Malkiewich,
Mike Mandell,
Akhil Mathew,
Lennart Meier,
Irakli Patchkoria,
Nathan Perlmutter,
Maxime Ramzi,
Oscar Randal-Williams,
Andrew Ranicki,
George Raptis,
Marco Schlichting,
Peter Scholze,
Stefan Schwede,
Graeme Segal,
Markus Spitzweck,
Jan Steinebrunner,
Georg Tamme,
Ulrike Tillmann,
Michael Weiss,
Christoph Winges,
and
Maria Yakerson.

Furthermore, we owe a tremendous intellectual debt to Jacob Lurie for creating the framework we exploit here, and to Søren Galatius for originally spotting the overlap between various separate projects of ours; his insight ultimately led to the present collaboration.

Finally, it is a pleasure to heartily thank an anonymous referee for many helpful suggestions.

\medskip

The authors would also like to thank the Hausdorff Center for Mathematics at the University of Bonn, the Newton Institute at the University of Cambridge, the University of Copenhagen and the Mathematical Research Institute Oberwolfach for hospitality and support while parts of this project were undertaken.
\medskip

BC was supported by the French National Centre for Scientific Research (CNRS) through a ``délégation'' at LAGA, University Paris 13, and by the french French National Research Agency (ANR) project ``Motivic homotopy, quadratic invariants and diagonal classes'' (ANR grant no.\ 21-CE40-0015) at the University of Burgundy.
ED was supported by the German Research Foundation (DFG) through the priority program ``Homotopy theory and Algebraic Geometry'' (DFG grant no.\ SPP 1786) at the University of Bonn and WS by the priority program ``Geometry at Infinity'' (DFG grant no.\ SPP 2026) at the University of Augsburg. ED was further supported by the Engineering and Physical Sciences Research Council (EPSRC) through the grant ``Characteristic polynomials for symmetric forms'' (EPSRC grant no.\ EP/W019620/1) at the University of Warwick.
YH was supported by the European Research Council (ERC) through the project ``Foundations of Motivic Real K-Theory'' (ERC grant no.\ 949583) at Paris Cité University.
YH and DN were supported by the ANR through the grant ``Chromatic Homotopy and K-theory'' (ANR grant no.\ 16-CE40-0003) at LAGA, University of Paris 13.
FH was a member of the Hausdorff Center for Mathematics (DFG grant no.\ EXC 2047 390685813) at the University of Bonn and TN of the cluster ``Mathematics Münster: Dynamics-Geometry-Structure'' (DFG grant no.\ EXC 2044 390685587) at the University of Münster. 
FH was further supported by the DFG through the collaborative research center ``Integral structures in Geometry and Representation Theory'' (DFG grant no.\ TRR 358 4913924) at the University of Bielefeld, TN through the centre ``Geometry: Deformations and Rigidity" (DFG grant no.\ SFB 1442 427320536) at the University of Münster and ML and DN through the centre ``Higher Invariants'' (DFG grant no.\ SFB 1085 224262486) at the University of Regensburg. 
FH was also supported by the ERC through the project ``Moduli spaces, Manifolds and Arithmetic'' (ERC grant no.\ 682922) of Søren Galatius and KM by the project ``$\K$-theory, $\L^2$-invariants, manifolds, groups and their interactions'' (ERC grant no.\ 662400) of Wolfgang Lück. 
FH, TN and WS were further supported by the EPSRC through the program ``Homotopy harnessing higher structures'' at the Isaac Newton Institute for Mathematical Sciences (EPSRC grants no.\ EP/K032208/1 and EP/R014604/1). 
ML was supported by the research fellowship ``New methods in algebraic K-theory'' (DFG grant no.\ 424239956) and by the Danish National Research Foundation (DNRF) through the Center for Symmetry and Deformation (DNRF grant no.\ 92) and the Copenhagen Centre for Geometry and Topology (DNRF grant no.\ 151) at the University of Copenhagen. 
KM was also supported by the K\&A Wallenberg Foundation at the University of Stockholm.

\section*{Recollection}
In this section we recall some of the material from \paperone and \papertwo on Poincaré structures on the perfect derived $\infty$-category of a ring and their Grothendieck-Witt and L spectra, as we rely on this framework in the rest of the paper. We will also review the general form of Karoubi's periodicity Theorem~\reftwo{corollary:karoubi-fundamental}, which does not require that $2$ is a unit in the base ring.

In \papertwo, we view Grothendieck-Witt theory as an invariant of what we call a \emph{Poincaré} $\infty$-\emph{category}.  A Poincaré $\infty$-category is a pair $(\C,\QF)$ consisting of a small stable $\infty$-category $\C$ equipped with a \emph{Poincaré structure} $\QF$, that is a functor $\QF\colon \C\op \to \Spa$ which is reduced and $2$-excisive in the sense of Goodwillie's functor calculus, and whose symmetric cross-effect $\B\colon \C\op\times \C\op\to\Spa$ is of the form $\B(X,Y)=\map_{\C}(X,\Dual Y)$ for some equivalence of categories $\Dual\colon\C\op {\to}\C$. 
Poincaré $\infty$-categories were introduced by Lurie as a novel framework for Ranicki's L-theory (see \cite{Lurie-L-theory} and \papertwo). A Poincaré structure provides a formal notion of ``hermitian form'' on the objects of $\C$. Indeed, there is as a space of Poincaré objects $\Poinc(\C,\QF)$ which consists of pairs $(X,q)$ where $X$ is an object of $\C$ and $q\in \Omega^{\infty}\QF(X)$ is such that a certain canonical map $X\to \Dual X$ is an equivalence (see Definition~\refone{definition:poinc-forms}). There are then canonical transformations 
\[
\Poinc(\C,\QF)\to \Omega^{\infty}\GW(\C,\QF)\quad\text{and}\quad  \Poinc(\C,\QF)\to \Omega^{\infty}\L(\C,\QF)
\]
which exhibit the Grothendieck-Witt and L-theory functors as the universal approximation of $\Poinc$ by a Verdier localising, respectively a bordism invariant additive, functor (see Observation~\reftwo{observation:univ-GW-space} and Theorem~\reftwo{theorem:lisbordgw}). 
These universal properties are similar to the universal property of the map from the groupoid core to K-theory $\core\C\to \K(\C)$ of a small stable $\infty$-category provided by \cite{BGT}.

In the present paper we will be concerned with the perfect derived $\infty$-category $\Dperf(R)$ of a ring $R$, which is the $\infty$-categorical localisation of the category of bounded chain complexes of finitely generated projective (left) $R$-modules at the quasi-isomorphisms, or equivalently the $\infty$-category of compact objects of the localisation $\D(R)$ of all chain complexes at the quasi-isomorphisms.
 Given a Poincaré structure $\QF\colon \Dperf(R)\op\to \Spa$, we will denote the corresponding Grothendieck-Witt spectrum by
\[
\GW(R;\QF):=\GW(\Dperf(R),\QF).
\]
We are going to consider a specific collection of Poincaré structures on $\Dperf(R)$ associated to modules with involution, which we now introduce. For what follows, let $\otimes^\mathrm{U}$ denote the underived tensor product of rings over $\ZZ$.
Given an $R{\otimes^\mathrm{U}}R$-module $M$, we let $M\op$ denote the $R{\otimes^\mathrm{U}}R$-module defined by $M$ with the module action $r\otimes s\cdot m:=s\otimes r\cdot m$ for all $r,s$ in $R$ and $m$ in $M$.

\begin{definition-r-three}
\label{definition:modinv}%
A $\ZZ$-\defi{module with involution over $R$} is an $R\otimes^\mathrm{U} R$-module $M$ together with an $R{\otimes^\mathrm{U}}R$-module map $\inv{\bullet}\colon M^{op}\to M$ such that $\bar{\bar{m}}=m$. We say that $M$ is \emph{invertible} if it is finitely generated projective for either of its $R$-module structures, and the map 
\[
R\longrightarrow \Hom_R(M,M)
\]
which sends $1$ to $\inv{\bullet}$ is an isomorphism, where $M$ is regarded as an $R$-module via the first $R$-factor in the source, and the second one in the target. 
\end{definition-r-three}

The notion of (invertible) $\ZZ$-modules with involution over $R$ is discussed in \S\refone{subsection:discrete-rings} and appears in Definitions~\refone{definition:ordinary-module-with-involution} and \refone{definition:ordinary-invertible-module}.
In Definition~\refone{definition:module-with-involution}, we have also defined a notion of invertible $\GEM\ZZ$-modules with involution over $R$ which is slightly more general: In the terminology of loc.\ cit.\ we would view $R$ as an $\Eone$-algebra over the $\Einf$-ring $\GEM\ZZ$, or equivalently as a dg-algebra over \(\ZZ\), and form the derived tensor product $R\otimes R$. We write \(\D(R \otimes R)\) for its derived \(\infty\)-category, obtained as in the discrete case by localising the category of \(R \otimes R\)-dg-modules at the quasi-isomorphisms. An invertible $\GEM\ZZ$-module with involution over $R$ is then an $R\otimes R$-module in $\GEM\ZZ$-modules with $\Ct$-action, which is \emph{perfect} in either of its two $R$-module structures, and such that the canonical map $R\to \map_{\D(R)}(M,M)$ is an equivalence. Restricting along the canonical ring map $R \otimes R \to R{\otimes^\mathrm{U}}R$ from the derived to the underived tensor product, an invertible $\ZZ$-module with involution in the sense of gives rise to an invertible $\GEM\ZZ$-module with involution over $R$. In fact, invertible $\ZZ$-modules with involution over $R$ are precisely those invertible $\GEM\ZZ$-modules with involution 
which are finitely generated projective, rather than merely perfect, in either $R$-module structure. This is because the restriction functor $\D(R{\otimes^\mathrm{U}}R) \to \D(R \otimes R)$ is fully faithful on discrete objects and $\map_{\D(R)}(M,M)$ is equivalent to $\Hom_R(M,M)$ if $M$ is projective. We work with the stronger notion of invertible $\ZZ$-modules with involution of because for many of our arguments we need the associated duality on $\Dperf(R)$ to preserve finitely generated projective modules, and this is the case if and only if $M$ itself is finitely generated projective.

\begin{example-r-three}
\label{example:moduleswithinv}%
\
\begin{enumerate}
\item When $R$ is commutative, any line bundle $L$ over $R$ gives rise to an invertible $\ZZ$-module with involution over $R$, with $M=L$ and $\inv{\bullet}=\id$.
\item Let $\epsilon\in R$ be a unit. We recall that an $\eps$-involution on $R$ consists of a ring isomorphism \(\inv{\bullet} \colon R \to R\op\) such that \(\bar{\bar{r}}=\eps r \eps^{-1}\) and \(\inv{\eps}=\eps^{-1}\). In this case $M=R$ equipped with the $R\otimes^\mathrm{U} R$-module structure $r\otimes s\cdot x=rx\inv{s}$ and the involution $\epsilon(\inv{\bullet})$ is an invertible $\ZZ$-module with involution over $R$, that we denote by $R(\eps)$. This is the structure commonly used by Ranicki as input for L-theory \cite{RanickiATS1}.
\item Given a $\ZZ$-module with involution $M$ over $R$, we can define a new $\ZZ$-module with involution over $R$ denoted $-M$, with the same underlying $R\otimes^\mathrm{U} R$-module $M$ but with involution $-(\overline{\bullet})$. In the case where $M=R$ we have by definition that $-R=R(-1)$.
\item
\label{item:4}%
If $M$ is an invertible $\ZZ$-module with involution over $R$, then $M^\vee = \hom_R(M,R)$ is canonically an invertible $\ZZ$-module with involution over $R\op$, see also Remark~\refthree{remark:passing-to-opposite-ring}.
\end{enumerate}
\end{example-r-three}

For every pair of objects $X$ and $Y$ of $\Dperf(R)$, viewing $M$ as an $R\otimes R$-module as explained above, we may form the mapping spectrum
\[
\B(X,Y):=\map_{R\otimes R}(X\otimes Y,M)
\]
in the derived \(\infty\)-category $\D(R\otimes R)$, 
where the tensor product \(X \otimes Y\) is also to be understood as derived over \(\ZZ\).
Then $\B$ is a symmetric bilinear functor, so the spectrum $\B(X,X)$ inherits a $\Ct$-action by conjugating the flip action on \(X \otimes X\) and the involution of $M$; see \S\refone{subsection:modules-with-involution}. 

Given a spectrum with $\Ct$-action $X\colon B\Ct\to \Spa$, we denote by $X^{\hC}$ and $X_{\hC}$ its homotopy fixed points and homotopy orbits, respectively. Similarly, we let $X^{\tC}$ denote its Tate construction, defined as the cofibre of the norm map $N\colon X_{\hC}\to X^{\hC}$ as defined in \cite{HA}*{\S 6.1.6}, see also \cite{NS}*{I.1.11}. We consider the Tate construction \(\GEM M^{\tC}\) of the  Eilenberg-MacLane spectrum associated to \(M\) as an object of  \(\Der(R)\), as follows. The spectrum \(\GEM M^{\tC}\) is equipped with an action of \([\GEM(R \otimes R)]^{\tC} \simeq [\GEM R \otimes_{\GEM \ZZ} \GEM R]^{\tC}\), that we restrict to an action of \(\GEM R\) via the \(\GEM \ZZ\)-linear Tate diagonal \(\GEM R \to [\GEM R \otimes_{\GEM \ZZ} \GEM R]^{\tC}\) (see \S\refone{subsection:genuine-modules}). As an \(\GEM R\)-module spectrum, \(\GEM M^{\tC}\) determines a unique object in \(\Der(R)\), which we denote by the same symbol, via the canonical equivalence \(\Der(R) \simeq \Mod_R\).
As in \S\refone{subsection:discrete-rings}
we then make the following definition:

\begin{definition-r-three}
\label{definition:genuine-structures}%
Let $M$ be an invertible $\ZZ$-module with involution over $R$.
For every \(m\in\ZZ\cup\{\pm\infty\}\), we define a functor $\Qgen{m}{M}\colon\Dperf(R)\op\to\Spa$ as the pullback
\begin{equation*}
\begin{tikzcd}
\Qgen{m}{M}(X) \ar[r] \ar[d] & \map_R(X, \tau_{\geq m}\GEM M^{\tC}) \ar[d] \\
\map_{R\otimes R}(X\otimes X,M)^{\hC} \ar[r] & \map_R(X,\GEM M^{\tC}).
\end{tikzcd}
\end{equation*}
Here, the right hand vertical map is induced by the $m$-connective cover $\tau_{\geq m}\GEM M^{\tC}\to \GEM M^{\tC}$, and the bottom horizontal map is induced by 
the equivalence
\[
\map_{R\otimes R}(X\otimes X,M)^{\tC} \simeq \map_R(X,\GEM M^{\tC}),
\]
of Lemma~\refone{lemma:tate-M}. 
In the special cases where 
$m=\pm\infty$ we will denote these functors by 
\[
\QF^{\qdr}_{M}:=\Qgen{\infty}{M} = \map_{R\otimes R}(X\otimes X,M)_{\hC} 
\quad \text{ and } \quad 
\QF^{\s}_{M}:=\Qgen{-\infty}{M} = \map_{R\otimes R}(X\otimes X,M)^{\hC}.
\]
\end{definition-r-three}
The functors $\Qgen{m}{M}$ are indeed Poincaré structures by Examples~\refone{example:truncation}.
By construction, they all share the same underlying duality
\[
\Dual X=\map_R(X,M),
\]
where the mapping spectrum acquires a residual $R$-module structure from the $R\otimes R$-module structure of $M$.
The canonical connective cover maps $\tau_{\geq m+1}\to \tau_{\geq m}$ define an infinite sequence of natural transformations
\[
\QF^{\qdr}_M=\Qgen{\infty}{M}\to\dots\to \Qgen{(m+1)}{M}\to\Qgen{m}{M}\to\Qgen{(m-1)}{M}\to\dots\to \Qgen{-\infty}{M}=\QF^{\s}_M,
\]
and hence analogous sequences between the corresponding Grothendieck-Witt and L spectra.

\begin{remark-r-three}
\label{remark:genuineeq}%
Let $\hat{\mathrm{H}}^m(\Ct;M)=\pi_{-m}\GEM M^{\tC}$ denote the Tate cohomology of $\Ct$ with coefficients in the underlying $\ZZ[\Ct]$-module of $M$.
When $\hat{\mathrm{H}}^{-m}(\Ct;M)=0$ the map $\tau_{\geq m+1}\GEM M^{\tC}\to \tau_{\geq m}\GEM M^{\tC}$ is an equivalence. Therefore, in this case, $\Qgen{m+1}{M}\to\Qgen{m}{M}$ is an equivalence, and it induces equivalences on the corresponding Grothendieck-Witt and $\L$ spectra
\[
\GW(R;\Qgen{(m+1)}{M})\stackrel{\sim}{\longrightarrow}\GW(R; \Qgen{m}{M})\quad\text{and}\quad \L(R;\Qgen{(m+1)}{M})\stackrel{\sim}{\longrightarrow}\L(R;\Qgen{m}{M}).
\]
Moreover, $\hat{\mathrm{H}}^\ast(\Ct;M)$ is $2$-periodic, so if this happens for $m$ it also does for all $m+2k$.
In particular, if $2\in R$ is a unit all the natural transformations $\Qgen{m+1}{M}\to \Qgen{m}{M}$ are equivalences. If $2$ is not invertible however, the Grothendieck-Witt and L spectra for different $m$ are not generally equivalent, for instance this is the case for $R=\ZZ$.
\end{remark-r-three}

\begin{remark-r-three}
Among the Poincaré structures of Definition~\refthree{definition:genuine-structures}, $\Qgen{2}{M}, \Qgen{1}{M}$ and $ \Qgen{0}{M}$ are the ones which send finitely generated projective $R$-modules $P$ (regarded as chain complexes concentrated in degree zero) to abelian groups (regarded as discrete spectra). The values of $\Qgen{2}{M}$ and $\Qgen{0}{M}$ are the abelian groups of strict coinvariants and invariants, respectively,
\[
\Qgen{2}{M}(P)=\Hom_{R\otimes^\mathrm{U} R}(P\otimes^\mathrm{U} P,M)_{\Ct} 
\quad \text{ and } \quad 
\Qgen{0}{M}(P)=\Hom_{R\otimes^\mathrm{U} R}(P\otimes^\mathrm{U} P,M)^{\Ct},
\]
which are canonically isomorphic to the usual abelian groups of $M$-valued quadratic and symmetric forms on $P$, respectively, see \S\refone{subsection:discrete-rings}. Moreover, the group $\Qgen{1}{M}(P)$ is the image of the norm (or symmetrization) map $\Qgen{2}{M}(P)\to \Qgen{0}{M}(P)$. The functors $\Qgen{2}{M}, \Qgen{1}{M}$ and $ \Qgen{0}{M}$ are the \emph{non-abelian derived functors} of these functors of classical forms on modules, as shown in Proposition~\refone{proposition:classical-derived}. We call them the \emph{genuine} quadratic, \emph{genuine} even, and \emph{genuine} symmetric Poincaré structures respectively, and we denote them by
\[
\QF^{\g\qdr}_M:=\Qgen{2}{M} \quad, \quad  \QF^{\gev}_M:=\Qgen{1}{M} \quad \text{ and } \quad \QF^{\g\s}_M:=\Qgen{0}{M}.
\]
The connective covers of the associated Grothendieck-Witt spectra are the group-completions of the corresponding spaces of forms
\[
\tau_{\geq 0}\GW(R;\QF^{\g\qdr}_M)\simeq \GW^{\qdr}_{\cl}(R;M) 
\ \  , \quad 
\tau_{\geq 0}\GW(R;\QF^{\gev}_M)\simeq \GW^{\ev}_{\cl}(R;M) 
\  \ , \quad 
\tau_{\geq 0}\GW(R;\QF^{\g\s}_M)\simeq \GW^{\s}_{\cl}(R;M) 
\]
by the main result of \cite{comparison}. In particular if $M=R(\epsilon)$ is the $\ZZ$-module with involution defined from an $\epsilon$-involution on $R$ these are the classical Grothendieck-Witt spaces of $\epsilon$-quadratic, $\epsilon$-even, and $\epsilon$-symmetric forms on $R$.
\end{remark-r-three}

There is a periodicity phenomenon that relates the Poincaré structures $\Qgen{m}{M}$, that we now review. We recall that a hermitian morphism of Poincaré $\infty$-categories $(\C,\QF)\to (\Ctwo,\QFtwo)$ consists of an exact functor $f\colon \C\to \Ctwo$ and a natural transformation 
\[
\eta\colon \QF\to f^\ast \QFtwo=\QFtwo\circ f.
\]
We say that a hermitian morphism $(f,\eta)$ is a Poincaré morphism if a canonical induced map $f\Dual\to \Dual f$ is an equivalence; see \S\refone{subsection:hermitian-and-poincare-cats}. A Poincaré morphism $(f,\eta)$ is an equivalence of Poincaré \(\infty\)-categories precisely when $f$ is an equivalence of categories and $\eta$ is a natural equivalence. 
We recall the following proposition (see Proposition~\refone{proposition:general-equivalence-of-poincare-infty-categories} and Corollary~\reftwo{corollary:karoubi-fundamental}), first observed by Lurie in the cases where $m=\pm\infty$. 

\begin{proposition-r-three}
\label{proposition:periodcat}%
For every invertible $\ZZ$-module with involution $M$ over $R$ and $m\in\ZZ\cup{\{\pm\infty\}}$,
the loop functor \(\Om\colon \Dperf(R) \to \Dperf(R)\) extends to an equivalence of Poincaré \(\infty\)-categories
\[
(\Dperf(R),(\Qgen{m}{M})\qshift{2}) \stackrel{\sim}{\longrightarrow} (\Dperf(R),\Qgen{m+1}{-M}),
\]
where $\QF\qshift{k}:=\Sigma^k\QF$ denotes the $k$-fold shift of a Poincaré structure, and $-M$ is the twist by a sign of Example~\refthree{example:moduleswithinv}.
\end{proposition-r-three}

\begin{remark-r-three}
For a commutative ring $R$, we may apply Proposition~\refthree{proposition:periodcat} with $M=R$. If we set $\GW^{[n]}(R) = \tau_{\geq 0}\GW(\Dperf(R);(\Qgen{0}{R})\qshift{n})$, we obtain from \cite{comparison} the equivalences
\[
\GW^{[0]}(R) \simeq \GW^\s_\cl(R) \quad \text{ , } \quad \GW^{[2]}(R) \simeq \GW_\cl^{-\ev}(R) \quad \text{ and } \quad \GW^{[4]}(R) \simeq \GW_\cl^\qdr(R).
\]
These equivalences were also announced by Schlichting, see \cite{Schlichting-integers}*{Theorem 3.1}, where $\GW^{[2]}(R)$ is described in terms of symplectic forms.
Concretely, symplectic forms are those skew-symmetric forms $b \colon P \otimes^\mathrm{U} P \to R$ which vanish on the diagonal, i.e.\ $b(x,x) =0$ for all $x$ in $P$, compare \cite{SchlichtinghigherI}*{Definition 3.8 \& Example 3.11}. This condition is in fact equivalent to admitting a $(-1)$-quadratic refinement, so symplectic forms are precisely the $(-1)$-even forms.

To see this, we claim that the Tate cohomology
 $\hat{\mathrm{H}}^0(\Ct;\Hom_{R\otimes^\mathrm{U} R}(P\otimes^\mathrm{U} P,R(-1)))$ is isomorphic to $\map_R(P,R_2)$ where $R_2$ denotes the 2-torsion in $R$. Combining this isomorphism with the canonical map from ordinary cohomology to Tate cohomology gives a map 
\[
\mathrm{H}^0(\Ct;\Hom_{R\otimes^\mathrm{U} R}(P\otimes^\mathrm{U} P,R(-1))) \lto \map_R(P,R_2).
\]
Elements of the domain are skew-symmetric forms $b$, and they are sent under this map to the map $x \mapsto b(x,x)$. Note that this is an additive map which indeed takes values in the 2-torsion of $R$ if $b$ is skew-symmetric. Hence the obstruction to lifting a skew-symmetric form $b$ along the norm map
\[
\mathrm{H}_0(\Ct;\Hom_{R\otimes^\mathrm{U} R}(P\otimes^\mathrm{U} P,R(-1))) \lto \mathrm{H}^0(\Ct; \Hom_{R\otimes^\mathrm{U} R}(P\otimes^\mathrm{U} P,R(-1))),
\]
is given by the vanishing of $b$ on the diagonal as claimed. Of course, one can also give a direct argument for the existence of a quadratic refinement under the assumption $b(x,x) = 0$.
\end{remark-r-three}

The shifted quadratic functor relates to that of the original Poincaré $\infty$-category by means of the Bott-Genauer sequence, which we now recall. Given a Poincaré $\infty$-category $(\C,\QF)$ we can functorially form an $\infty$-category $\Met(\C,\QF)$ whose underlying $\infty$-category is the $\infty$-category of arrows  in $\C$, where the Poincaré structure is
defined by
\[
\QF_\met(f\colon L \to X) = \fib\left( \QF(f)\colon \QF(X) \to \QF(L)\right),
\]
and with underlying duality $\Dual(f\colon L \to X) = (\Dual(X/L) \to \Dual X)$, see Definition~\refone{definition:metabolic-cat}. 
The Poincaré objects of $\Met(\C,\QF)$ are given by Poincaré objects $X$ of $(\C,\QF)$ equipped with a Lagrangian $L$ (see \S\refone{subsection:metabolic-and-L}); classically, forms equipped with a Lagrangian are called metabolic forms, hence the notation $\Met(\C,\QF)$.
The Bott-Genauer sequence is the sequence of Poincaré $\infty$-categories
\[
(\C,\QF\qshift{-1})\stackrel{}{\longrightarrow}\Met(\C,\QF)\stackrel{}{\longrightarrow}(\C,\QF)
\]
where the underlying functors send an object $L$ of $\C$ to the arrow $L\to0$, and an object $f\colon L\to X$ in the arrow category to its target $X$, respectively; see Lemma~\refone{lemma:maps-with-metabolic-category}. The Bott-Genauer sequence is both a fibre and a cofibre sequence of Poincaré $\infty$-categories, that is a Poincaré-Verdier sequence in the terminology of \papertwo, see Example~\reftwo{example:metabolicfseq}. One of the main results of \papertwo is that Grothendieck-Witt theory is Verdier localising, that is that it sends Poincaré-Verdier sequences to fibre sequences of spectra. There is moreover a natural equivalence 
\[
\GW(\Met(\C,\QF))\simeq \K(\C)
\]
established in Corollary~\reftwo{corollary:GWHyp}. Under this identification the Bott-Genauer sequence induces a fibre sequence of spectra
\[
\GW(\C,\QF\qshift{-1})\stackrel{\fgt}{\longrightarrow}\K(\C)\stackrel{\hyp}{\longrightarrow}\GW(\C,\QF)
\]
where the maps are induced by the projection $\Poinc(\C,\QF\qshift{-1})\to \core \C$ that forgets the form and the map $\core \C\to \Poinc(\C,\QF)$ that sends an object to its hyperbolic form, respectively. By combining these ingredients with the periodicity of Proposition~\refthree{proposition:periodcat} we obtain the following general form of Karoubi's periodicity theorem.
Let $\U(R,\QF)$ and $\V(R,\QF)$ be the fibre of $\hyp$ and $\fgt$, respectively. The fibre sequence above provides an equivalence $\V(R;\QF)\simeq \Omega\U(R;\QF\qshift{2})$ since both are equivalent to $\Omega \GW(R;\QF\qshift{1})$. Combined with Proposition~\refthree{proposition:periodcat} we obtain the following.

\begin{theorem-r-three}[\reftwo{corollary:karoubi-fundamental}]
\label{theorem:Karoubi}%
Let $R$ be a ring and $M$ an invertible $\ZZ$-module with involution over $R$. Then there is a natural equivalence 
\[
\V(R;\Qgen{m}{M}) \simeq \Omega \U(R;\Qgen{(m+1)}{-M})
\]
for every $m\in\ZZ$, where $-M$ is the $R\otimes^\mathrm{U} R$-module $M$ with the involution $(\inv{\bullet})$ replaced by $-(\inv{\bullet})$.
\end{theorem-r-three}

\begin{remark-r-three}
\label{remark:Karoubi}%
If $2\in R$ is a unit, Theorem~\refthree{theorem:Karoubi} is due to Karoubi \cite{Karoubi-Le-theoreme-fondamental}. Since in this case the Poincaré structures $\Qgen{m}{M}$ are all equivalent, it takes the form 
\[
\V(R;\QF^{\s}_M) \simeq \Omega \U(R;\QF^{\s}_{-M}).
\]

There is another case where this theorem simplifies, but where $2$ does not need to be invertible. Let $R$ be a commutative ring which is 2-torsion free, for instance the ring of integers in a number field, and let $M=R$ with the trivial involution. In this case $\hat{\mathrm{H}}^0(\Ct;-R)=0$ and $\hat{\mathrm{H}}^{-1}(\Ct;R)=0$, and by Remark~\refthree{remark:genuineeq} we have that $\QF^{\gev}_{-R}=\QF^{\g\s}_{-R}$ and $\QF^{\g\qdr}_R=\QF^{\gev}_R$. Therefore the periodicity Theorem gives us that
\[
\V(R;\QF^{\g\s}_R) \simeq \Omega \U(R;\QF^{\g\s}_{-R})\quad ,\quad \V(R;\QF^{\g\s}_{-R}) \simeq \Omega \U(R;\QF^{\g\qdr}_R)\quad \text{and}\quad  \V(R;\QF^{\g\qdr}_R) \simeq \Omega \U(R;\QF^{\g\qdr}_{-R}).
\]
Curiously, $\V(R;\QF^{\g\qdr}_{-R}) \simeq \Omega \U(R;\Qgen{3}{R})$, and to the best of our knowledge $\Qgen{3}{R}$ cannot be expressed in terms of classical forms.
\end{remark-r-three}

Given any invariant $\F$ of Poincaré $\infty$-categories which is Verdier localising, the Bott-Genauer sequence induces a fibre sequence upon applying $\F$. If in addition $\F(\Met(\C,\QF)) = 0$ for any Poincaré $\infty$-category $(\C,\QF)$ (i.e.\ $\F$ is in addition bordism invariant in the terminology of \papertwo), one obtains a canonical equivalence
\[
\F(\C,\QF\qshift{n})\simeq \Sigma^n\F(\C,\QF)
\]
for every $n\in \ZZ$ (see \cite{Lurie-L-theory} and Proposition~\reftwo{proposition:bordism-invariant-sig}). Examples of bordism invariant functors are L-theory $\L(\C,\QF)$ and the Tate construction on K-theory $\K(\C,\QF)^{\tC}$, where $\K(\C,\QF)$ denotes the K-theory spectrum of $\C$ with the $\Ct$-action induced by the duality underlying $\QF$. In fact, for $\K(\C,\QF)^{\tC}$ this is an immediate consequence of classical additivity: $\K(\Met(\C,\QF))$ is equivalent to the induced object $\mathrm{ind}_{\{e\}}^\Ct \K(\C)$, so that its Tate construction vanishes.
We can again combine these results with the periodicity of Proposition~\refthree{proposition:periodcat} to obtain the following.

\begin{corollary-r-three}
\label{corollary:periodicity-L}%
Let $R$ be a ring and $M$ an invertible $\ZZ$-module with involution over $R$. Then there are natural equivalences
\[
\L(R;\Qgen{m}{M})\simeq \Om^{2} \L(R;\Qgen{(m+1)}{-M}) \quad\text{and}\quad \K(R;\Qgen{m}{M})^{\tC}\simeq \Om^{2} \K(R;\Qgen{(m+1)}{-M})^{\tC}.
\]
In particular, the spectra $\L(R;\QF^{\s}_M)$, $\L(R;\QF^{\qdr}_M)$ and $\K(R;M)^\tC$ are $4$-periodic, and $2$ periodic if $R$ is an $\FF_2$-algebra.
\end{corollary-r-three}

The last observation on the periodicity of the quadratic and symmetric L-spectra is of course due to Ranicki, and it has been reworked in the present language by Lurie \cite{Lurie-L-theory}.

\begin{remark-r-three}
\label{remark:passing-to-opposite-ring}%
Let $R$ be a ring. We note that the association $X \mapsto \hom_R(X,R)$ refines to an equivalence of stable $\infty$-categories
\[
\hom_R(-,R) \colon \Dperf(R)\op \stackrel{\simeq}{\lto} \Dperf(R\op).
\]
Now let $\QF$ be a Poincaré structure on $\Dperf(R)$. 
By pulling back $\QF$ along its induced duality $\Dual$, we obtain a Poincaré structure on $\Dperf(R)\op$, and further pulling back along the above equivalence a Poincaré structure on $\Dperf(R\op)$ which we denote by $\QF^\vee$. That is, $\QF^\vee$ is the composite 
\[
\QF^\vee\colon \Dperf(R\op)\op \stackrel{\simeq}{\lto} \Dperf(R) \stackrel{\Dual}{\lto} \Dperf(R)\op \stackrel{\QF}{\lto} \Spa.
\]
By construction, there is therefore an equivalence of Poincaré $\infty$-categories
\[
(\Dperf(R),\QF) \simeq (\Dperf(R\op),\QF^{^\vee}).
\]
In particular, we have $\GW(R\op;\QF^\vee) \simeq \GW(R;\QF)$ and likewise $\L(R\op;\QF^\vee) \simeq \L(R;\QF)$. For a ring with involution $R$, the invertible $\ZZ$-module $M=R$ is self dual, so that one finds $\GW(R) \simeq \GW(R\op)$ and likewise $\L(R) \simeq \L(R\op)$.
\end{remark-r-three}

\begin{notation-r-three}
\label{notation:deco}%
Let $\F$ be a functor, such as $\GW$ or $\L$, from the category of Poincaré $\infty$-categories to spectra. We introduce the following compact notation for the value of $\F$ at the perfect derived $\infty$-category of $R$ with one of the Poincaré structures $\QF_{M}^\alpha$ discussed above:
\[
\F^{\alpha}(R;M):=\F(\Dperf(R),\QF^{\alpha}_M)
\]
If $M=R(\eps)$ is the $\ZZ$-module with involution associated to an $\epsilon$-involution on $R$ as in Example~\refthree{example:moduleswithinv}, we write 
\[
\QF^{\alpha}_{\eps}:=\QF^{\alpha}_{R(\eps)}\quad\text{and}\quad \F^{\alpha}(R;\eps):=\F^{\alpha}(R;R(\epsilon))
\]
for any of the decorations $\alpha$ above. In the special cases where $\eps=\pm 1$ we will further write 
\begin{align*}
\QF^{\alpha}&:=\QF^{\alpha}_{1}=\QF^{\alpha}_{R}
\\
\QF^{\alpha}_{-}&:=\QF^{\alpha}_{-1}=\QF^{\alpha}_{R(-1)}
\\
\F^{\alpha}(R)&:=\F^{\alpha}(R;1)=\F^{\alpha}(R;R)
\\
\F^{-\alpha}(R)&:=\F^{\alpha}(R;-1)=\F^{\alpha}(R;R(-1)).
\end{align*}
The homotopy groups of any of these spectra will be denoted by adding a subscript  $\F^{\alpha}_n(R;M):=\pi_n \F^{\alpha}(R;M)$ for every $n\in\ZZ$.
\end{notation-r-three}

\section{L-theory and algebraic surgery}
\label{section:L-theory}%
This section is devoted to exploring \(\L\)-theory in the context of modules with involution. In \S\refthree{subsection:L-preliminaries}
 we recall the generators and relations description of the L-groups, and an important construction which allows to manipulate representatives in such L-groups (without changing the class in L-theory) called \emph{algebraic surgery}. 

In \S\refthree{subsection:surgery-quadratic}, we prove a surgery result for Poincaré structures which we call $m$-quadratic, for $m\in\ZZ$, and use this to represent \(\L\)-theory classes by Poincaré objects which satisfy certain connectivity bounds.
In particular this allows us to show that the \(\L\)-groups $\L_n^{\gs}(R;M)$ coincide with Ranicki's original definition of symmetric L-theory of short complexes, Theorem~\refthree{theorem:Ranicki-coincides} from the introduction.

Finally, in \S\refthree{subsection:surgery-symmetric} we prove a surgery result for Poincaré structures which we call $r$-symmetric, for $r\in \ZZ$, in case the ring under consideration is coherent of finite global dimension. We will use this to show that the genuine symmetric L-groups are isomorphic to the symmetric L-groups in sufficiently high degrees, and consequently the analogous statement for the Grothendieck-Witt groups, which are Theorem~\refthree{theorem:gs-and-s-agreement-range} and Theorem~\refthree{theorem:GW-classical-quad-sym} of the introduction.

\subsection{L-theoretic preliminaries}
\label{subsection:L-preliminaries}%

For the whole section we let $R$ be a ring, $M$ an invertible $\ZZ$-module with involution over $R$, and $\Dual=\hom_R(-,M)$ the corresponding duality on $\Dperf(R)$. We recall that $\QF^{\qdr}_M$ denotes the quadratic Poincaré structure on $\Dperf(R)$, defined as the homotopy coinvariants $\QF^{\qdr}_M(X)=\map_{R\otimes R}(X\otimes X,M)_{\hC}$, and that the symmetric Poincaré structure $\QF^{\s}_M$ is defined in an analogous way by taking homotopy invariants.

\begin{remark}
\label{remark:compatible-structures}%
If a Poincaré structure $\QF\colon\Dperf(R)^{\op}\to\Spa$ has underlying duality $\Dual$, we will say that $\QF$ is \defi{compatible with $M$}. In this case, the canonical map \(\QF^{\qdr}_M \to \QF^{\s}_M\) factors as in Construction~\refone{construction:functors-associated-to-module-with-genuine-involution} into a pair of natural transformations 
\begin{equation*}
\QF^{\qdr}_M \lto \QF \lto \QF^{\s}_M,
\end{equation*}
exhibiting $\QF^{\qdr}_M$ and $\QF^{\s}_M$ respectively as the initial and the final Poincaré structure compatible with $M$, see Corollary~\refone{corollary:hom-universal}.
\end{remark}

We recall that a spectrum $E$ is $m$-connective for some integer $m\in\ZZ$ if $\pi_k E=0$ for all $k<m$, and $m$-truncated if $\pi_k E=0$ for all $k>m$.

\begin{definition}
\label{definition:r-sym}%
For every \(r \in \ZZ\) we will say that \(\QF\) is \defi{\(r\)-symmetric} if for every finitely generated projective module $P \in \Proj(R)$ the fibre of \(\QF(P[0]) \to \QF^{\s}_M(P[0])\) is \((-r)\)-truncated. Dually, for \(m \in \ZZ\) we will say that \(\QF\) is \defi{\(m\)-quadratic} if the cofibre of \(\QF^{\qdr}_M(P[0]) \to \QF(P[0])\) is \(m\)-connective for every \(P \in \Proj(R)\). 
\end{definition}

\begin{remark}
\label{remark:connectivity}%
Note that the fibre of $\QF \to  \QF^{\s}_M$ and the cofibre of $ \QF^{\qdr}_M\to \QF$ are exact (contravariant) functors. It thus suffices to check the conditions in the definition for $m$-quadratic and $r$-symmetric Poincaré structures only in the case where $P=R$.

It also follows that the collection of \(X \in \Dperf(R)\) for which the above fibre is \((-r)\)-truncated for a given \(r \in \ZZ\) is closed under suspensions and extensions. 
In particular, if \(\QF\) is \(r\)-symmetric then the fibre of \(\QF(X) \to \QF^{\s}_M(X)\) is \((-r-k)\)-truncated for every \(k\)-connective \(X\).

Dually, for a given \(m \in \ZZ\) the collection of \(\Dual X \in \Dperf(R)\) for which the above cofibre is \(m\)-connective is closed under suspensions and extensions. In particular, if \(\QF\) is \(m\)-quadratic then the cofibre of \(\QF^{\qdr}_M(X) \to \QF(X)\) is \((m+k)\)-connective whenever \(\Dual X\) is \(k\)-connective.
\end{remark}

\begin{example}
\label{example:symm-quad-qf}%
The symmetric Poincaré structure \( \QF^{\s}_M\) is \(r\)-symmetric for every \(r\) and the quadratic Poincaré structure \( \QF^{\qdr}_M\) is \(m\)-quadratic for every \(m\).
More generally, from the fibre sequences 
\[
\tau_{\leq m-2}\Omega M^{\tC} \to  \Qgen{m}{M}(R) \to \QF^{\s}_M(R)  \quad\text{and}\quad \QF^{\qdr}_M(R) \to \Qgen{m}{M}(R) \to \tau_{\geq m}M^{\tC}
\]
we find that the Poincaré structure \(\Qgen{m}{M}\) is \(m\)-quadratic and \((2-m)\)-symmetric. In particular, \(\QF^{\gs}_M\) is \(2\)-symmetric and \(0\)-quadratic, \(\QF^{\gev}_M\) is \(1\)-symmetric and \(1\)-quadratic and \(\QF^{\gq}_M\) is \(0\)-symmetric and \(2\)-quadratic.
\end{example}

\begin{example}
\label{example:opposite-poincare-structure}%
Let $R$ be a ring and $\QF$ an $M$-compatible $r$-symmetric and $m$-quadratic Poincaré structure on $\Dperf(R)$. Then $\QF^\vee$ is an $M^\vee$-compatible $r$-symmetric and $m$-quadratic Poincaré structure on $\Dperf(R\op)$. To see that $\QF^\vee$ is $m$-quadratic, it suffices to show that the cofibre of the map
\[
\QF_{M^\vee}^{\qdr}(R\op) \lto \QF^\vee(R\op)
\]
is $m$-connective. By Remark~\refthree{remark:passing-to-opposite-ring}, this map is given by the map 
\[
(\QF^\qdr_M)(M[0]) \lto \QF(M[0])
\]
whose cofibre is $m$-connective by the assumption that $\QF$ is $m$-quadratic and that $M$ is finitely generated projective. A similar argument shows that $\QF^\vee$ is $r$-symmetric.
\end{example}

As explained earlier, one goal of this paper is to show that the genuine symmetric L-groups coincide with Ranicki's classical symmetric L-groups of \cite{RanickiATS1}, for which elements can be represented by chain complexes $X$ which are concentrated in a specific range of degrees. The following lemma shows that this can be equivalently phrased in terms of connectivity estimates for $X$ and $\Dual X$. The latter will be more convenient to work with for us.

\begin{lemma}
\label{lemma:range}%
Let \(X \in \Dperf(R)\) a perfect \(R\)-module and \(k \leq l\) integers. Then the following conditions are equivalent:
\begin{enumerate}
\item
\label{item:complex}%
\(X\) can be represented by a chain complex of the form 
\begin{equation*}
 \cdots \to 0 \to P_l \to P_{l-1} \to \cdots \to P_k \to 0 \to \cdots 
\end{equation*}
where each \(P_i\) is a finitely generated projective \(R\)-module concentrated in homological degree \(i\). 
\item
\label{item:connectivity}%
\(X\) is \(k\)-connective and \(\Dual X\) is \((-l)\)-connective.
\end{enumerate}
\end{lemma}

\begin{proof}
The implication \refthreeitem{item:complex} $\Rightarrow$ \refthreeitem{item:connectivity} is clear. For the other implication, let \(C\) be a complex of finitely generated projective \(R\)-modules of minimum length representing \(X\). We claim that \(C\) is concentrated in the range \([k,l]\). Let \(i\) be the minimal integer such that \(C_i \neq 0\). We claim that \(i \geq k\). Indeed, suppose that \(i < k\). Since \(X\) is \(k\)-connective we have that \(\mathrm{H}_i(C)=0\) and so the differential \(C_{i+1} \to C_{i}\) is a surjection of projective modules, hence a split surjection of projective modules, hence a surjection whose kernel \(N:=\ker(C_{i+1} \to C_i)\) is projective. Removing \(C_i\) and replacing \(C_{i+1}\) with \(N\) thus yields a shorter complex representing \(X\), contradicting the minimality of \(C\). We may hence conclude that \(C\) is concentrated in degrees \(\geq k\).

Let now \(\Dual C\) be the complex given by \((\Dual C)_i := \Dual(C_{-i})\). Since $M$ is finitely generated projective \(\Dual C =\hom_R(C,M) \in \Ch(R)\) represents \(\Dual X \in \Dperf(R)\), and is thus also a complex of minimal length representing \(\Dual X\). Since \(\Dual X\) is assumed to be \((-l)\)-connective, the same argument as above shows that \(\Dual C\) is concentrated in degrees \(\geq -l\). It then follows that \(C\) is concentrated in degrees \(\leq l\), and hence in the range \([k,l]\), as desired. 
\end{proof}

\begin{remark}
\label{remark:addendum}%
Lemma~\refthree{lemma:range} does not really require a duality. In its absence the statement still holds if we treat \(\Dual X=\hom_R(X,R)\) as an object of \(\Dperf(R\op)\). For later use, we also remark that our proof also shows that $X$ is $k$-connective if and only if it can be represented by a chain complex of finitely generated projective modules which are trivial below degree $k$.
\end{remark}

\begin{remark}
\label{remark:more-general}%
If $M$ is moreover free as an $R$-module,
the proof of Lemma~\refthree{lemma:range} works verbatim to show that for \(X\in\Dfree(R)\), condition~\refthreeitem{item:connectivity} above is equivalent to \(X\) being representable by a complex as in Lemma~\refthree{lemma:range} with each \(P_i\) a finitely generated \emph{stably free} \(R\)-module. More generally, if \(\Free(R) \subseteq \C \subseteq \Proj(R)\) is any intermediate full subcategory closed under the duality and under direct sums, and \(X\) can be represented by a bounded complex valued in \(\C\), then the argument in the proof below yields that condition~\refthreeitem{item:connectivity} above is equivalent to \(X\) being representable by a complex as in~\refthreeitem{item:complex}  with each \(P_i\) stably in \(\C\) (that is, such that there exist \(Q_i \in \C\) with \(P_i \oplus Q_i \in \C\)).
\end{remark}

\subsubsection*{L-theory and surgery}

The purpose of this subsection is to recall some fundamental properties of L-theory. For the construction of the L-theory spectra, we refer to \cite{Lurie-L-theory} and \S\reftwo{subsection:L+tate}. However, a key feature of the L-spectrum $\L(\C,\QF)$ is that its homotopy groups have a very simple presentation: They are given by cobordism groups of Poincaré objects. Let us explain what this means precisely, as we rely on this construction throughout the section. We recall that the space $\Poinc(\C,\QF)$ is the space of Poincaré objects, that is of pairs $(X,\qone)$ of an object $X$ in $\C$ and a point $\qone\in \Omega^\infty\QF(X)$ such that a canonically associated map 
\[
\qone_\sharp \colon X \lto \Dual X
\]
is an equivalence. Likewise, there is the space $\Poinc^\partial(\C,\QF)$ of Poincaré pairs, that is of triples $(f\colon L \to X,\qone, \eta)$ with $q \in \Omega^\infty\QF(X)$ and $\eta$ a nullhomotopy of $f^*(q)$, such that the canonically associated map 
\[
\eta_\sharp\colon X/L \lto \Dual L
\]
induced on the quotient $X/L$ by $\eta$
is an equivalence. In this case we say that $L$ is a Lagrangian in $X$ (or also that $L$ is a nullcobordism of $X$). It turns out that $\Poinc^\partial(\C,\QF) = \Poinc(\Met(\C,\QF))$
where $\Met(\C,\QF)$ is the metabolic category associated to $\QF$ as in \S\refone{subsection:metabolic-and-L}.
We find that forgetting the Lagrangian provides a map $\Poinc^\partial(\C,\QF) \to \Poinc(\C,\QF)$, which is induced from the Poincaré functor $\Met(\C,\QF) \to (\C,\QF)$ sending $L \to X$ to $X$.

\begin{definition}
We say that Poincaré objects $(X,\qone)$ and $(X',\qone')$ are cobordant
if $(X\oplus X',\qone\oplus (-\qone'))$ admits a Lagrangian, i.e.\ is nullcobordant. We define the $n$'th L-group $\L_n(\C,\QF)$ as the group of cobordism classes of Poincaré objects $(X,\qone)$ for the Poincaré structure $\QF\qshift{-n}:=\Omega^n\QF$.
\end{definition}

\begin{remark}
\label{remark:L-groups}%
We remark that the cobordism relation is a congruence relation with respect to $\oplus$, and that the diagonal $X \to X\oplus X$ is a canonical Lagrangian for $(X \oplus X, \qone\oplus (-\qone))$, so that $\L_n(\C,\QF)$ is indeed an abelian group. 
In particular, the L-groups fit into an exact sequence of monoids
\[
\pi_0(\Poinc^\partial(\C,\QF{\qshift{-n}})) \stackrel{\partial}{\lto} \pi_0(\Poinc(\C,\QF{\qshift{-n}})) \lto \L_n(\C,\QF) \lto 0.
\]
We also note that the above definition of \(\L\)-groups is equivalent to the one of Definition~\refone{definition:L-groups}, where \(\L_n(\C,\QF)\) was defined as the cokernel of the map \(\partial\) in the category of ordinary commutative monoids. Indeed, by exactness of the above sequence, we obtain a map of monoids $\coker(\partial)\to \L_n(\C,\QF)$ which is surjective and has trivial kernel. Since $\coker(\partial)$ is in fact a group by Lemma~\refone{lemma:L-is-group}, this map is an isomorphism.
\end{remark}

\begin{notation} In the case of the category $\C=\Dperf(R)$, we will denote the L-groups and L-spectra respectively by
\[
\L_n(R;\QF):=\L_n(\Dperf(R),\QF)\quad\text{and}\quad \L(R;\QF):=\L(\Dperf(R),\QF).
\]
When $\QF=\QF^\alpha$ is one of the genuine functors associated to an invertible $\ZZ$-module with involution $M$ analysed in the previous section, we use the notation $\L^\alpha(R;M)$ established in Notation~\refthree{notation:deco} for  the corresponding L-groups.
\end{notation}

\begin{remark}
It is immediate from the definition that $\L^{\qdr}(R;M)$ and $\L^{\s}(R;M)$ are respectively the usual quadratic and symmetric $\L$-theory spectra of $R$ of \cite{Lurie-L-theory} which also agree with the $\L$-spectra of Ranicki \cite{Ranickiblue}. The other variants are, however, more mysterious, and their study is the focus of this section.
\end{remark}

We note that given a Lagrangian for $(X,\qone)$, i.e.\ a Poincaré object of the metabolic category, through the eyes of L-theory, we may replace $(X,\qone)$ by $0$. Such a procedure in fact works more generally if we start with only a hermitian object for the metabolic category and is the content of algebraic surgery.
We recall that the hermitian objects of the metabolic category consist of triples $(f \colon L \to X, \qone,\eta)$ such that $\qone\in \Omega^\infty\QF(X)$ and $\eta$ is a nullhomotopy of $f^*(\qone)$. In many cases of interest, the object $(X,q)$ is Poincaré, and in this situation
we will refer to $L$, or more precisely to $(f,\eta)$, as a \emph{surgery datum} on $(X,q)$.
The non-degeneracy condition for this triple to be a Poincaré object for the metabolic category is that the map $ \eta_\sharp\colon X/L \lto \Dual L $ defined above is an equivalence, i.e.\ if its fibre $X'$ is $0$. In general, $X'$ need not vanish, but nevertheless acquires a canonical Poincaré form $\qone'$ induced from $(f,q,\eta)$.
In fact, we have the following result; see \S\reftwo{subsection:algebraic-surgery} for a general discussion of algebraic surgery.

\begin{proposition}
\label{proposition:naive-surgery}%
Let $(X,\qone)$ be a Poincaré object for $\QF$ with surgery datum $(f\colon L\to X,\eta)$. Then the object $X'$ carries a canonical Poincaré form $\qone'$ such that $(X,\qone)$ and $(X',\qone')$ are cobordant. 
\end{proposition}

\begin{remark}
\label{remark:surgerydiag}%
The underlying object of $X'$ and the of the cobordism $\chi(f)$ between $X$ and $X'$ are summarised in the following surgery diagram consisting of horizontal and vertical fibre sequences.
\begin{equation}
\label{equation:surgery-diagram}%
\begin{tikzcd}
L \ar[r,equal] \ar[d] & L \ar[r] \ar[d] & 0 \ar[d] \\
\chi(f) \ar[r] \ar[d] & X \ar[r,"{\Dual f\circ \qone_\sharp}"] \ar[d] & \Dual L \ar[d,equal] \\
X' \ar[r] & X/L \ar[r,"\eta_\sharp"] & \Dual L
\end{tikzcd}
\end{equation}
\end{remark}

We can use this to perform the following construction, which we will refer to as \emph{Lagrangian surgery}.
\begin{construction}
\label{construction:lagrangian-surgery}%
Let $(L \to X,\qone,\eta)$ be a Lagrangian for a Poincaré object $(X,\qone)$. Equivalently, we may view $(L\to X,\qone,\eta)$ as a Poincaré object of the metabolic category $\Met(\C,\QF)$.
Now, given a surgery datum for this Poincaré object, i.e.\ a commutative diagram
\[
\begin{tikzcd}
Z \ar[r] \ar[d] & W \ar[d] \\
L \ar[r] & X
\end{tikzcd}
\]
and a null-homotopy of $\Phi^*(\qone,\eta)$ in $\QF_\met(Z \to W)$, we may thus perform surgery by Proposition~\refthree{proposition:naive-surgery} to obtain a new Poincaré object $(L'\to X',\qone',\eta')$ of $\Met(\C,\QF)$. 
We observe that the map $W \to X$ is canonically a surgery datum on $(X,\qone)$ and that $(X',\qone')$ is the result of surgery with this surgery datum. In particular, if $W = 0$, then $(X',\qone')$ is canonically equivalent to $(X,\qone)$. Moreover, by diagram~\eqrefthree{equation:surgery-diagram} the new Lagrangian $L'$ sits inside a fibre sequence
\[
L' \lto L/Z \lto \Omega \Dual Z.
\]
We will refer to such surgery data as \emph{Lagrangian surgery data} and refer to the surgery as a \emph{Lagrangian surgery}. We will then also say that \(L\) is cobordant to \(L'\) relative to \(X\). For future reference, we notice that the underlying map of a Lagrangian surgery datum is equivalently described by a map $Z \to N = \fib(L \to X)$. If we denote by $N'$ the fibre of the map $L' \to X$, then we obtain likewise a fibre sequence $N' \to N/Z \to \Omega \Dual Z$.
\end{construction}

\subsection{Surgery for $m$-quadratic structures}
\label{subsection:surgery-quadratic}%

In this section we will show how to apply algebraic surgery to Poincaré structure which are sufficiently quadratic and use this to show that the genuine symmetric L-groups coincide with Ranicki's symmetric L-groups of short complexes; see Theorem~\refthree{theorem:main-theorem-L-theory}. We also show that  in sufficiently small degrees, the L-groups of an $m$-quadratic functor coincide with the quadratic L-groups; see Corollary~\refthree{corollary:genuine-is-quadratic}.
The surgery arguments we present below are designed to replace (shifted) Poincaré objects and Lagrangians by cobordant counterparts which are suitably connective. The following definition summarises the kind of connectivity we seek:

\begin{definition}
\label{definition:connective-L}%
Let $M$ be an invertible $\ZZ$-module with involution over $R$, \(\QF\) a Poincaré structure on \(\Dperf(R)\) compatible with $M$. Let $n,a,b \in \ZZ$ be such that $a,b \geq -1$, $b \geq a-1$, 
and $(n+a)$ is even. 
\begin{enumerate}
\item
\label{item:one}%
We denote by \(\Poinc_{n}^{a}(R,\QF) \subseteq \Poinc(\Dperf(R),\QF{\qshift{-n}})\) the subspace spanned by those Poincaré objects \((X,q)\) such that \(X\) is \((\frac{-n-a}{2})\)-connective. 
\item
\label{item:two}%
We denote by \(\pM_{n}^{a,b}(R,\QF) \subseteq \Poinc^{\partial}(\Dperf(R),\QF{\qshift{-n}})\) the subspace spanned by those Poincaré pairs \((L \to X,q,\eta)\) such that \(X\) is \((\frac{-n-a}{2})\)-connective, \(L\) is \(\lceil \frac{-n-1-b}{2} \rceil\)-connective and \(N:=\fib(L \to X) \simeq \Om^{n+1}\Dual L\) is \(\lfloor \frac{-n-1-b}{2} \rfloor\)-connective. We refer to such an $L$ as an \emph{allowed Lagrangian} for $(X,\qone)$.
\end{enumerate}
Finally, we define 
\[
\L^{a,b}_n(R;\QF) = \coker\left( \pi_0\pM_{n}^{a,b}(R,\QF) \to \pi_0\Poinc_{n}^{a}(R,\QF)\right)
\]
as the cokernel in the category of monoids of the map that forgets the Lagrangian.
\end{definition}

\begin{remark}
\label{remark:intervalsLab}%
The definitions are made such that $\L_n^{a,b}(R;\QF)$ is the monoid of $n$-dimensional Poincaré objects of width $a$ modulo those which admit a Lagrangian of width $b$. 

More precisely,
the connectivity assumption of case~\refthreeitem{item:one} guarantees that $X$ can be represented by a complex 
concentrated in the range  
$[\frac{-n-a}{2},\frac{-n+a}{2}]$, and consequently having width $a$. This uses Lemma~\refthree{lemma:range} and that $\Dual X \simeq X[n]$ is \((\frac{n-a}{2})\)-connective. In particular, as an example we have $\Poinc_{n}^{-1}(R,\QF) \simeq \ast$. 

In case~\refthreeitem{item:two}, $X$ can again be represented by a complex concentrated in degrees $[\frac{-n-a}{2},\frac{-n+a}{2}]$, and the assumptions on $L$ and $N$ can be divided into the following two cases.
If $b+n$ is odd, the connectivity assumptions on $X$ and $N$ imply the one on $L$, and both $L$ and $N$ can be represented by complexes in degrees  $[ \frac{-n-1-b}{2} , \frac{-n-1+b}{2} ]$. Here for the upper bound on $L$ we use the connectivity of $N$ and that $\Dual L = N[n+1]$.
If however $b+n$ is even, then necessarily we have \(b \geq a\), and it is the connectivity assumptions on $X$ and $L$ which imply the one on $N$. Then $L$ can be represented by a complex in degrees  $[ \frac{-n-b}{2} , \frac{-n+b}{2} ]$ whereas $N$ by one in degrees $[ \frac{-n-b}{2}-1 , \frac{-n+b}{2}-1 ]$.

In the current section on quadratic surgery we exclusively use the case where $b+n$ is odd. However, in the next section on symmetric surgery we will have to consider both parities of $b+n$.
\end{remark}

\begin{remark}
\label{remark:diagonaladmissible}%
If $b\geq a$, the diagonal inside $(X,\qone)\oplus (X,-\qone)$ is an allowed Lagrangian, so that $\L_n^{a,b}(R;\QF)$ is in fact a group. In Proposition~\refthree{proposition:surgery-quadratic} we will show, under additional hypotheses on the Poincaré structure $\QF$, that $\L_n^{a,a-1}(R;\QF)$ is also a group.
\end{remark}

\begin{remark}
\label{remark:Witt-group}%
Let $\QF\colon \Dperf(R)\op \to \Spa$ be a Poincaré structure whose associated duality preserves $\Proj(R) \subseteq \Dperf(R)$. A Poincaré object $(P,q) \in \pi_0\Poinc_0^0(R,\QF)$, i.e.\ one where $P \in \Proj(R) \subseteq \Dperf(R)$, is said to be strictly metabolic if there exists a submodule $L \subseteq P$ such that $L$ is projective, $q_{|L}$ vanishes in $\pi_0\QF(L)$ and the sequence 
\begin{equation}
\label{equation:strict-metabolic}%
0 \lto L \lto P \lto \Dual L \lto 0
\end{equation}
is exact. One defines the Witt group $\W(\Proj(R);\QF)$ as the quotient the monoid $\pi_0\Poinc_0^0(R,\QF)$ by the submonoid generated by the strictly metabolic objects. We note that if $\QF$ restricts to a functor $\QF \colon \Proj(R)\op \to \Ab \subseteq \Spa$, this is the classical Witt group as defined for instance in \cite{knebusch}. 

We now argue that there is a canonical isomorphism $\W(\Proj(R);\QF) \cong \L_0^{0,0}(R;\QF)$. Indeed, it suffices to verify that a Poincaré object $(P,q) \in \pi_0\Poinc_0^0(R,\QF)$ is strictly metabolic if and only if if it admits an allowed Lagrangian. The ``if'' direction is clear, and to see the ``only if'' part observe that if $(P,q)$ is strictly metabolic then
any null homotopy $\eta$ of $q_{|L}$ makes the pair $(L \to X, \eta)$ an allowed Lagrangian, because the question whether or not the square 
\[
\begin{tikzcd}
	L \ar[r] \ar[d] & P \ar[d] \\
	0 \ar[r] & \Dual L
\end{tikzcd}
\]
is cartesian is equivalent to the exactness of the sequence \eqrefthree{equation:strict-metabolic} and hence independent of the chosen null homotopy $\eta$.
\end{remark}

\begin{remark}
\label{remark:surgery-1-quadratic}%
Let us briefly digress about strictly metabolic objects for $1$-quadratic Poincaré structures. So suppose that $P$ is a finitely generated projective $R$-module, and that $(P,q)$ is a strictly metabolic Poincaré object with respect to a 1-quadratic Poincaré structure.
Let $L \to P$ be a strict Lagrangian, so that the fibre of the map $L \to P$ is equivalent to $(\Dual L)[-1]$ and that $P \cong L \oplus \Dual L$. By the algebraic Thom isomorphism \refone{corollary:algebraic-thom-iso}, the space of Poincaré structures on the object $L \to P$ of the metabolic category is equivalently described by the space of shifted forms $\Omega \QF(\Dual L[-1])$, which is connected as $\QF$ is 1-quadratic. It follows that $(P,q)$ is equivalent to $\hyp(L)$, the hyperbolic form on $L$. This recovers the well-known classical fact that a strictly metabolic quadratic form on a finitely generated projective module is hyperbolic. 
\end{remark}

For the remainder of the section we fix an invertible $\ZZ$-module with involution $M$ over $R$, and we consider only Poincaré structures $\QF$ on $\Dperf(R)$ which are compatible with $M$, and we denote the underlying duality by $\Dual=\hom_R(-,M)$.
To put the assumptions of the next result into context, recall that the Poincaré structure $\Qgen{m}{M}$ is $m$-quadratic.

\begin{proposition}[Surgery for \(m\)-quadratic Poincaré structures]
\label{proposition:surgery-quadratic}%
Let \(\QF\) be an $m$-quadratic Poincaré structure on \(\Dperf(R)\). 
Fix an $n\in \ZZ$ and let \(a,b \geq 0\) be two non-negative integers with $b\geq a-1$, and such that $n+a$ and $n+1+b$ are even.
\begin{enumerate}
\item
\label{item:qsurjobject}%
If $a \geq n-2m$ then every Poincaré object in \((\Dperf(R),\QF\qshift{-n})\) is cobordant to one which is \(\big(\tfrac{-n-a}{2}\big)\)-connective.
\item
\label{item:qsurjLagrangian}%
If $b \geq n-2m+1$ then every Lagrangian \(L \to X\) of a \(\big(\tfrac{-n-a}{2}\big)\)-connective Poincaré object \((X,q) \in \Poinc(\Dperf(R),\QF\qshift{-n})\) is cobordant relative to \(X\) to a Lagrangian \(L'\to X\) such that both \(L'\) and \(\fib[L'\to X]\) are \(\big(\tfrac{-n-1-b}{2}\big)\)-connective. 
\end{enumerate}
In particular, if both inequalities above hold the monoid $\L^{a,b}_n(R;\QF)$ is a group and the canonical map $\L^{a,b}_n(R;\QF) \to \L_n(R;\QF)$ is an isomorphism.
\end{proposition}

\begin{remark}
\label{remark:recall-intervalsLab}%
Recalling Remark~\refthree{remark:intervalsLab} we note that the Poincaré objects appearing in~\refthreeitem{item:qsurjobject} above are concentrated in degrees \([\frac{-n-a}{2},\frac{-n+a}{2}]\) and the Lagrangians in~\refthreeitem{item:qsurjLagrangian} are concentrated in degrees \( \frac{-n-1-b}{2} , \frac{-n-1+b}{2} ]\).
\end{remark}

Before diving into the proof of this proposition, let us give some immediate consequences.
\begin{example}
The quadratic Poincaré structure \(\QF^{\qdr}_M\) is \(m\)-quadratic for every \(m\). Given \(n = 2k \in \ZZ\), we may apply Proposition~\refthree{proposition:surgery-quadratic} to \(\QF^{\qdr}_M\) with \((a,b) = (0,1)\) and deduce that every class in \(\L^{\qdr}_n(R;M)\) can be represented by a Poincaré object which is concentrated in degree \(-k\), and that such a Poincaré object represents zero in \(\L^{\qdr}_n(R;M)\) if and only if it admits a Lagrangian which is concentrated in degrees \([-k-1,-k]\). On the other hand, if \(n=2k+1\) is odd we may apply Proposition~\refthree{proposition:surgery-quadratic} to \(\QF^{\qdr}_M\) with \((a,b) = (1,0)\) and get that every class in \(\L^{\qdr}_n(R;M)\) can be represented by a Poincaré object which is concentrated in degree \([-k-1,-k]\), and that such a Poincaré object represents zero in \(\L^{\qdr}_n(R;M)\) if and only if it admits a Lagrangian which is concentrated in degree \(-k-1\). This is often referred to in the literature as \emph{surgery below the middle dimension}. In fact, Proposition~\refthree{proposition:surgery-quadratic} gives this statement for any \(m\)-quadratic Poincaré structure, as long as we take \(n \leq 2m\).
\end{example}

\begin{corollary}
\label{corollary:surgery-genuine-symmetric}%
For any $n\geq 0$,
every class in \(\L^{\gs}_n(R;M)\) can be represented by a Poincaré object which is concentrated in degrees \([-n,0]\), and such a Poincaré object represents zero in \(\L^{\gs}_n(R;M)\) if and only if it admits a Lagrangian which is concentrated in degrees \([-n-1,0]\).
In particular, the canonical map $\L^{n,n+1}_n(R;\QF^{\gs}_M) \to \L_n(R;\QF^{\gs}_M)=\L^{\gs}_n(R;M)$ is an isomorphism for all $n\geq 0$.
\end{corollary}
\begin{proof}
The genuine symmetric Poincaré structure \(\QF^{\gs}_M\) is \(0\)-quadratic. Given \(n \geq 0 \in \ZZ\) we may thus apply Proposition~\refthree{proposition:surgery-quadratic} to \(\QF^{\gs}_M\) with \((a,b) = (n,n+1)\) so that \(\L(R;\QF^{\gs}_M) \cong \L^{n,n+1}_n(R;\QF^{\gs}_M)\). 
\end{proof}

The proof of Proposition~\refthree{proposition:surgery-quadratic} will require the following connectivity estimate:
\begin{lemma}
\label{lemma:connectivity-2}%
Suppose that \(\QF\) is an \(m\)-quadratic Poincaré structure on \(\Dperf(R)\). Then for every projective module \(P \in \Proj(R)\) and every \(k \in \ZZ\) the spectrum \(\QF(P[k])\) is \(\min(-2k,m-k)\)-connective.
\end{lemma}
\begin{proof}
The cofibre of the map $\QF^\qdr_M(P[k]) \to \QF(P[k])$
is $(m-k)$-connective by the assumption that $\QF$ is $m$-quadratic, see Remark~\refthree{remark:connectivity}. Furthermore, $\QF^\qdr_M(P[k]) = (\map_{R\otimes R}(P \otimes P,M)[-2k])_{\hC}$
and is thus $(-2k)$-connective as $P$ is projective and homotopy orbits preserve connectivity.
\end{proof}
 
\begin{proof}[Proof of Proposition~\refthree{proposition:surgery-quadratic}]
To prove~\refthreeitem{item:qsurjobject}, suppose that \((X,q)\) is a Poincaré object in \((\Dperf(R),\QF{\qshift{-n}})\).
If \(X\) itself is \((\frac{-n-a}{2})\)-connective, we are done. Otherwise, since \(X\) is perfect, there exists some \(k< \frac{-n-a}{2}\) such that \(X\) is \(k\)-connective. 
By Lemma~\refthree{lemma:range}
the object \(X\) can be represented by a chain complex of projectives concentrated in degrees \(\geq k\) and so there exists a projective module \(P\) and a map \(f\colon P[k] \to X\) which is surjective on $\mathrm{H}_k$.
By Lemma~\refthree{lemma:connectivity-2}, the spectrum \(\QF(P[k])\) is \(\min(-2k,m-k)\)-connective and since \(k < \frac{-n-a}{2}\) we have that
\[
\min(-2k,m-k) > \min\left(n+a,\frac{2m+n+a}{2}\right) \geq n
\]
by the inequalities in our assumptions.
It then follows that \(\Om^{\infty+n}\QF(P[k])\) is connected and hence \(q\) restricted to \(P[k]\) is null-homotopic, so that any nullhomotopy $\eta$, makes $(P[k] \to X,\qone,\eta)$ a hermitian form for the metabolic category. We may therefore apply Proposition~\refthree{proposition:naive-surgery} and perform surgery along \(f\colon P[k] \to X\) to obtain a cobordant Poincaré object \(X'\), given by the fibre of the induced map $X/P[k] \to \Dual(P)[-k-n]$. Since $-2k-1 > n+a \geq n$ (here we use that $n+a$ is even) we have that $-k-n > k+1$ and so 
\[
\mathrm{H}_{k'}(X')=\mathrm{H}_{k'}(X) = 0 \quad\text{for } k'<k \quad\text{and}\quad \mathrm{H}_{k}(X') \cong \coker[P \to \mathrm{H}_{k}(X)] = 0 ,
\]
which means that \(X'\) is \((k+1)\)-connective. Proceeding inductively we may thus obtain a Poincaré object \((X'',q'')\) which is cobordant to \((X,q)\) and which is \((\frac{-n-a}{2})\)-connective. 
Let us now prove Claim~\refthreeitem{item:qsurjLagrangian}. Let $(X,q)\in\Poinc(\Dperf(R),\QF{\qshift{-n}})$, and suppose that $X$ is $(\frac{-n-a}{2})$-connective and that it admits a Lagrangian $(L\to X,\eta)$.
Let $N$ be the fibre of the map $L\to X$. If $N$ is $(\frac{-n-1-b}{2})$-connective, then, since $b\geq a-1$ so is $L$ and we are done.  
Otherwise, let $l<\frac{-n-1-b}{2}$ be such that $N$ is $l$-connective.
We can then find a projective module $P$ and a map $P[l] \to N$ which is surjective on $\mathrm{H}_l$. We may view the map $P[l] \to N$ equivalently as a map $(P[l] \to 0) \to (L \to X)$ in the metabolic category. We claim that this map extends to a Lagrangian surgery datum in the sense of Construction~\refthree{construction:lagrangian-surgery}, 
for which it suffices to see that $\QF_\met(P[l]\to 0) \simeq \Omega \QF(P[l])$ is $(n+1)$-connective.
By Lemma~\refthree{lemma:connectivity-2}, the spectrum $\QF(P[l])$ is then $\min(-2l,m-l)$-connective and since $l < \tfrac{-n-1-b}{2}$ we have that
\[
\min(-2l,m-l) > \min\left(n+1+b,\frac{2m+n+1+b}{2}\right) \geq n+1
\]
by the inequalities in our assumptions. We may therefore perform Lagrangian surgery along \(P[l] \to L\), see Construction~\refthree{construction:lagrangian-surgery}, to obtain a new Lagrangian \(L' \to X\) such that the fibre $N'$ of the map $L' \to X$ fits in a fibre sequence
\[
N' \lto N/P[l] \lto \Dual(P)[-l-n-1].
\]
Since $2l < -n-1-b$ and $n+1+b$ is even, we have that \(-2l-1 > n+1+b \geq n+1\). Thus \(-l-n-1 > l+1\), and so 
\[
\mathrm{H}_{l'}(N')=\mathrm{H}_{l'}(N) = 0 \quad\text{for } l'<l \quad\text{and}\quad \mathrm{H}_{l}(N') \cong \coker[P \to \mathrm{H}_{l}(N)] = 0,
\]
which means that \(N'\) is \((l+1)\)-connective.
Proceeding inductively we may thus obtain a Lagrangian $L'' \to X$ for which $N''$, and thus $L''$, is \((\frac{-n-1-b}{2})\)-connective.

To see the final claim, we now recall that the diagonal defines a Lagrangian for $(X\oplus X,q\oplus (-q))$, and so it follows from~\refthreeitem{item:qsurjLagrangian} that the commutative monoid $\L_n^{a,b}(R;\QF)$ is a group. The surjectivity and injectivity of the homomorphism $\L^{a,b}_n(R;\QF) \to \L_n(R;\QF)$ then follow from~\refthreeitem{item:qsurjobject} and~\refthreeitem{item:qsurjLagrangian}, respectively.
\end{proof}

\begin{remark}
\label{remark:cofibre-is-fibre}%
Under the assumptions of Proposition~\refthree{proposition:surgery-quadratic}, the 
sequence 
\[
\pi_0\pM_{n}^{a,b}(R,\QF) \lto \pi_0\Poinc_{n}^{a}(R,\QF) \lto \L^{a,b}_n(R;\QF)
\]
is exact in the middle, just as in the case of ordinary L-groups; see Remark~\refthree{remark:L-groups}.
Indeed if a Poincaré object of $\Poinc_{n}^{a}(R,\QF)$ represents zero in $\L^{a,b}_n(R;\QF)$, and therefore in $\L_n(R;\QF)$, it admits a Lagrangian, and by Part~\refthreeitem{item:qsurjLagrangian} also a Lagrangian with the connectivity assumptions required to define an element of $\pi_0\pM_{n}^{a,b}(R,\QF)$.
\end{remark}

\begin{corollary}
\label{corollary:genuine-is-quadratic}%
If \(\QF\) is \(m\)-quadratic (e.g., \(\QF=\Qgen{m}{M}\)) then the natural map
\[
\L^{\qdr}_n(R;M) =\L_n(R;\QF^{\qdr}_M) \lto \L_n(R;\QF)
\]
is an isomorphism for \(n \leq 2m-3\) and surjective for \(n = 2m-2\). 
\end{corollary}
\begin{proof}
Fix an \(n \leq 2m-2\) and let \(a \in \{0,1\}\) be such that \(n+a\) is even. Then \(a \geq n-2m\) and so 
by Proposition~\refthree{proposition:surgery-quadratic}~\refthreeitem{item:qsurjobject} every class in either \(\L^{\qdr}_n(R;M)\) or \(\L_n(R;\QF)\) can be represented by a Poincaré object which is \(\big(\tfrac{-n-a}{2}\big)\)-connective. 
To prove that the map \(\L_n(R;\QF^{\qdr}_M) \lto \L_n(R;\QF)\) is surjective in this range it will then suffice to show that the monoid map
\[
\pi_0\Poinc_{n}^{a}(R,\QF^{\qdr}_M) \lto \pi_0\Poinc_{n}^{a}(R,\QF)
\]
is surjective. 
Let $X$ be $(\frac{-n-a}{2})$-connective and equipped with a Poincaré form $\qone$ for $\QF\qshift{-n}$. As $\QF$ is $m$-
quadratic and $\Dual X \simeq \Sig^n X$ is $(\frac{n-a}{2})$-connective 
the cofibre of the map $\QF^{\qdr}_M(X) \to \QF(X)$ is $(m+\frac{n-a}{2})$-connective; see Remark~\refthree{remark:connectivity}. Since $m+(\frac{n-a}{2}) = \frac{2m+n-a}{2} \geq n+1$ (where we note that \(a=0\) when \(n=2m-2\)) it follows that the map \(\QF^{\qdr}_M(X) \to \QF(X)\) is surjective on \(\pi_n\), and so the surjectivity part of the statement is established.

To prove injectivity, let us now assume that \(n \leq 2m-3\) and let \(b \in \{0,1\}\) be such that \(n+1+b\) is even. Then \(b \geq n-2m+1\) and so by Proposition~\refthree{proposition:surgery-quadratic}~\refthreeitem{item:qsurjLagrangian} every \(\big(\tfrac{-n-a}{2}\big)\)-connective Poincaré object \((X,q)\) in \((\Dperf(R),\QF\qshift{-n})\) which admits a Lagrangian, also admits a Lagrangian \(L \to X\) such that \(L\) and \(\fib[L \to X]\) are \(\big(\tfrac{-n-1-b}{2}\big)\)-connective. 
Then $\Dual L \simeq \Sigma^{n+1}N$ is $\big(\frac{n+1-b}{2}\big)$-connective and so by Remark~\refthree{remark:connectivity} and the assumption that $\QF$ is $m$-quadratic we deduce that the cofibre of the map $\QF^{\qdr}_M(L) \to \QF(L)$ is $\big(\frac{2m+n+1-b}{2}\big)$-connective. Since $n \leq 2m-3$, we have $\frac{2m+n+1-b}{2} \geq n+2$ (where we note that \(b=0\) when \(n=2m-3\)) 
and so the map
\[
\QF^{\qdr}_M(L) \lto \QF(L)
\]
is bijective on \(\pi_n\) and surjective on \(\pi_{n+1}\). We conclude that \(L\) can be refined to a Lagrangian of \((X,q)\) with respect to \(\QF^{\qdr}_M\) and so the map \(\L_n(R;\QF^{\qdr}_M) \lto \L_n(R;\QF)
\) is also injective (hence bijective) when \(n \leq 2n-3\), as desired.
\end{proof}

\begin{remark} The range in which the map of Corollary~\refthree{corollary:genuine-is-quadratic} is an isomorphism is essentially optimal. For example for $R=\ZZ$ and $m=0$, the map
\[
\ZZ/2 \cong \L^{\qdr}_{-2}(\ZZ) \lto \L_{-2}^{\gs}(\ZZ) = 0
\]
is not an isomorphism, see Example~\refthree{example:L-of-integers} for the calculation of these groups.
\end{remark}

\begin{proposition}
\label{proposition:00-01}%
Let $\QF$ a Poincaré structure on $\Dperf(R)$ and $n\in \ZZ$. Then for every \(a \geq 0\) such that \(a+n\) is even the canonical map 
\[
\L_{n}^{a,a}(R;\QF) \lto \L_{n}^{a,a+1}(R;\QF)
\]
is an isomorphism.
\end{proposition}
\begin{proof} 
The map in question is clearly surjective, as both groups are generated by the same Poincaré objects. To show injectivity, let us first consider $(L \to X) \in \pi_0\M_n^{a,a+1}(R,\QF)$. We will show momentarily that then $X$ represents the trivial element of $\L^{a,a}_n(R;\QF)$, and argue first how this statement implies the injectivity of the map in question. So assume $X$ is an element in the kernel of the map $\L_n^{a,a}(R;\QF) \to \L_n^{a,a+1}(R;\QF)$. This means that there exists an element $(L' \to X') \in \pi_0\M_n^{a,a+1}(R,\QF)$ such that $X \oplus X'$ is the boundary of further element $(L\to X\oplus X') \in \pi_0\M_n^{a,a+1}(R,\QF)$. Then $X'$ as well as $X \oplus X'$ represent the trivial element of $\L_n^{a,a}(R;\QF)$, and hence so does $X$.

To prove the remaining claim, consider again $(L \to X) \in \pi_0\M_n^{a,a+1}(R,\QF)$. Let us write $k= \tfrac{-n-a}{2}$. Then, by definition, $X$ is $k$-connective, and $L$ and $N= \fib(L\to X)$ are $k-1$-connective. By Remark~\refthree{remark:intervalsLab}, $L$ can be represented by a chain complex concentrated in degrees $[k-1,k+a]$. 
There is consequently a fibre sequence
\[
L'' \lto L \lto L'
\]
such that $L''$ is a projective module $P$ concentrated in degree $k-1$ and $L'$ is $k$-connective. Let us consider the diagram
\[
\begin{tikzcd}
	L'' \ar[r, equal] \ar[d] & L'' \ar[d] \\
	L \ar[r] & X
\end{tikzcd}
\]
as a morphism (drawn vertically) in $\Met(\Dperf(R);\QF)$ and recall that $L \to X$ is canonically a Poincaré object in the metabolic category. Since $\QF_\met(\id_{L''})$ is contractible, we find that the above diagram canonically refines to a surgery datum on $L \to X$. Tracing through the definition of algebraic surgery, the surgery output is then a Lagrangian $L' \to X'$, with \(X'\) still \(k\)-connected (since \(\Om^n\Dual L''\) is concentrated in degree \(-k+1-n \geq k+1\)),
and so $(L'\to X') \in \pi_0\M_n^{a,a}(R,\QF)$.
We now analyse the Poincaré form $q'$ on $X'$. By construction, $(X',q')$ is the output of surgery along the surgery datum $L'' \to X$.
We note that the composite $L'' \to L \to X$ is null homotopic since $L'' = P[k-1]$ and $X$ is $k$-connective.
We deduce that the restriction of $q$ along the map $L'' \to X$ is null homotopic in two ways, and therefore determines a loop in $\Om^{n}\QF(L'')$. The surgery datum \(L'' \to X\) is then equivalent to the direct sum of the trivial surgery datum \(0 \to X\) and the surgery data \(L'' \to 0\) determined by the above loop.
Consequently, $(X',q')$ is the orthogonal sum of $(X,q)$ with the output $Z$ of surgery on $L'' \to 0$. Being a summand of $X'$ we find that $Z$ is also $k$-connective, and moreover by construction $\Omega^{n+1} \Dual L''$ is a Lagrangian in $Z$. Since $\Omega^{n+1}\Dual L'' = (\Dual P)[-n-k]$ is also $k$-connective (since $-k-n \geq k$ by definition of $k$), we see that $\Omega^{n+1} \Dual L'' \to Z$ belongs to $\M_n^{a,a}(R,\QF)$. Therefore we deduce that $[X,q] = [X',q'] = 0 \in \L_n^{a,a}(R;\QF)$ as claimed.
\end{proof}

\begin{proposition}
\label{proposition:witt}%
Let \(\QF\) be a $0$-quadratic Poincaré structure on \(\Dperf(R)\).
Then the canonical map
\[
\W(\Proj(R);\QF)  \lto \L_0(R;\QF)
\]
is an isomorphism. 
\end{proposition}
\begin{proof}
First we recall from Remark~\refthree{remark:Witt-group} that $\W(\Proj(R);\QF) \cong \L_0^{0,0}(R;\QF)$. Under this isomorphism, the map under consideration factors as
\[
\L^{0,0}_0(R;\QF) \to \L^{0,1}_0(R;\QF) \to \L_0(R;\QF).
\]
The first map is an isomorphism by Proposition~\refthree{proposition:00-01} and the second map is an isomorphism by Proposition~\refthree{proposition:surgery-quadratic} since $\QF$ is $0$-quadratic. 
\end{proof}

\begin{corollary}
\label{corollary:L-zero}%
\ %
\begin{enumerate}
\item
\label{item:L-groups-one}%
\(\L^{\gs}_0(R;M)\) is naturally isomorphic to the Witt group 
of $M$-valued symmetric forms over \(R\).
\item
\label{item:L-groups-two}%
\(\L^{\gev}_0(R;M)\) is naturally isomorphic to the Witt group of \(M\)-valued even forms over \(R\).
\item
\label{item:L-groups-three}%
\(\L_0(R;\Qgen{m}{M})\) is naturally isomorphic to the Witt group 
of \(M\)-valued quadratic forms over \(R\) for every \(m \geq 2\).
\end{enumerate}
\end{corollary}
\begin{proof}
For \(\L_0(R;\Qgen{m}{M})\) with \(m=0,1,2\) we simply apply Proposition~\refthree{proposition:witt} and invoke the explicit description of \(\Qgen{m}{M}(P[0])\) for \(m=0,1,2\) in terms of symmetric, even and quadratic forms, respectively. The case of \(\L_0(R;\Qgen{m}{M})\) for \(m > 2\) reduces to that of \(m=2\) since the natural map
\[
\L^{\qdr}_0(R;M) \lto  \L_0(R;\Qgen{m}{M})
\]
is an isomorphism for \(m \geq 2\) by Corollary~\refthree{corollary:genuine-is-quadratic}.
\end{proof}

\begin{remark}
\label{remark:L-one}%
Combining Corollary~\refthree{corollary:L-zero} and the equivalences
\(\L^{\gs}_{-2k}(R;M((-1)^k)) \cong \L_0\big(R;\Qgen{k}{M}\big)\)
given by Corollary~\refthree{corollary:periodicity-L} we obtain a description of all the \emph{even non-positive} genuine symmetric \(\L\)-groups of \(R\) in terms of Witt groups of symmetric, even or quadratic forms, see Theorem~\refthree{theorem:main-theorem-L-theory}. A similar description can be obtained for the corresponding \emph{odd} \(\L\)-groups of degrees \(\leq 1\) in terms of symmetric (in degree \(1\)), even (in degree \(-1\)) and quadratic (in odd degrees \(\leq -3\)) \emph{formations}. We leave the details to the motivated reader. 
\end{remark}

\begin{remark}
When $M$ is free as an $R$-module, the results of this section apply equally well if we restrict attention to the full subcategory \(\Dfree(R) \subseteq \Dperf(R)\) of finite free complexes. Indeed, in the proof of Proposition~\refthree{proposition:surgery-quadratic} we may simply choose \(P\) to be free, in which case the algebraic surgery procedure stays within \(\Dfree(R)\). In addition, the connectivity bounds obtained by that proposition translate into the same type of representation by complexes concentrated in certain intervals, only that now these complexes consist of stably free modules, see Remark~\refthree{remark:more-general}. 
For example, if \(\QF\) is a \(0\)-quadratic Poincaré structure then any element of \(\L_n(\Dfree(R),\QF)\) with \(n \geq 0\) can be represented by a Poincaré form on a complex of stable free \(R\)-modules concentrated in degrees \([-n,0]\), and such a Poincaré complex represents zero if and only if it admits a Lagrangian represented by a complex of stable free modules concentrated in degrees \([-n-1,0]\). Similarly, the proof of Proposition~\refthree{proposition:witt} can be run verbatim with stably free modules instead of projective modules, and so we get that in the situation of that proposition, \(\L_0(\Dfree(R),\QF)\) is isomorphic to the corresponding Witt groups of stably free Poincaré objects. More generally, one can take any intermediate subcategory \(\Dfree(R) \subseteq\C\subseteq \Dperf(R)\)  which is closed under the duality.
A typical such \(\C\) is the full subcategory of objects whose class in \(\K_0(R)\) lies in a given involution-closed subgroup of \(\K_0(R)\). 
We will consider this framework again in \S\refthree{section:localisation-devissage} when we will discuss \emph{control} on $\GW$ and $\L$ spectra.
\end{remark}

\subsubsection*{Genuine symmetric L-theory} 
In this subsection, we use the previous surgery results to  identify the genuine symmetric L-groups 
\[
\L^{\gs}_n(R;M):=\L_n(R;\QF^{\gs}_M)
\]
with Ranicki's original definition of symmetric L-groups, which we recall now.

 We let \(\Ch(\Proj(R))\) be the category of bounded chain complexes of finitely generated projective \(R\)-modules. We will say that \(C \in \Ch(\Proj(R))\) is \defi{\(n\)-dimensional} if it is concentrated in the range \([0,n]\), that is, if \(C_i = 0\) whenever \(i < 0\) or \(i > n\).  
Recall that the \(\infty\)-category \(\Dperf(R)\) of perfect left \(R\)-modules can be identified with the \(\infty\)-categorical localisation \(\Ch(\Proj(R))[W^{-1}]\) of \(\Ch(\Proj(R))\) with respect to the collection \(W\) of quasi-isomorphisms.

\begin{definition}
We let $n \geq 0$ be a non-negative integer. An \defi{\(n\)-dimensional Poincaré complex} in the sense of~\cite{RanickiATS1}*{\S 1} is a pair \((C,q)\) where \(C \in \Ch(\Proj(R))\) is an \(n\)-dimensional complex and \(q\) is an element of \(\mathrm{H}_n(\DHom_R(\Dual C, C)^{\hC})\) whose image in \(\mathrm{H}_n(\DHom_R(\Dual C, C)) = [\Dual(C)[n],C]\) is an isomorphism \(\Dual(C)[n] {\to} C\) in the homotopy category of \(\Ch(\Proj(R))\). Here, \(\DHom_R(-,-)\) denotes the internal Hom complex and \((-)^{\hC}\) is the homotopy fixed point construction (described explicitly in~\cite{RanickiATS1} using the standard projective resolution of \(\ZZ\) as a trivial \(C_2\)-module). An \defi{$(n+1)$-dimensional Poincaré pair} is a pair \((f,\eta)\) where \(f\colon C \to C'\) is a map in \(\Ch(\Proj(R))\) from an \(n\)-dimensional complex to an \((n+1)\)-dimensional complex and \(\eta\) is an element of 
\[
\mathrm{H}_n(\fib[\DHom_R(\Dual C, C) \to \DHom_R(\Dual C', C')]^{\hC})
\]
whose respective images in \([\Dual(C)[n], C]\) and \([\cof(\Dual C' \to \Dual C)[n], C']\) are isomorphisms in the homotopy category. Every Poincaré pair \((C \to C',\eta)\) determines, in particular, a Poincaré complex \((C,\eta|_C)\), and we say that a Poincaré complex is \defi{null-cobordant} if it is obtained in this way. Similarly, two \(n\)-dimensional Poincaré complexes \((C,q),(C',q')\) are said to be \defi{cobordant} if \((C \oplus C',q \oplus -q')\) is null-cobordant. The set of equivalence classes of \(n\)-dimensional Poincaré complexes modulo the cobordism relation above forms an abelian group  \(\L^{\short}_{n}(R;M)\) under direct sum, with the inverse of \((C,q)\) given by \((C,-q)\). We will refer to these groups as the  \defi{short} symmetric \(\L\)-groups of \(R\). 
\end{definition}

\begin{remark}\
\begin{enumerate}
\item 
The groups $\L^{\short}_n(R;M)$ are formally defined in \cite{RanickiATS1} only when $M$ is of the form $R(\epsilon)$ for some $\eps$-involution on $R$. However, the definition only makes use of the induced duality on chain complexes and it therefore makes sense for any invertible $\ZZ$-module $M$ with involution. 
\item
In \cite{RanickiATS1} Ranicki extends the definition of the classical symmetric L-groups to negative integers as follows:
\[
\L^{\short}_n(R;M) = \begin{cases} 
\L^{\short,\ev}_{n+2}(R;-M) 
& \text{ for } n=-2,-1 \\ 
\L^{\qdr}_{n}(R;M)
 & \text{ for } n \leq -3 \end{cases}
\]
Here $\L^{\short,\ev}_{n}(R;M)$ are the even L-groups
from \cite{RanickiATS1}*{\S3}.
We recall that an $n$-dimensional Poincaré complex $(C,\varphi)$ is called even in \cite{RanickiATS1} if a certain Wu class $v_0(\phi) \colon \mathrm{H}^n(C) \to \hat{\mathrm{H}}^0(\Ct;M)$ vanishes. Likewise, an $(n+1)$-dimensional Poincaré pair $f\colon C \to C'$ is called even if its relative Wu class $\mathrm{H}^{n+1}(f) \to \hat{\mathrm{H}}^0(\Ct;M)$ vanishes. The short even $\L$-groups $\L^{\short,\ev}_n(R;M)$ are then the cobordism groups of $n$-dimensional even complexes. For $n\geq 0$ we will also show that they are equivalent to our genuine even L-groups, see the proof of Theorem~\refthree{theorem:main-theorem-L-theory}. 
\end{enumerate}
\end{remark}

\begin{remark}
Two Poincaré complexes \((C,q),(C',q')\) are said to be quasi-isomorphic if there exists a quasi-isomorphism \(f \colon C \to C'\) such that \(f_*q = q'\). This yields an equivalence relation which is finer than cobordism: Given a quasi-isomorphism \(f\colon (C,q) \to (C',q')\) one can construct a Poincaré pair of the form \((\id,f)\colon C \to C \oplus C'\) witnessing  \((C,q)\) and \((C',q')\) as cobordant. In particular, if \(\Poinc^{\short}_n(R)\) denotes the monoid of quasi-isomorphism classes of \(n\)-dimensional Poincaré complexes and 
\(\pM^{\short}_n(R)\) the monoid of $n$-dimensional Poincaré pairs,
then \(\L^{\short}_{n}(R;M)\) is naturally isomorphic to the cokernel of $\pM^{\short}_n(R)\to \Poinc^{\short}_n(R)$ in the category of commutative monoids. 
\end{remark}

The remainder of this subsection is devoted to a proof of the following theorem.
\begin{theorem}
\label{theorem:main-theorem-L-theory}%
Let $R$ be a ring and $M$ an invertible $\ZZ$-module with involution over $R$. Then for all integers $n$, there is a natural isomorphism
\[
\L^{\short}_{n}(R;M) \cong \L^{\gs}_n(R;M)
\]
between Ranicki's classical symmetric L-groups and the genuine symmetric L-groups.
\end{theorem}

The proof will proceed in several steps. We first compare, for $n\geq0$, Ranicki's classical L-group $\L^{\short}_{n}(R;M)$ to the group \(\L^{n,n+1}_n(R;\QF^{\s}_M)\) of Definition~\refthree{definition:connective-L} as follows.
Let \(\I_n \subseteq \Ch(\Proj(R))\) be the subcategory consisting of the \(n\)-dimensional complexes and quasi-isomorphisms between them, and let  \(\J_{[0,n]} \subseteq \Dperf(R)^{\simeq}\) be the full sub-\(\infty\)-groupoid spanned by those perfect \(R\)-modules which can be represented by a complex in \(\Proj(R)\) concentrated in degrees \([0,n]\). The canonical localisation map then restricts to a functor \(\pi\colon \I_n \to \J_{[0,n]}\), and it induces isomorphisms
\[
\rho\colon \mathrm{H}_n(\DHom_R(\Dual(C), C)^{\hC})\stackrel{\cong}{\longrightarrow} \pi_n(\map_R(\Dual(\pi C), \pi C)^{\hC})\cong\pi_0\Omega^n\QF^{\s}_{M}(\pi \Dual C)
\]
natural in the object $C$ of $\I_n$. We can then define a map of sets
\[
\Poinc^{\short}_n(R)\longrightarrow \pi_0\Poinc_{n}^{n}(R;\QF^{\s}_{M})
\]
by sending an $n$-dimensional Poincaré complex $(C,q)$ to the component determined by $(\pi (\Dual C),\rho(q))$. This map is in fact an isomorphism, since the functor $\pi\colon \I_n \to \J_{[0,n]}$ is an equivalence on homotopy categories, and $\rho$ is an isomorphism.
Since the localisation functor \(\Ch(R) \to \Dperf(R)\) preserves direct sums this is moreover an isomorphism of monoids.

\begin{proposition}
\label{proposition:alternative-defi}%
For every \(n \geq 0\), the previously defined map induces a group isomorphism
\[
\L^{\short}_{n}(R;M) \cong \L^{n,n+1}_n(R;\QF^{\s}_{M}).
\]
\end{proposition}
\begin{proof}
By replacing \(\Ch(R)\) and \(\Dperf(R)\) by their arrow categories a similar construction provides a morphism of monoids
\[
\pM^{\short}_n(R) \longrightarrow \pi_0\pM_{n}^{n,n+1}(R;\QF^{\s}_{M})
\]
which is compatible with the morphism \(\Poinc^{\short}_n(R) \to \pi_0\Poinc_{n}^{n}(R;\QF^{\s}_{M})\).
Thus we obtain a well-defined group homomorphism on L-groups. Since every arrow in \(\Dperf(R)\) can be lifted to a map of chain complexes, an argument similar to the one above shows that this map is also surjective. This suffices to induce an isomorphism on quotients.
\end{proof}

The next step for the proof of Theorem~\refthree{theorem:main-theorem-L-theory} is to see that on $n$-dimensional complexes, the datum of a symmetric form is the same as the datum of a genuine symmetric form. We record here the corresponding statement for L-groups, For the following result, we keep in mind that $\QF^{\gs}_M$ is 2-symmetric:
\begin{lemma}
\label{lemma:genuine-to-symmetric}%
Let \(\QF\) be \(r\)-symmetric for \(r \in \ZZ\). 
Let $a,b,n \in \ZZ$ be as in Definition~\refthree{definition:connective-L}, and suppose additionally that
\(a \leq n+2r-4\) and \(b \leq n+2r-3\). Then the map $\L^{a,b}_n(R;\QF) \to \L^{a,b}_n(R;\QF^{\s}_M)$
is an isomorphism. In particular, the map 
\[
\L_n^{n,n+1}(R;\QF^{\gs}_M) \lto \L_n^{n,n+1}(R;\QF^{\s}_M)
\]
is an isomorphism for every $n\geq 0$.
\end{lemma}
\begin{proof}
It will suffice to show that the monoid homomorphisms 
\[
\pi_0\Poinc_{n}^{a}(R;\QF) \to \pi_0\Poinc_{n}^{a}(R;\QF^{\s}_M) \quad\text{and}\quad \pi_0\pM_{n}^{a,b}(R;\QF) \to \pi_0\pM_{n}^{a,b}(R;\QF^{\s}_M)
\]
are isomorphisms. Now the left homomorphism is an isomorphism since \(\QF\) is \(r\)-symmetric and so the map \(\Om^{\infty+n}\QF(X) \to \Om^{\infty+n}\QF^{\s}_M(X)\) is an equivalence by Remark~\refthree{remark:connectivity} whenever \(X\) is \(\left(\frac{-n-a}{2}\right)\)-connective, taking into account that \(-r-\frac{-n-a}{2} = \frac{-2r+n+a}{2} \leq n-2\) by our assumption. Concerning the right map, it will suffice to show that whenever \(L \to X\) is such that \(L\) is \(\lceil \frac{-n-1-b}{2} \rceil\)-connective and \(X\) is \(\left(\frac{-n-a}{2}\right)\)-connective the map
\[
\QF_{\Met}(L \to X) = \fib[\QF(X) \to \QF(L)] \to \fib[\QF^{\s}_M(X) \to \QF^{\s}_M(L)] = \QF^{\s}_{\Met}(L \to X)
\]
has an \((n-2)\)-truncated fibre. Equivalently, this is the same as saying that the square
\[
\begin{tikzcd}
\QF(X) \ar[r] \ar[d] & \QF^{\s}_M(X) \ar[d]\\
\QF(L) \ar[r] & \QF^{\s}_M(L)
\end{tikzcd}
\]
has \((n-2)\)-truncated total fibre. As in the first part of the proof, we find that the top horizontal map is \((n-2)\)-truncated and the bottom horizontal map is \((n-1)\)-truncated since \(L\) is \(\lceil \frac{-n-1-b}{2} \rceil\)-connective and \(-r - \lceil \frac{-n-1-b}{2} \rceil \leq \frac{-2r+n+1+b}{2} \leq n-1\).
\end{proof}

\begin{proof}[Proof of Theorem~\refthree{theorem:main-theorem-L-theory}]
First, we consider the case $n \geq 0$ where we simply combine the isomorphisms
\[
\L^{\short}_{n}(R;M) \cong \L^{n,n+1}_n(R;\QF^{\s}_{M}) \cong \L^{n,n+1}_n(R;\QF^{\gs}_{M}) \cong \L_n(R;\QF^{\gs}_{M})
\]
of Proposition~\refthree{proposition:alternative-defi}, Lemma~\refthree{lemma:genuine-to-symmetric} and Corollary~\refthree{corollary:surgery-genuine-symmetric}, respectively.

The case $n \leq -3$ is covered by Corollary~\refthree{corollary:genuine-is-quadratic}. It then suffices to treat the case $n=-2,-1$. Here, we use the ``periodicity''
\[
\L_n^{\gs}(R;M) \cong \L_{n+2}^{\gev}(R;-M)
\]
of Corollary~\refthree{corollary:periodicity-L} and will now argue more generally that for $n\geq 0$, a canonical map $\L^{\short,\ev}_n(R;M) \to \L^{\gev}_n(R;M)$ is an isomorphism. To construct the map, we consider an $n$-dimensional even complex $(C,\varphi)$ and obtain from the construction preceding Proposition~\refthree{proposition:alternative-defi} a canonical element $(X,q)$ of $\L_n^{n,n+1}(R;\QF^\gs_M)$, represented by $\Dual C$. We want to argue that $(X,q)$ refines to an element of $\L_n^{n,n+1}(R;\QF^\gev_M)$. Let us consider the fibre sequence
\[
\Omega^n\QF^\gev(X) \lto \Omega^n \QF^\gs(X) \lto  \map_R(X;\hat{\mathrm{H}}^0(\Ct;M)[-n]).
\]
We note the equivalence $\Omega^\infty \map_R(X;\hat{\mathrm{H}}^0(\Ct;M)[-n]) \simeq \mathrm{Hom}_R(\mathrm{H}_{-n}(X),\hat{\mathrm{H}}^0(\Ct;M))$. Tracing through the definitions, the symmetric structure $q$ is sent to the Wu class $v_0(\varphi)$ which is zero by assumption. We deduce that $(X,q)$ canonically refines to an element of $\L_n^{n,n+1}(R;\QF^\gev_M)$. Likewise, an $(n+1)$-dimensional even pair gives rise to a genuine even structure on the associated cobordism. Reversing the above argument, we deduce that the map
\[
\L^{\short,\ev}_n(R;M) \lto \L_n^{n,n+1}(R;\QF^\gev_M)
\]
is an isomorphism. Combining this with the isomorphism
\[
\L^{n,n+1}_n(R;\QF^\gev_M) \lto \L_n(R;\QF^\gev_M)
\]
obtained from Proposition~\refthree{proposition:surgery-quadratic} just as Corollary~\refthree{corollary:surgery-genuine-symmetric}, the theorem follows.
\end{proof}

\begin{remark}
In~\cite{RanickiATS1} Ranicki also defines, for \(n \geq 0\), the \(n\)'th \emph{quadratic \(\L\)-group}, using quadratic Poincaré complexes of dimension \(n\). The same argument as above shows that this group coincides with \(\L^{n,n+1}_n(R;\QF^{\qdr}_{M})\), and hence with \(\L_n(R;\QF^{\qdr}_{M})\) by Proposition~\refthree{proposition:surgery-quadratic}. As shown in~\cite{RanickiATS1}, these are also the same as the quadratic \(\L\)-groups of Wall~\cite{wall}, which arise in manifold theory as the natural recipient of surgery obstructions. We warn the reader that these groups do not agree with $\L_n(R;\QF^{\gq}_{M})$, which by periodicity are isomorphic to $\L_{n-4}(R;\QF^{\gs}_{M})$.
\end{remark}

\subsubsection*{Surgery for connective ring spectra}

In this section we apply the surgery arguments previously developed to the case where $R$ is replaced with a connective ring spectrum \(A\) (that is, a connective $\Eone$-algebra in the monoidal $\infty$-category of spectra). The symbol $\otimes$ here consequently refers to the tensor product over the sphere spectrum. The perfect derived category \(\Dperf(R)\) is then replaced with the stable $\infty$-category $\Mod_A^\omega$ of compact $A$-module spectra. 
We call a compact \(A\)-module \emph{projective} if it is a retract of \(A^n\) for some \(n\), and write \(\Proj(A) \subseteq \Modp{A}\) for the full subcategory spanned by the projective \(A\)-modules.
We now fix for the remainder of this section a Poincaré structure on \(\Modp{A}\). The duality \(\Dual\) underlying $\QF$ on $\Mod_A^\omega$ is then induced by an invertible module with involution $M$ as in Definition~\refone{definition:module-with-involution}, which we assume to be projective in either of its two $A$-module structures, so that 
\(\Dual\) preserves $\Proj(A)$. We refer to \S\refone{section:modules} for a complete treatment of Poincaré structures on  categories of module spectra.

As in Definition~\refthree{definition:r-sym}, we say that $\QF$ is $m$-quadratic if the cofibre of the map 
\[
\QF^\qdr_M(X)=\hom_{A\otimes A}(X\otimes X,M)_{\hC} \lto \QF(X)
\]
sends $A$, or equivalently the full subcategory $\Proj(A) \subseteq \Modp{A}$, to $m$-connective spectra. 
For \(k \in \ZZ\) we say that an \(A\)-module \(X\) is \(k\)-connective if its underlying spectrum is \(k\)-connective, that is, has vanishing homotopy groups in degrees \(< k\). With this notion of connectivity, the definition of $\L_n^{a,b}(A;\QF)$ from~\refthree{definition:connective-L} carries over verbatim to this more general situation. Proposition~\refthree{proposition:surgery-quadratic} does as well:

\begin{proposition}[Surgery for connective ring spectra]
\label{proposition:surgery-quadratic-Eone}%
Suppose that \(\QF\) is $m$-quadratic. 
Fix an $n\in \ZZ$ and let \(a,b \geq 0\) be two non-negative integers with $b\geq a-1$, and such that $n+a$ and $n+1+b$ are even.
\begin{enumerate}
\item
\label{item:surgery-ring-1}%
If $a \geq n-2m$ then any Poincaré object in \((\Modp{A},\QF\qshift{-n})\) is cobordant to one which is \(\big(\tfrac{-n-a}{2}\big)\)-connective.
\item
\label{item:surgery-ring-2}%
If $b \geq n-2m+1$ then any Lagrangian \(L \to X\) of a \(\big(\tfrac{-n-a}{2}\big)\)-connective Poincaré object \((X,q) \in \Poinc(\Modp{A},\QF\qshift{-n})\) is cobordant relative to \(X\) to a Lagrangian \(L'\to X\) such that both \(L'\) and \(\fib[L'\to X]\) are \(\big(\tfrac{-n-1-b}{2}\big)\)-connective. 
\end{enumerate}
In particular, if both inequalities of \refthreeitem{item:surgery-ring-1} and \refthreeitem{item:surgery-ring-2} hold, the monoid $\L^{a,b}_n(A;\QF)$ is a group and the canonical map $\L^{a,b}_n(A;\QF) \to \L_n(A;\QF)$ is an isomorphism.
\end{proposition}

The proof of Proposition~\refthree{proposition:surgery-quadratic-Eone} will use the following standard fact.
\begin{lemma}
\label{lemma:finitely-generated}%
Let \(X\) be a compact connective \(A\)-module. Then \(\pi_0 X\) is finitely presented as a \(\pi_0 A\)-module. In particular, there exists a map \(A^n \to X\) whose cofibre is \(1\)-connective.
\end{lemma}
\begin{proof}
We note that $\pi_0(A) \otimes_A X$ is a connective and perfect $\pi_0(A)$-module with $\pi_0(\pi_0(A) \otimes_A X) \cong \pi_0(X)$. We may therefore assume that $A$ is discrete, in which case it follows from the fact that connective and perfect modules over a discrete ring can be represented by finite chain complexes of finitely generated projective $A$-modules concentrated in non-negative degrees, see also Lemma~\refthree{lemma:range}.
\end{proof}

\begin{proof}[Proof of Proposition~\refthree{proposition:surgery-quadratic-Eone}]
We first note that the statement of Lemma~\refthree{lemma:connectivity-2}, which is used in the proof of Proposition~\refthree{proposition:surgery-quadratic}, holds in the present setting with the exact same proof, where \(P[k]\) is understood as the \(k\)-fold suspension \(\Sig^k P \in \Modp{A}\) of a projective \(A\)-module \(P\). All the other results referred to in the proof of Proposition~\refthree{proposition:surgery-quadratic} are formulated for a general Poincaré \(\infty\)-category 
except for Lemma~\refthree{lemma:range}, which is used solely for the following purpose: to show that if \(X\) is \(k\)-connective then there exists a projective module \(P\) and a map \(P[k] \to X\) whose cofibre is \((k+1)\)-connective. This last statement holds in the present setting as well (taking again \(P[k]\) to mean the \(k\)-fold suspension of \(P\)) by Lemma~\refthree{lemma:finitely-generated}. 
Having taken care of all these preliminaries, the proof of Proposition~\refthree{proposition:surgery-quadratic-Eone} now proceeds verbatim as that of Proposition~\refthree{proposition:surgery-quadratic}.
\end{proof}

\begin{remark}
\label{remark:weight}%
Using Remark~\refthree{lemma:range} we observed in Remark~\refthree{remark:intervalsLab} that Poincaré objects satisfying the connectivity requirements for \(\L_n^{a,b}(A;\QF)\) are actually representable by complexes concentrated in certain degrees, and similarly for Lagrangians. We may recognize some aspects of this behaviour in the present more general setting:
\begin{enumerate}
\item
If \(X \in \Modp{A}\) is projective then so is \(\Dual X\), and in particular both \(X\) and \(\Dual X\) are connective. Conversely, if both \(X\) and \(\Dual X\) are connective then \(X\) is projective. To see this, let 
\[
A^n \to X \xrightarrow{f} Y
\]
be a fibre sequence such that \(Y\) 1-connective (such a sequence exists by Lemma~\refthree{lemma:finitely-generated}). Now the condition that \(M\) is an invertible \(A \otimes_{\SS} A\)-module implies that \(\map_A(X,Y) \simeq \map_A(M,Y) \otimes_A \Dual X\), where we consider \(\map_A(M,Y)\) as a right \(A\)-module via pre-composition by the second \(A\)-action on \(M\). Since \(\Dual X\) is connective and \(Y\) is 1-connective we conclude that \(\map_A(X,Y)\) is 1-connective, and hence the map \(f\) is null-homotopic. It follows that the map \(A^n \to X\) admits a section up to homotopy, so that \(X\) is a retract of \(A^n\). More generally, if \(X\) is \(k\)-connective and \(\Dual X\) is \((-k)\)-connective then \(X\) is the \(k\)-fold suspension of a projective module, and vice versa.
\item
If \(X\) is connective and \(\Dual(X)\) is \((-1)\)-connective then \(X\) might fail to be projective, but only mildly so: in this case, we can write \(X\) as a cofibre of a map \(P \to Q\) between two projective \(A\)-module spectra. If we compare with the situation of perfect derived categories of discrete rings, this property corresponds to being a representable by a complex of projectives concentrated in degrees \([0,1]\). To see that this holds in our setting let us again consider a fibre sequence as above 
with \(Y\) 1-connective. Then \(P := \Om Y\) is connective and \(\Dual_M P\) is the cofibre of the dual map \(\Dual_M X \to \Dual_M A^n = M^n\) and is hence connective. It then follows by the previous point that \(P\) is projective, and so \(X\) is the cofibre of a map \(P \to A^n\) between projective modules, as claimed. More generally, if \(X\) is \(k\)-connective and \(\Dual X\) is \((-k-1)\)-connective then \(X\) is the \(k\)-fold suspension of the cofibre of a map of projective modules, and vice versa.
\end{enumerate}
We spell out these two points since we will make use of them just below, but let us point out that these are just two steps in an exhaustive filtration on connective \(A\)-module spectra obtained by requiring weaker and weaker connectivity conditions on the duals (or more generally, a bi-indexed filtration on all of \(\Modp{A}\)). This can be conveniently axiomatised in the framework of \emph{weight structures}, see~\cite{comparison}*{\S 3} for a treatment in the setting of Poincaré \(\infty\)-categories.
\end{remark}

Using surgery methods we now obtain the following result, see also \cite{Lurie-L-theory}*{Lecture 14} for a proof of the algebraic $\pi$-$\pi$-theorem, Corollary~\refthree{corollary:pi-pi}\refthreeitem{item:cor-pi-pi-one} below. We say that a map of spectra is $k$-connective for some $k\in\ZZ$ if its fibre is. As in~\paperone and~\papertwo, we write $\Lin_\QF$ for the linear part of a quadratic functor $\QF$.

\begin{proposition}
\label{proposition:pi-pi}%
Fox \(m,p \in \ZZ\). Let $f\colon A\to B$ be a \(1\)-connective map of connective ring spectra, \(M,N\) projective invertible modules with involution over \(A,B\) respectively, and $\QF_M$ and $\QF_N$ Poincaré structures on $\Mod_A^\omega$ and $\Mod_B^\omega$ which are \(m\)-quadratic with respect to \(M,N\), respectively. Suppose given a refinement of the extension of scalars functor $f_!$ to a Poincaré functor $(\Mod_A^\omega,\QF_M) \to (\Mod_B^\omega,\QF_N)$ such that the induced map \(M = \Bil_{\QF_M}(A,A) \to \Bil_{\QF_N}(B,B) = N\) is \(1\)-connective and the induced map
$\Lin_{\QF_M}(A) \to \Lin_{\QF_N}(B)$ is $p$-connective. Then the map
\begin{equation}
\label{equation:pi-pi-map}%
\L_n(A;\QF_M) \lto \L_n(B;\QF_N) 
\end{equation}
is an isomorphism for $n\leq \min(2p-1,2m)$ and surjective for $n=\min(2p,2m+1)$.
\end{proposition}

Taking \(f\) to be the identity of $A$ and \(\QF_M\) to be \(\QF_M^{\qdr}\), we recover a generalisation of Corollary~\refthree{corollary:genuine-is-quadratic} to the present setting:
\begin{corollary}
\label{corollary:genuine-is-quadratic-Eone}%
Let \(\QF_M\) be a Poincaré structure on \(\Modp{A}\) which is $m$-quadratic with respect to a projective invertible module with involution \(M\). Then the canonical map 
\[
\L_n^\qdr(A;M) = \L_n(A;\QF^\qdr_M) \lto \L_n(A;\QF_M)
\]
is an isomorphism for \(n \leq 2m-3\) and surjective for \(n = 2m-2\). \end{corollary}

\begin{example}
Corollary~\refthree{corollary:genuine-is-quadratic-Eone} can be applied for the universal Poincaré structure $\QF^\uni$ on $\Mod_\SS^\omega$ from Example~\refone{example:universal-category} which is 0-quadratic as its linear part is given by \(\map(-,\SS)\), to obtain that $\L^\qdr_n(\SS) \to \L^\uni_n(\SS)$ is an equivalence in degrees $\leq -3$. Interestingly, since $\SS^{\tC}\simeq \SS^\cwedge_2$ is connective by Lin's Theorem \cite{lin}, the symmetric Poincaré structure $\QF^\s$ on $\Mod_\SS^\omega$ is also $0$-quadratic.
Thus, the map $\L^\qdr_n(\SS) \to \L^\s_n(\SS)$ is also an isomorphism for $n\leq -3$.
Both of these observations also follow  from work of Weiss-Williams \cite{WWIII} who give an explicit formula for the cofibre of the maps in question. 
\end{example}

\begin{remark}
\label{remark:homotopy-category-Proj}%
In the proof of Proposition~\refthree{proposition:pi-pi}, we will frequently make use of the following fact \cite{HA}*{Corollary 7.2.2.19}. Namely, for a map $f\colon A \to B$ between connective ring spectra which induces an isomorphism on $\pi_0$, in particular for $1$-connective $f$ as in Proposition~\refthree{proposition:pi-pi}, the extension of scalars functor $f_! \colon \Proj(A) \to \Proj(B)$ induces an equivalence on homotopy categories. In particular, we can lift (uniquely up to homotopy) projective $B$-modules and maps between such to projective $A$-modules and maps between such. Consequently, we can also lift perfect $B$-modules presented as (co)fibres of maps between projective $B$-modules to perfect $A$-modules of the same kind.
\end{remark}

\begin{proof}[Proof of Proposition~\refthree{proposition:pi-pi}]
First, we show that for an \(A\)-module spectrum $X \in \Modp{A}$ such that $\Dual X$ is $r$-connective the map $\QF_M(X) \to \QF_N(f_!X)$ is $\min(2r+1,r+p)$-connective.
To see this, let us consider the diagram of horizontal cofibre sequences
\[
\begin{tikzcd}
\QF_M^\qdr(X) \ar[r] \ar[d] & \QF_M(X) \ar[r] \ar[d] & \Lin_{\QF_M}(X)[r] \ar[d] \\
\QF_N^\qdr(f_!X) \ar[r] & \QF_N(f_!X) \ar[r] & \Lin_{\QF_N}(f_!X)[r].
\end{tikzcd}
\]
Then the connectivity assumptions on the maps induced on bilinear and linear parts imply that the left vertical map is $(2r+1)$-connective and that the right vertical map is $(r+p)$-connective, respectively, and so the claim follows. In particular, the map $\QF_M(X) \to \QF_N(f_!X)$ is \((2r+1)\)-connective when $r \leq p-1$, and \(2r\)-connective when \(r=p\).

Let us now fix an \(n \leq \min(2p,2m+1)\) and let \(a \in \{0,1\}\) be such that \(n+a\) is even. Then \(a \geq n-2m\) (where we note that if \(n=2m+1\) then \(a=1\)) and so by Proposition~\refthree{proposition:surgery-quadratic-Eone}~\refthreeitem{item:surgery-ring-1} every class in either \(\L_n(A;\QF_M)\) or \(\L_n(B;\QF_N)\) can be represented by a Poincaré object which is \(\big(\tfrac{-n-a}{2}\big)\)-connective. 
To prove the surjectivity part of the statement it will then suffice to show that the monoid map
\[
\pi_0\Poinc_{n}^{a}(A,\QF_M) \lto \pi_0\Poinc_{n}^{a}(B,\QF_N)
\]
is surjective. 
Now by Remark~\refthree{remark:weight} we have that if \(X \in \Modp{B}\) 
is a $(\frac{-n-a}{2})$-connective \(B\)-module such that \(\Dual X\) is \(\big(\frac{n-a}{2}\big)\)-connective, 
then the \(\big(\tfrac{n-a}{2}\big)\)-fold suspension of \(X\) is 
a projective module if \(a=0\) and the fibre of a map between projective modules if \(a=1\). 
By Remark~\refthree{remark:homotopy-category-Proj}
such \(B\)-modules can be lifted to \(A\)-modules of the same type, and so it will suffice to prove that for every $(\tfrac{-n-a}{2})$-connective \(A\)-module \(X \in \Modp{A}\) the map 
\[
\QF_M(X) \lto \QF_N(f_!X)
\]
is surjective on \(\pi_n\). Indeed, this holds by the argument in the beginning of the proof since \(\Dual X \simeq \Sig^n X\) is \(r\)-connective with \(r:= \tfrac{n-a}{2} \leq p\), and so the map in question 
is at least \(n\)-connective (where we note that \(a=0\) if \(r=p\)). 

To prove the injectivity of the map~\eqrefthree{equation:pi-pi-map}, let us now assume further that \(n \leq \min(2p-1,2m)\) and let \(b \in \{0,1\}\) be such that \(n+1+b\) is even (keeping \(a\) as above). 
Then \(b \geq n-2m+1\) (where we note that if \(n=2m\) then \(b=1\)) and so by Proposition~\refthree{proposition:surgery-quadratic-Eone}~\refthreeitem{item:surgery-ring-2} every \(\big(\tfrac{-n-a}{2}\big)\)-connective Poincaré object $(Y,q)$ in \((\Modp{B},\QF_N\qshift{-n})\) which admits a Lagrangian, also admits a Lagrangian \(L \to Y\) such that \(L\) and \(\fib[L \to Y]\) are \(\big(\tfrac{-n-1-b}{2}\big)\)-connective. 
Let us therefore assume that \((X,q)\) is a \((\frac{-n-a}{2})\)-connective Poincaré object in \((\Modp{A},\QF_M\qshift{-n})\) whose image in \((\Modp{B},\QF_N\qshift{-n})\) is equipped with a Lagrangian \(L \to f_!X\) such that \(L\) and $N := \fib(L \to  f_!X)$ are \(\big(\tfrac{-n-1-b}{2}\big)\)-connective. We will then show that \(L \to f_!(X)\) can be lifted to a Lagrangian of \((X,q)\) with respect to \(\QF_M\qshift{-n}\). Indeed, we first note that \(L\) is a \(\big(\tfrac{-n-1-b}{2}\big)\)-connective \(B\)-module whose dual \(\Dual L \simeq \Sig^{n+1}N\) is \(r' := \big(\tfrac{n+1-b}{2}\big)\)-connective, and so by Remark~\refthree{remark:weight} 
we may write $L = \fib(U \to V)$ for $U,V \in \Proj(B)[-r']$ and similarly we may write $X= \cof(P\to Q)$ for $P,Q \in \Proj(A)[-r']$. In addition, we may choose $P = 0$ if $a=0,b=1$ and $V = 0$ if $a=1,b=0$. 
It is a direct consequence of the projectivity of $P,Q,U$, and $V$ that the map $L \to f_!X)$ factors as the composition
\[
L \lto U \stackrel{\alpha}{\lto} f_!Q \lto f_!X.
\]
Using Remark~\refthree{remark:homotopy-category-Proj} we may lift $U \to V$ to a map $U' \to V'$,  and the map $\alpha$ to $\alpha' \colon U' \to Q$. We then define $L'$ as the fibre of $U'\to V'$ and therefore we lift the map $L \to f_!X$ to the composite $L' \to U' \to Q \to X$.
In order to extend this lift to a lift of the entire surgery datum we claim that the map
\[
\QF_M(L') \lto \QF_N(f_!L')
\]
is injective on \(\pi_n\) and surjective on \(\pi_{n+1}\). Indeed, this follows from the argument at the beginning of the proof since \(\Dual L \simeq \Sig^{n+1} N\) is \(r'\)-connective, and so the map \(\QF_M(L') \to \QF_N(f_!L')\) is 
at least \((n+1)\)-connective (where we note that \(b=0\) when \(r'=p\)).
\end{proof}

We can now apply Proposition~\refthree{proposition:pi-pi} to compare the L-spectra of connective ring spectra to the L-spectra of their $\pi_0$'s. Let $A$ be a connective ring spectrum and $M$ be a projective invertible module with involution over $A$. Associated to it, we have the Poincaré structures $\QF_M^\qdr$ and $\QF_M^\sym$ given by  
\[
\QF_M^\qdr(X) = \map_{A\otimes A}(X\otimes X,M)_{\hC} \quad \text{ and } \quad \QF_M^\sym(X) = \map_{A\otimes A}(X\otimes X,M)^{\hC}.
\]
In addition, we have the Poincaré structure $\QF_M^{\gs} = \QF_M^{\geq 0}$, see \refone{notation:genuine-zero-one-two-quadratic-functors} and for an $\Einf$-ring spectrum $B$, the Tate Poincaré structure $\QF^\tate_B$, see Example~\refone{example:tate-structure}. Explicitly, these are given by the following pullbacks
\[
\begin{tikzcd}
	\QF^{\gs}_M(X) \ar[r] \ar[d] & \map_A(X,\tau_{\geq 0}M^{\tC}) \ar[d] & & \QF^{\tate}_B(Y) \ar[r] \ar[d] & \map_B(Y,B) \ar[d] \\
	\QF^\sym_M(X) \ar[r] & \map_A(X,M^{\tC}) & & \QF^\sym_B(Y) \ar[r] & \map_B(Y,B^{\tC})
\end{tikzcd}
\]
where the right most vertical map is induced by the Tate-valued Frobenius $B\to B^{\tC}$. The maps $A \to \pi_0(A)$ and  $M \to \pi_0(M)$ induce canonical Poincaré functors $(\Mod_A^\omega,\QF_M) \to (\Mod_{\pi_0(A)}^\omega,\QF_{\pi_0(M)})$ for $\QF_M$ denoting both of the Poincaré structures $\QF^\qdr_M$ and $\QF^\gs_M$. 
Likewise, $B\to \pi_0(B)$ induces a canonical Poincaré functor $(\Mod_B^\omega,\QF^\tate_B) \to (\Mod_{\pi_0(B)}^\omega,\QF^\tate_{\pi_0(B)})$.

\begin{corollary}
\label{corollary:pi-pi}%
For a connective ring spectrum $A$, a projective invertible module with involution $M$ over $A$, and a connective $\Einf$-ring spectrum $B$, we have that:
\begin{enumerate}
\item
\label{item:cor-pi-pi-one}%
The map $\L^\qdr(A;M) \to \L^\qdr(\pi_0 A;\pi_0(M))$ is an equivalence. 
\item
\label{item:cor-pi-pi-two}%
The map $\L^\tate_n(B) \to \L^\tate_n(\pi_0(B))$ is an isomorphism for $n\leq 0$ and surjective for $n=1$.
\item
\label{item:cor-pi-pi-three}%
The map $\L^\gs_n(A;M) \to \L^\gs_n(\pi_0(A);\pi_0(M))$ is an isomorphism for $n\leq -3$ and surjective for $n=-2$. If the $\Ct$-action on the spectrum underlying $M$ is trivial, the map $\L^\gs_n(A;M) \to \L^\gs_n(\pi_0(A);\pi_0(M))$ is an isomorphism for $n \leq -1$ and surjective for $n=0$.
\end{enumerate}
\end{corollary}

\begin{proof}

The claims follow from Proposition~\refthree{proposition:pi-pi}. For \refthreeitem{item:cor-pi-pi-one} this is because the quadratic Poincaré structure $\QF^\qdr_M$ is $m$-quadratic for every $m$, and it has trivial linear part. For \refthreeitem{item:cor-pi-pi-two} we use that $\QF_B^\tate$ is $0$-quadratic, and that the map induced on linear parts is $B \to \pi_0(B)$ which is 1-connective.
Similarly, for \refthreeitem{item:cor-pi-pi-three} we use that $\QF_M^\gs$ is $0$-quadratic, that $A\to \pi_0(A)$ is $1$-connective, and that the map on linear parts  $\tau_{\geq 0} M^{\tC} \to \tau_{\geq 0}(\pi_0(M)^{\tC})$ is $(-1)$-connective. If moreover the involution on $M$ is trivial, the map $\pi_0(M^{\hC})\to \pi_0((\pi_0M)^{\hC})$ is surjective, since in this case the target  is $\pi_0M$ and $M$ is a retract of $M^{\hC}$. Then by the commutativity of the diagram
\[
\begin{tikzcd}
	\pi_0(M^{\hC}) \ar[r] \ar[d] &\pi_0(M^{\tC}) \ar[d]  \\
	\pi_0((\pi_0M)^{\hC}) \ar[r] &\pi_0((\pi_0M)^{\tC})
\end{tikzcd}
\]
and the fact that the bottom horizontal map is surjective, it follows that the right vertical map is surjective and thus that  the map on linear parts  $\tau_{\geq 0} M^{\tC} \to \tau_{\geq 0}(\pi_0(M)^{\tC})$ is in fact $0$-connective.
\end{proof}

\begin{remark}
We note that $\SS^{2-2\sigma}\otimes \ko$ is an example of an invertible module with involution over $\ko$ where the map $(\SS^{2-2\sigma}\otimes \ko)^{\tC} \to (\pi_0(\SS^{2-2\sigma}\otimes\ko))^{\tC} \cong \ZZ^{\tC}$ is not $\pi_0$-surjective, so that the map on linear terms relevant for part~\refthreeitem{item:cor-pi-pi-three} of Corollary~\refthree{corollary:pi-pi} is indeed only $(-1)$-connective in general. To see this, we use that $\pi_*(\ko^{\tC}) = \ZZ^\cwedge_2[u^{\pm}]$ with $|u|=4$, so that $(\SS^{2-2\sigma}\otimes \ko)^{\tC} \simeq \Sigma^2 \ko^{\tC}$ has trivial $\pi_0$, whereas $\ZZ^{\tC}$ has non-trivial $\pi_0$. Another example is given by $\SS^{2-2\sigma}$ as an invertible module with involution over $\SS$.
\end{remark}

We finish this section with further examples related to the sphere spectrum.
\begin{example}
\label{example:pi-pi}%
First, we note that the universal Poincaré structure $\QF^\uni$ agrees with the Tate Poincaré structure $\QF^\tate$ on $\Mod_\SS^\omega$, so we obtain that $\L^\uni_n(\SS) \to \L^\tate_n(\ZZ)$ is an isomorphism for $n\leq 0$ and surjective for $n=1$. %

We may also consider the canonical Poincaré functor $(\Mod_\SS^\omega,\QF^\uni) \to (\Dperf(\ZZ),\QF^\gs_\ZZ)$. Both $\QF^\uni $ and $\QF^\gs_\ZZ$ are $0$-quadratic, and the map to investigate on linear parts in order to apply Proposition~\refthree{proposition:pi-pi} is given by $\SS \to \tau_{\geq 0} \ZZ^{\tC}$, which is 0-connective. We deduce from Proposition~\refthree{proposition:pi-pi} that the map $\L^\uni_n(\SS) \to \L_n^{\gs}(\ZZ)$ is an isomorphism for $n\leq -1$ and a surjection for $n=0$.
\end{example}

\subsection{Surgery for $r$-symmetric structures}
\label{subsection:surgery-symmetric}%

In this section, we prove a comparison result between genuine symmetric and symmetric L-theory. Algebraic surgery for symmetric Poincaré structures is not as straightforward as for the quadratic ones, and we will need to further assume that the base ring is left-coherent of finite left-global dimension. We recall that a ring is called left-coherent, if its finitely presented left modules form an abelian category. From what follows we omit the word \emph{left} from the notation, and stress here that in the present section, commutativity of $R$ is not needed. We note that Noetherian rings are coherent, and that a standard example of coherent but not necessarily Noetherian rings are valuation rings.

Let $R$ be a ring and $M$ an invertible $\ZZ$-module with involution over $R$. For an integer $d\geq0$, we recall that \(R\) has global dimension \(d\) if every \(R\)-module $N$ has a projective resolution of length at most \(d\). 
When \(R\) is in addition coherent, one can find such a resolution where the modules are moreover finitely generated, provided $N$ is finitely presented. 
In this case the connective cover functor \(\tau_{\geq 0}\) and the truncation functor \(\tau_{\leq 0}\) preserve perfect \(R\)-modules, and so \(\Dperf(R)\) inherits from \(\D(R)\) its Postnikov \(t\)-structure, so that \(\Dperf(R)_{\geq 0}\) consists of the \(0\)-connective perfect \(R\)-modules and \(\Dperf(R)_{\leq 0}\) of the \(0\)-truncated perfect \(R\)-modules. This uses that the lowest non-trivial homotopy group of a perfect complex is finitely presented. 
The duality
 \(\Dual\colon \Dperf(R)\op \to \Dperf(R)\) induced by $M$ interacts with the $t$-structure as follows:
\[
\Dual(\Dperf(R)_{\geq 0}) \subseteq \Dperf(R)_{\leq 0}
\quad\text{and}\quad
 \Dual(\Dperf(R)_{\leq 0}) \subseteq \Dperf(R)_{\geq -d}.
\]
The first inclusion is immediate from Remark~\refthree{remark:addendum}, and the second one follows from the Universal Coefficient spectral sequence computing $\mathrm{H}_\ast\map_R(X,M)$, since $\mathrm{Ext}_{R}^i=0$ for every $i\geq d+1$ as
 $R$ has global dimension $\leq d$.
 
We will cast the algebraic surgery argument for symmetric Poincaré structures in the setting of a general Poincaré \(\infty\)-category \((\C,\QF)\) equipped with a \(t\)-structure which interacts with the underlying duality \(\Dual\colon \C\op \to \C\) in the way described above, as we shall use this more general setup for a category of torsion modules, which is not itself given by modules over a discrete ring, in Theorem~\refthree{theorem:devissage} below. Similarly to Definition~\refthree{definition:r-sym} (see also Remark~\refthree{remark:connectivity}), we will say that the Poincaré structure \(\QF\) is \defi{\(r\)-symmetric} if the fibre of \(\QF(X) \to \QF^{\s}_{\Dual}(X)\) is \((-r)\)-truncated for every \(X \in \C_{\geq 0}\), where 
\[
\QF^{\s}_{\Dual}(X)=\map_{\C}(X,\Dual X)^{\hC}
\]
is the symmetric Poincaré structure associated to the duality $\Dual$. Given integers \(a,b \geq -1\) we define \(\Poinc_n^a(\C,\QF), \pM^{a,b}_n(\C,\QF)\) and \(\L^{a,b}_n(\C,\QF)\) as in Definition~\refthree{definition:connective-L}, where we interpret the connectivity requirement on Poincaré objects and Lagrangians as pertaining to the given \(t\)-structure on \(\C\). For a ring coherent ring $R$ of finite global dimension, $\L^{a,b}_n(\Dperf(R);\QF)$ coincides by definition with the earlier defined $\L^{a,b}_n(R;\QF)$. Moreover, as in Remark~\refthree{remark:diagonaladmissible}, $\L^{a,b}_n(\C,\QF)$ is a group if $b\geq a$.

\begin{proposition}[Surgery for \(r\)-symmetric Poincaré structures]
\label{proposition:surgery-global-dim}%
Let $\C$ be a stable $\infty$-category with a bounded  \(t\)-structure \(\C_{\geq 0},\C_{\leq 0}\). Let $\QF$ be an $r$-symmetric Poincaré structure on $\C$ with duality \(\Dual\colon \C \to \C\op\) such that \(\Dual(\C_{\leq 0}) \subseteq \C_{\geq -d}\) for some integer $d\geq 0$.
Fix an \(n \in \ZZ\) and let $a\geq d-1$, $b\geq d$ be integers with $b\geq a$, and such that $n+a$ is even and $a\geq -n+2d-2r$.
Then:
\begin{enumerate}
\item
\label{item:ssurjobject}%
Every Poincaré object in $\Poinc(\Dperf(R),\QF{\qshift{-n}})$ is cobordant to one which is \(\big(\frac{-n-a}{2}\big)\)-connective.
\item
\label{item:ssurjLagrangian}%
Every Lagrangian \(L \to X\) of a \(\big(\tfrac{-n-a}{2}\big)\)-connective Poincaré object \((X,q) \in \Poinc(\Dperf(R),\QF\qshift{-n})\) is cobordant relative to \(X\) to a Lagrangian \(L'\to X\) such that \(L'\) is \(\lceil \frac{-n-1-b}{2} \rceil\)-connective and \(\fib[L'\to X]\) is \(\lfloor\tfrac{-n-1-b}{2}\rfloor\)-connective. 
\end{enumerate}
In particular, the canonical map $\L^{a,b}_n(\C,\QF) \to \L_n(\C,\QF)$ is an isomorphism. 
\end{proposition}

\begin{remark}
\label{remark:recall-intervalsLab-sym}%
Recalling Remark~\refthree{remark:intervalsLab} we note that the Poincaré objects appearing in~\refthreeitem{item:ssurjobject} above are concentrated in degrees \([\frac{-n-a}{2},\frac{-n+a}{2}]\) and the Lagrangians in~\refthreeitem{item:ssurjLagrangian} are concentrated in degrees \([\lceil \frac{-n-1-b}{2}\rceil , \lceil \frac{-n-1+b}{2} \rceil]\).
\end{remark}

Before giving a proof of Proposition~\refthree{proposition:surgery-global-dim}, we establish some of its consequences and
start by specialising Proposition~\refthree{proposition:surgery-global-dim} to the case \(\C = \Dperf(R)\):
\begin{corollary}
\label{corollary:surgery-global-dim-for-rings}%
Let $M$ be an invertible $\ZZ$-module with involution over $R$ and suppose that \(R\) is coherent of finite global dimension \(d\). Let \(\QF\) be an $r$-symmetric compatible Poincaré structure on \(\Dperf(R)\), for \(r \in \ZZ\).
Then for \(n \geq d-2r\) the following holds: 
\begin{enumerate}
\item 
If \(n+d\) is even, the canonical map
\[
\L^{d,d}_n(R;\QF) \lto \L_n(R;\QF)
\]
is an isomorphism.
\item
If \(n+d\) is odd, the canonical map
\[
\L^{d-1,d}_n(R;\QF) \lto \L_n(R;\QF)
\]
is an isomorphism.
\end{enumerate}
\end{corollary}

\begin{corollary}
\label{corollary:d-zero}%
Let $M$ be an invertible $\ZZ$-module with involution over $R$ and assume that  \(R\) is coherent of finite global dimension \(d\). Let \(\QF\) be an $r$-symmetric compatible Poincaré structure on \(\Dperf(R)\), for \(r \in \ZZ\). Then the following holds:
\begin{enumerate}
\item
\label{item:cor-one}%
If $d= 0$, then \(\L_{2k-1}(R;\QF)=0\) whenever \(k \geq 1-r\);
\item
\label{item:cor-two}%
If $d\leq1$, then \(\L_{2k}(R;\QF) \cong \W(\Proj(R);\Omega^{2k}\QF(\Omega^k-))\) whenever \(k \geq 1-r\). 
For $\QF=\Qgen{m}{M}$, the latter is isomorphic to the classical Witt group of $(-1)^k$-quadratic forms for $k=m-2$, $(-1)^k$-even forms for $k=m-1$, and $(-1)^k$-symmetric forms for all $k\geq m$.
\end{enumerate}
\end{corollary}
\begin{proof}
For part~\refthreeitem{item:cor-one}, we apply Proposition~\refthree{proposition:surgery-global-dim} in case $d=0$ with $(a,b) = (-1,0)$. For part~\refthreeitem{item:cor-two}, we apply Corollary~\refthree{corollary:surgery-global-dim-for-rings} in case $d=1$, combined with Proposition~\refthree{proposition:00-01} and obtain isomorphisms
\[
\L_{2k}(R;\QF) \cong \L_{2k}^{0,1}(R;\QF) \cong \L_{2k}^{0,0}(R;\QF).
\]
Then, we observe that there is a canonical isomorphism $\W(\Proj(R);\Omega^{2k}\QF(\Omega^k -)) \cong \L_{2k}^{0,0}(R;\QF)$ induced by sending an element represented by $(P,q)$ to the element represented by $(P[-k],q)$.
The remaining claim follows by inspection of $\pi_{2k}(\Qgen{m}{M}(P[-k])) \cong \pi_0(\Qgen{m-k}{(-1)^kM}(P))$.
\end{proof}

\begin{remark}
\label{remark:odd-vanishing}%
We notice that a ring $R$ is of global dimension 0 if and only if it is semisimple, and that semisimple rings are Noetherian and thus coherent. Part~\refthreeitem{item:cor-one} above hence recovers Ranicki's result that the odd-dimensional symmetric and quadratic L-groups of semisimple rings vanish \cite{Ranickiblue}*{Proposition 22.7}: Indeed, the symmetric case follows from the above since $\QF^\s$ is $r$-symmetric for every $r$, and hence $\L_{2k-1}(R;\QF^\s)=0$ for all $k$. For the quadratic case, by Corollary~\refthree{corollary:periodicity-L} applied to $\QF^{\qdr}_{M}$, it suffices to show that $\L^{\qdr}_{-3}(R;M) = 0$. But by Corollary~\refthree{corollary:genuine-is-quadratic} we have that $\L^{\qdr}_{-3}(R;M) \cong \L^{\gs}_{-3}(R;M) = \L_{-3}(R;\QF^{\gs}_{M})$. Now, $\QF^\gs$ is $2$-symmetric, so \refthreeitem{item:cor-one} applies for $k=-1$.

For completeness, we note that if $K$ is a field of characteristic different from $2$, also $\L^\qdr_{4k+2}(K) \cong \L^\s_{4k+2}(K)$ vanishes: By Corollary~\refthree{corollary:d-zero} it is given by the Witt group of anti-symmetric forms over $K$, but any such form admits a symplectic basis and hence a Lagrangian.
\end{remark}

\begin{proof} [Proof of \refthree{proposition:surgery-global-dim}]
For part~\refthreeitem{item:ssurjobject}, it suffices to show that every Poincaré object \((X,q) \in \Poinc(\C,\QF\qshift{-n})\) is cobordant to one whose underlying object is \((\frac{-n-a}{2})\)-connective. Let $k = \tfrac{-n-a-2}{2}$, define $W :=  \Om^n\Dual \tau_{\leq k}X$ and let
\[
f\colon W \lto \Om^n\Dual X \st{q}{\simeq} X
\]
be the map dual to the truncation map \(X \to \tau_{\leq k}X\). Since \(\Dual(\C_{\leq 0}) \subseteq \C_{\geq -d}\) we have that \(W\) is \((-n-k-d)\)-connective. Since \(\Dual W \simeq \Sig^n\tau_{\leq k}X\) is \((n+k)\)-truncated we conclude that \(\Om^n\map_{\C}(W,\Dual W)^{\hC}\) is \((n+2k+d)\)-truncated. On the other hand, since $\QF$ is $r$-symmetric, the fibre of the map 
\begin{equation*}
\Om^n\QF(W) \lto \Om^n\QF^{\s}_{\Dual}(W)=\Om^n\map_{\C}(W,\Dual W)^{\hC}
\end{equation*} 
is \((k+d-r)\)-truncated, so that \(\Om^{n}\QF(W)\) is \(\max(n+2k+d,k+d-r)\)-truncated. Spelling out the definition of $k$ and using the estimates in the assumptions, we find that
\[
\max(n+2k+d,k+d-r) < 0.
\]
We hence get that \(\Om^n\QF(W)\) is \((-1)\)-truncated and so \(\Om^{\infty+n}\QF(W) \simeq \ast\). The restriction of \(q\) to \(W\) is consequently null-homotopic, and we may therefore perform surgery along \(f\colon W \to X\) to obtain a new Poincaré object \((X',q')\), given by the cofibre of the resulting map \(W \to \tau_{\geq k+1}X\) (see diagram \eqrefthree{equation:surgery-diagram}). Since \(W\) is \((-n-k-d)\)-connective it is in particular \((k+1)\)-connective (since \(-2k = n+a+2 \geq n+d+1\)), and so \(X'\) is \((k+1)\)-connective. Since $k+1 = \frac{-n-a}{2}$, part \refthreeitem{item:ssurjobject} is shown.

To prove part \refthreeitem{item:ssurjLagrangian}, let \((X,q)\) be a Poincaré object such that \(X\) is \((\frac{-n-a}{2})\)-connective, and let  \((L \to X,q,\eta)\) be a Lagrangian.
Let \(N := \fib(L \to X)\), 
so that \(L \simeq \Om^{n+1}\Dual N\). If \(L\) is \(\lceil \frac{-n-1-b}{2} \rceil\)-connective then \(N\) is \(\lfloor \frac{-n-1-b}{2} \rfloor\)-connective (since \(X\) is \((\frac{-n-a}{2})\)-connective and \(b \geq a\)) and we are done. Otherwise, 
let $l = \lceil \frac{-n-1-b}{2} \rceil -1$, define \(N' := \Om^{n+1}\Dual \tau_{\leq l}L\) and let 
\[
f\colon N' \lto N
\]
be the map dual to the truncation map \(L \to \tau_{\leq l}L\). We may view this map as a map $(N' \to 0) \to (L \to X)$ in the metabolic category, and we claim that it extends to a Lagrangian surgery datum for which it suffices to show that $\QF_\met(N' \to 0) \simeq \Omega \QF(N')$ is $(n-1)$-truncated.
Since \(\Dual(\C_{\leq 0}) \subseteq \C_{\geq -d}\) we have that \(N'\) is \((-n-1-l-d)\)-
connective. Since \(\Dual N' \simeq \Sig^{n+1}\tau_{\leq l}L\) is \((n+1+l)\)-truncated we have that \(\Om^{n+1}\map_{\C}(N',\Dual N')^{\hC}\) is \((n+1+2l+d)\)-truncated. On the other hand since $\QF$ is $r$-symmetric the fibre of the map 
\begin{equation*}
\Om^{n+1}\QF(N') \lto \Om^{n+1}\map_{\C}(N',\Dual N')^{\hC}
\end{equation*} 
is \((l+d-r)\)-truncated, so that \(\Om^{n+1}\QF(N')\) is \(\max(n+1+2l+d,l+d-r)\)-truncated. Now by definition of $l$ and the estimates in the assumptions we have that 
\[
\max(n+1+2l+d,l+d-r) <  0
\]
We hence get that
\(\Om^{n+1}\QF(N')\) is \((-1)\)-truncated as needed.
We may therefore perform Lagrangian surgery along \(N' \to N\) to obtain a new Lagrangian \(L'\to X\), such that \(L'\) is given by 
the cofibre of the resulting map \(N' \to \tau_{\geq l+1}L\). Since \(N'\) is \((-n-1-l-d)\)-connective it is in particular \((l+1)\)-connective (since \(-2l \geq n+b+2 \geq n+d+2\)), and so \(L'\) is \((l+1)\)-connective. 
This proves part \refthreeitem{item:ssurjLagrangian}.

Finally, the bijectivity of the map $\L^{a,b}_n(\C,\QF) \to \L_n(\C,\QF)$ follows just as in the proof of Proposition~\refthree{proposition:surgery-quadratic}.

\end{proof}

\begin{remark}
\label{remark:cofibre-is-fibre-two}%
Similarly to Remark~\refthree{remark:cofibre-is-fibre}, 
Proposition~\refthree{proposition:surgery-global-dim} allows us to conclude that the sequence 
\[
\pi_0\pM_{n}^{a,b}(\C,\QF) \lto \pi_0\Poinc_{n}^{a}(\C,\QF) \lto \L^{a,b}_n(\C,\QF)
\]
is exact in the middle, when $a,b$ and $n$ satisfy the assumptions of Proposition~\refthree{proposition:surgery-global-dim}. In particular for $a=b=n=0$ we find that, for every coherent ring $R$ of global dimension $0$,
every symmetric, even or quadratic $M$-valued Poincaré object in \(\Proj(R)\) which is zero in the Witt group is strictly metabolic.
\end{remark}

We now use Proposition~\refthree{proposition:surgery-global-dim} in the case \(d=0\) to describe the symmetric L-groups of stable $\infty$-categories with a $t$-structure in terms of Witt groups associated to the heart of the $t$-structure.
So let \(\C\) be a stable \(\infty\)-category with a \(t\)-structure \((\C_{\geq 0},\C_{\leq 0})\) and a duality \(\Dual\colon\C \to \C^{{\op}}\) such that \(\Dual\) sends \(\C_{\leq 0}\) to \(\C_{\geq 0}\) and vice versa. In particular, \(\Dual\) induces a duality \(\Dual^{\heartsuit}\colon \C^{\heartsuit} \to \C^{\heartsuit}\) on the heart of \(\C\), and we may consider \(\C^{\heartsuit}\) an abelian category with duality. 
As in the previous section, we let $\QF_{\Dual^{\heartsuit}}^{\g\s}\colon \C^{\heartsuit}\to \Ab$ be the quadratic functor that takes $A\in  \C^{\heartsuit}$ to the abelian subgroup of strict invariants
\[
\QF_{\Dual^{\heartsuit}}^{\g\s}(A):=\hom_{\C^{\heartsuit}}(A,\Dual^{\heartsuit}\!A)^{\Ct}.
\]
We let $\W(\C^{\heartsuit},\QF_{\Dual^{\heartsuit}}^{\g\s})$ be the corresponding symmetric Witt group, defined as the quotient of the monoid \(\pi_0\Poinc(\C^{\heartsuit},\QF_{\Dual^{\heartsuit}}^{\g\s})\) by the submonoid of strictly metabolic objects. Thus two elements of $\pi_0\Poinc(\C^{\heartsuit},\QF_{\Dual^{\heartsuit}}^{\g\s})$ are identified in the Witt group if they are isomorphic after adding strictly metabolic objects. 
By letting \(\QF^{\s}_{\Dual}\) be the symmetric Poincaré structure associated to \(\Dual\), we observe that since $\pi_0 \QF^{\s}_{\Dual}(A)=\QF_{\Dual^{\heartsuit}}^{\g\s}(A)$ for every object $A$ of $\C^{\heartsuit}$, there is a group isomorphism 
\[
\W(\C^{\heartsuit},\QF_{\Dual^{\heartsuit}}^{\g\s})\cong \L^{0,0}_0(\C,\QF^{\s}_{\Dual})
\]
analogous to the isomorphism for the Witt group of Remark~\refthree{remark:Witt-group}.

\begin{example}
If $R$ has global dimension $0$, the heart of \(\C = \Dperf(R)\) agrees with $\Proj(R)$ and by definition
\[
\W(\Dperf(R)^{\heartsuit},\QF_{\Dual_M^{\heartsuit}}^{\g\s})=\W(\Proj(R);\QF_M^{\g\s})
\]
where the latter is the classical Witt group of symmetric forms as defined in Remark~\refthree{remark:Witt-group}.

Another example is given by certain torsion-modules over a 1-dimensional ring, as we elaborate on and exploit in Theorem~\refthree{theorem:devissage}.
\end{example}

Let us also write $-\Dual^{\heartsuit}$ for the duality on $\C^{\heartsuit}$ defined by the functor $\Dual^{\heartsuit}$ but where we replace the isomorphism $\eta\colon \id\to (\Dual^{\heartsuit})\op\Dual^{\heartsuit}$ with $-\eta$.

\begin{corollary}
\label{corollary:dim-zero}%
In the above situation, let \(\QF^{\s}_{\Dual}\colon \C \to \Spa\) be the symmetric Poincaré structure associated to \(\Dual\). Then there are canonical isomorphisms
\[
\L_n(\C;\QF^{\s}_{\Dual}) 
\cong 
\begin{cases} 
\W(\C^{\heartsuit},\QF_{\Dual^{\heartsuit}}^{\g\s}) & \text{ for } n \equiv 0 \text{ mod } 4, \\
\W(\C^{\heartsuit},\QF_{-\Dual^{\heartsuit}}^{\g\s}) & \text{ for } n \equiv 2 \text{ mod } 4, \\
0 & \text{ else.}
\end{cases}
\]
In particular, every element of $\pi_0\Poinc(\C^{\heartsuit},\QF_{\Dual^{\heartsuit}}^{\g\s})$ which is zero in \(\W(\C^{\heartsuit},\QF_{\Dual^{\heartsuit}}^{\g\s})\) is metabolic.
\end{corollary}
\begin{proof}
Apply Proposition~\refthree{proposition:surgery-global-dim} in the case of \(r=\infty, d=0\) and take \((a,b)\) to be \((0,0)\) when \(n\) is even and \((-1,0)\) when \(n\) is odd. 
\end{proof}

Our next goal is to use the surgery results above in order to identify the \(\L\)-groups of an \(r\)-symmetric structure with the corresponding symmetric \(\L\)-groups in a suitable range. 
The following corollary should be compared with Corollary~\refthree{corollary:genuine-is-quadratic} above:

\begin{corollary}
\label{corollary:genuine-is-symmetric}%
Let $M$ be an invertible $\ZZ$-module with involution over $R$ and suppose that \(R\) is coherent  of finite global dimension \(d\). Let \(\QF\) be an $r$-symmetric Poincaré structure on \(\Dperf(R)\) compatible with $M$, for \(r \in \ZZ\). Then the canonical map
\[
\L_n(R;\QF) \lto \L_n(R;\QF^{\s}_M)=\L^{\s}_n(R;M)
\]
is injective for \(n \geq d-2r+2\) and bijective for \(n \geq d-2r+3\).
\end{corollary}
\begin{proof}
We consider the commutative diagram
\[
\begin{tikzcd}
\L_n^{d,d}(R;\QF) \ar[r] \ar[d] & \L_n(R;\QF) \ar[d] & \L_n^{d-1,d}(R;\QF) \ar[l] \ar[d] \\
\L_n^{d,d}(R;\QF^{\s}_M) \ar[r] & \L_n(R;\QF^{\s}_M) & \L_n^{d-1,d}(R;\QF^{\s}_M) \ar[l]
\end{tikzcd}
\]
and use Corollary~\refthree{corollary:surgery-global-dim-for-rings} and Lemma~\refthree{lemma:genuine-to-symmetric} to conclude the bijectivity claim of the corollary. To see injectivity for $n=d-2r+2$, again using Corollary~\refthree{corollary:surgery-global-dim-for-rings}, it will suffice to show that the left vertical map in the above diagram is injective.
In light of Remark~\refthree{remark:cofibre-is-fibre-two} it will suffice to show that if \((X,q)\) is a $(-d+r-1)$-connective Poincaré object in \((\Dperf(R),\QF\qshift{-n})\) whose associated Poincaré object in \((\Dperf(R),(\QF^{\s}_M)\qshift{-n})\) admits a Lagrangian \(L \to X\) such that \(L\) is $(-d+r-1)$-connective then \(L\) can be refined to a Lagrangian of \((X,q)\) with respect to \(\QF\). For this, it will suffice to show that for an \(L\) with this connectivity bound, the map
\[
\Om^n\QF(L) \lto \Om^n\QF^{\s}_M(L)
\]
is surjective on \(\pi_1\) and injective on \(\pi_0\). Indeed, this 
map is $(-1)$-truncated by Remark~\refthree{remark:connectivity} since \(\QF\) is \(r\)-symmetric.
\end{proof}

As a consequence, we obtain the following result, which proves Theorem~\refthree{theorem:gs-and-s-agreement-range} and the first part of Theorem~\refthree{theorem:GW-classical-quad-sym} from the introduction. 
\begin{corollary}
\label{corollary:improve}%
Let $M$ be an invertible $\ZZ$-module with involution over $R$ and suppose that \(R\) is coherent  of finite global dimension \(d\). Then the canonical maps 
\[
\L^{\gs}_n(R;M) \lto \L^{\s}_n(R;M) \quad \text{ and } \quad \GW^\gs_n(R;M) \lto \GW^\s_n(R;M)
\]
are injective for $n \geq d-2$ and bijective for $n\geq d-1$. Likewise, the maps
\[
\L^{\gq}_n(R;M) \lto \L^{\s}_n(R;M) \quad \text{ and } \quad \GW^\gq_n(R;M) \lto \GW^\s_n(R;M)
\]
are injective for $n\geq d+2$ and bijective for $n\geq d+3$.
\end{corollary}
\begin{proof}
By Theorem~\refthree{theorem:fiber-sequence-intro-three}, the canonical squares
\[
\begin{tikzcd}
\GW^\gq(R;M) \ar[r] \ar[d] & \GW^\gs(R;M) \ar[d] \ar[r] & \GW^\s(R;M) \ar[d] \\
\L^\gq(R;M) \ar[r] & \L^\gs(R;M) \ar[r] & \L^\s(R;M)
\end{tikzcd}
\]
are pullbacks. Hence the Grothendieck-Witt part of the corollary follows from the L-theory part. Moreover, the L-theory parts then follow from Corollary~\refthree{corollary:genuine-is-symmetric} using that $\QF^\gs$ is $2$-symmetric and $\QF^\gq$ is $0$-symmetric.
\end{proof}

\begin{example}
\label{example:dimension-bound-sharp}%
In general, the bounds obtained in Corollary~\refthree{corollary:improve} are sharp, as the following example shows. Consider the $d$-dimensional ring $\FF_2[\ZZ^d]$ as a ring with anti-involution induced by the inversion of the group $\ZZ^d$. We claim that the map $\L^\gs(\FF_2[\ZZ^d]) \to \L^\s(\FF_2[\ZZ^d])$ is not surjective on $\pi_{d-2}$. To see this we use the Ranicki-Shaneson splitting proved by Ranicki for $\L^\s$ and Milgram-Ranicki for $\L^\gs$. In \paperfour we give a proof of this result which works simultaneously for both variants, but for our purposes the following version is sufficient. Let $R$ be a ring with anti-involution and consider the ring $R[\ZZ]$ with anti-involution induced by the group inversion of $\ZZ$. Suppose that $\K_0(R) \cong \K_0(R[\ZZ]) \cong \ZZ$, for instance $R$ could be a field or the integers. Then, for $\QF = \QF^\gs,\QF^\s$, there is a natural equivalence
\[
\L(R[\ZZ];\QF) \simeq \L(R;\QF) \oplus \Sigma \L(R;\QF).
\]
By induction, we deduce that the map $\Sigma^d \L^\gs(\FF_2) \to \Sigma^d \L^\s(\FF_2)$ is a retract of the map $\L^\gs(\FF_2[\ZZ^d]) \to \L^\s(\FF_2[\ZZ^d])$. Therefore, in order to see that the latter map is not surjective on $\pi_{d-2}$ it suffices to argue that the map $\L^\gs(\FF_2) \to \L^\s(\FF_2)$ is not surjective on $\pi_{-2}$. Since $\L^\s_2(\FF_2) \cong \ZZ/2$ but $\L^\gs(\FF_2) = 0$, this is indeed the case. A similar argument shows that for $\eps=-1$, the map $\L_0^\gs(\ZZ[\ZZ];\eps) \to \L_0^\s(\ZZ[\ZZ];\eps)$ is not surjective on $\pi_0$.

We note that for $d=0$, this shows that the obtained bounds are sharp also for commutative rings viewed as rings with trivial anti-involution. Moreover, by \cite{Read}*{Corollary F}, the map $\L^\gs(\ZZ/4) \to \L^\s(\ZZ/4)$ is not an isomorphism on almost all positive homotopy groups, showing that also among general commutative rings with trivial involution, our results cannot be improved too much. However, Schlichting has recently shown that the map $\L^\gs(R) \to \L^\s(R)$ is an isomorphism on $\pi_n$ for $n\geq -1$ for many regular Noetherian domains $R$ of arbitrary finite Krull dimension \cite{Schlichting-genuine}, showing that among regular Noetherian rings, our bounds are quite far from being optimal.
\end{example}

\begin{remark}
\label{remark:right-coherent-variant}%
If $R$ is a right-coherent ring of finite right-global dimension, we may apply the results of Corollaries~\refthree{corollary:d-zero}, \refthree{corollary:genuine-is-symmetric}, and \refthree{corollary:improve} and Remark \refthree{remark:odd-vanishing} to the ring $R\op$ with Poincaré structure $\QF^\vee$ as described in Remark~\refthree{remark:passing-to-opposite-ring}: By Example~\refthree{example:opposite-poincare-structure} $\QF^\vee$ is $r$-symmetric if $\QF$ is. Using then the equivalence of Poincaré $\infty$-categories $(\Dperf(R\op),\QF^\vee) \simeq (\Dperf(R),\QF)$, we obtain the conclusions of Corollaries~\refthree{corollary:d-zero}, \refthree{corollary:genuine-is-symmetric}, and \refthree{corollary:improve} and Remark~\refthree{remark:odd-vanishing} also for right-coherent rings of finite right-global dimension.
\end{remark}

\begin{remark}
In Theorem~\refthree{theorem:main-theorem-L-theory} we have shown that the non-negative genuine symmetric \(\L\)-groups coincide with Ranicki's \(\L\)-groups of short complexes. The comparison range above then improves on Ranicki's classical theorem that established injectivity of the map $\L_n^{\gs}(R) \to \L^\s(R)$ for non-negative \(n \geq 2d-3\) and bijectivity for non-negative \(n \geq 2d-2\) for Noetherian rings of global dimension $d$.
\end{remark}

Since Dedekind rings have global dimension \(\leq 1\), and by applying the fibre sequence of Theorem~\refthree{theorem:fiber-sequence-intro-three} we immediately find:

\begin{corollary}
\label{corollary:dedekind}%
Let \(R\) be a Dedekind ring, e.g.\ the ring of integers in an algebraic number field. 
Then the canonical maps
\[
\L^{\gs}_n(R;M)\stackrel{}{\lto} \L^{\s}_n(R;M) \quad\text{and}\quad \GW^{\g\s}_n(R;M)\stackrel{}{\lto} \GW^{\s}_n(R;M)
\]
are injective for \(n=-1\) and bijective for \(n\geq 0\). In particular, the non-negative homotopy groups of $\L^{\gs}_n(R;M)$ are 4-periodic. Similarly, the maps
\[
\L^{\g\qdr}_n(R;M)\stackrel{}{\lto} \L^{\s}_n(R;M) \quad\text{and}\quad \GW^{\g\qdr}_n(R;M)\stackrel{}{\lto} \GW^{\s}_n(R;M)
\]
are injective for \(n=3\) and bijective for \(n\geq 4\).
\end{corollary}

We recall that by the main Theorem of \cite{comparison} the non-negative homotopy groups of
$\GW^{\g\s}_n(R;M)$ and $ \GW^{\g\qdr}_n(R;M)$ are isomorphic to the classical Grothendieck-Witt groups of symmetric and quadratic forms, respectively. Further, we find:

\begin{corollary}
\label{corollary:dedekind-gw}%
Let \(R\) be a Dedekind ring and $M$ an invertible $\ZZ$-module with involution over $R$. 
Then the canonical maps
\[
\GW_{\cl}^{\s}(R;M) \lto \tau_{\geq0}\GW^{\s}(R;M)\quad\text{and}\quad \tau_{\geq 4}\GW_{\cl}^{\qdr}(R;M) \lto \tau_{\geq4}\GW^{\s}(R;M)
\]
are equivalences.
\end{corollary}

\begin{remark}
\label{remark:improve-two}%
We now prove the second part of Theorem~\refthree{theorem:GW-classical-quad-sym} from the introduction. So let $R$ be a coherent ring of finite global dimension d. We have equivalences $\Sigma^2\L^\gs(R) \simeq \L^{-\gev}(R)$, and $\Sigma^2 \L^{\gev}(R) \simeq \L^{-\gq}(R)$. If $R$ is in addition 2-torsion free, for instance a Dedekind domain whose fraction field is of characteristic different from 2, then the canonical maps $\QF^{-\gev} \to \QF^{-\gs}$ and $\QF^{\gq} \to \QF^{\gev}$ are equivalences by Remark~\refthree{remark:genuineeq}. We deduce that for such rings, there are in fact canonical equivalences
\[
\Sigma^2\L^{\gs}(R) \simeq \L^{-\gs}(R) \quad \text{ and } \quad \Sigma^2 \L^{\gq}(R) \simeq \L^{-\gq}(R).
\]
The comparison map is compatible with these equivalences, so  we deduce that the map $\L^{\gq}(R) \to \L^{\gs}(R)$ is a 2-fold loop of the map $\L^{-\gq}(R) \to \L^{-\gs}(R)$. Corollary~\refthree{corollary:improve} then implies that the maps $\L^{\gq}_n(R) \to \L^\gs_n(R)$ and $\GW^\qdr_{\cl,n}(R) \to \GW^\s_{\cl,n}(R)$ are injective for $n=d$ and an isomorphism for $n\geq d+1$.
\end{remark}

\begin{remark}
\label{remark:form-paramters-and-symmetric-surgery}%
As described in Definition~\refone{definition:form-parameter}, there is a canonical non-abelian derived Poincaré structure $\QF^{\gfpm}_M$ associated to any form parameter $\lambda$, as described in the work of Baues and Schlichting \cite{Baues, SchlichtinghigherI} (extending the classical notion of Bak) on an invertible $\ZZ$-module with involution $M$. It then follows from the fact that $\QF^{\gfpm}_M(P[0])$ is discrete (by definition) that $\QF^{\gfpm}_M$ is 0-quadratic and 0-symmetric. In particular, the comparison results of Corollary~\refthree{corollary:genuine-is-quadratic} and Corollary~\refthree{corollary:genuine-is-symmetric} apply to $\QF^{\gfpm}_M$. Depending on $\lambda$, the Poincaré structure $\QF^{\gfpm}_M$ might in fact be 1-symmetric (as is the case for even forms) or 2-symmetric (as is the case for symmetric forms), or likewise $1$-quadratic (as in the case of even forms) or 2-quadratic (as in the case of quadratic forms). 
Hence, for a form parameter $\lambda = (M,Q)$ over a 1-dimensional ring $R$, we find that the map $\L^{\gfpm}_n(R) \to \L^\s_n(R)$ is injective for $n=3$ and bijective for $n\geq 4$. In addition, $\L_2^\gfpm(R)$ is isomorphic to the Witt group associated to the form parameter $(-M, \ker(M \to Q))$ as follows from Corollary~\refthree{corollary:d-zero}.
\end{remark}

\begin{example}
\label{example:dotto-ogle}%
In this example, we show how the surgery methods developed in the previous sections allow to determine the L-theory $\L^\burn(\ZZ) = \L(\Dperf(\ZZ),\QF^\burn)$ investigated first in \cite{Dotto-Ogle} and denoted $\L^\g(\mathbb{A})$ therein. The Poincaré structure $\QF^\burn$ is defined as the pullback
\[
\begin{tikzcd}
  \QF^\burn(X) \ar[r] \ar[d] & \map_\ZZ(X,\tau_{\geq 1/2}\ZZ^{\tC}) \ar[d] \\
  \QF^\gs(X) \ar[r] & \map_\ZZ(X,\tau_{\geq 0}\ZZ^{\tC})
\end{tikzcd}
\]
where $\tau_{\geq 1/2}\ZZ^{\tC}$ is the pullback $\ZZ \times_{\ZZ/2} \tau_{\geq 0} \ZZ^{\tC}$. We deduce that $\QF^\burn$ is the Poincaré structure associated to the Burnside ring form parameter which is given by $Q= \ZZ \times_{\ZZ/2} \ZZ$ and $M=\ZZ$, and where the map $M \to Q$ is the pair $(2,0)$. Since the map  $M \to Q$ is injective, we deduce from Remark~\refthree{remark:form-paramters-and-symmetric-surgery} that the map $\L^\burn_n(\ZZ) \to \L^\s_n(\ZZ)$ is an isomorphism for $n\geq 2$. Here we have used that $\L_3^\s(\ZZ) = 0$. Likewise, we may apply Proposition~\refthree{proposition:pi-pi} to the canonical Poincaré functor $(\Mod_\SS^\omega,\QF^\uni) \to (\Dperf(\ZZ),\QF^\burn)$: 
In this case, the map on linear terms to investigate is the map $\SS \to \tau_{\geq 1/2} \ZZ^{\tC}$ which is $1$-connective. We deduce that the map $\L^\uni_n(\SS) \to \L^\burn_n(\ZZ)$ is an isomorphism for $n\leq 0$ and a surjection for $n=1$. We will now show that the map $\L_1^\qdr(\ZZ) \to \L_1^\burn(\ZZ)$ is surjective, showing that $\L_1^\burn(\ZZ)$ vanishes. To see this, one may first  assume that an element of $\L_1^\burn(\ZZ)$ is represented by
$(X,q)$ with $X$ a complex concentrated in degrees $[-1,0]$. 
In this case, $\pi_0(X)$ is a finitely generated free module, and one can perform surgery along the canonical map $\pi_0(X) \to X$. The result is a Poincaré object $(X',q')$ for $\Omega\QF^\burn$ with homotopy groups concentrated in degree $-1$, and there necessarily a torsion group $T$. On such objects, the form $q'$ then lifts to a form for $\Omega\QF^\qdr$, since $\Omega\map_\ZZ(T[-1],\tau_{\geq 1/2} \ZZ^{\tC})$ has trivial $\pi_0$.
\end{example}

\section{L-theory of Dedekind rings}
\label{section:localisation-devissage}%
The goal of this section is to extend Quillen's localisation-dévissage sequence \cite{quillen}*{Corollary of Theorem 5} for Dedekind rings to hermitian K-theory, thereby proving Theorem~\refthree{theorem:localisation-devissage-Dedekind} of the introduction. We recall that Dedekind rings are commutative regular Noetherian domains of global dimension $1$, and the main example of interest to us are those whose fraction field is a number field. Other notable examples are discrete valuation rings and rings of functions of smooth affine curves over fields.

To prove the theorem we first construct a Poincaré-Verdier sequence induced by the map from a Dedekind ring $R$ to  its localisation \emph{away from a set $S$ of non-zero prime ideals}. The functor $\GW$ takes such sequences to cofibre sequences of spectra. Using a dévissage result for symmetric $\GW$-theory we then identify the fibre term as the sum of $\GW$-spectra of the residue fields  $R/\fp$, where $\fp$ ranges through the set $S$.
 
Results of this type have appeared in the literature from early on. Three- or four-term localisation sequences for Witt groups appear in the work of Knebusch~\cite{knebusch}, Milnor-Husemoller \cite{milnor-symmetric} and these are extended to long exact sequences of L-groups by Ranicki ~\cite{Ranickiyellowbook}*{\S 4.2} and Balmer-Witt groups in~\cite{balmer-witt}*{\S 1.5.2}. For Grothendieck-Witt groups there are exact sequences due to Karoubi~\cite{karoubi-localisation-I,karoubi-localisation-II} as well as Hornbostel-Schlichting~\cite{hornbostel,hornbostel-schlichting}, both under the assumption that 2 is invertible in the ring. 

Our results hold with no assumption on invertibility of $2$, for the homotopy theoretic symmetric Grothendieck-Witt spectrum $\GW^{\s}$. By the results of the previous section (see Corollary~\refthree{corollary:dedekind} and~\refthree{corollary:dedekind-gw}) the non-negative homotopy groups of $\GW^\s(R)$ agree with the classical higher Grothendieck-Witt groups, but it is only $\GW^\s(R)$ which is well-behaved in all degrees, see Remark~\refthree{remark:no-quadratic-devissage}.

\subsection{The localisation sequence}
\label{subsection:locdev}%

Let $R$ be a Dedekind ring and  $S$ a set of non-zero prime ideals of $R$. 
We let $R_S = \nO(U)$ where $\nO$ is the structure sheaf of $\spec(R)$ and $U = \spec(R)\setminus S$ is the complement of the set $S$ (in case $S$ is infinite $U$ is not open in $ \spec(R)$, and $ \nO(U)$ is defined as the colimit of $\nO$ on the open subsets of $\spec(R)$ containing $U$). Concretely, one can describe the ring $R_S$ as follows: For any non-zero prime ideal $\fp$, the localisation $R_{(\fp)}$ at $\fp$ is a discrete valuation ring, and the fraction field $K$ of $R$ hence acquires a $\fp$-adic valuation $\nu_\fp$. Then $R_S$ identifies with the subring of $K$ given by all elements $x \in K$ such that $\nu_{\fp}(x) \geq 0$ for all $\fp$ not contained in $S$. The ring $R_S$ can be thought of as the localisation of $R$ \emph{away} from the set of primes $S$.
We denote by $\Dperf_S(R) \subseteq \Dperf(R)$ the fibre of the functor $\Dperf(R) \to \Dperf(R_S)$. By the above, it coincides with the full subcategory spanned by those perfect \(R\)-modules whose homotopy groups are \(S\)-primary torsion modules, which we will also refer to as \(S_{\infty}\)-torsion modules.

\begin{example}
Let $R$ be a Dedekind ring.
\begin{enumerate}
\item If $S$ consists of all the non-zero prime ideals of $R$, then $R_S$ is the fraction field $K$,
\item Given a multiplicative subset $T\subset R$ we may consider the set of primes ideals $S=\{\fp\ |\ \fp\cap T\neq \emptyset\}$. Then $R_S=R[T^{-1}]$ is obtained from $R$ by inverting the elements of $T$. In particular:
\item If $S = \{ \fp_1,\dots,\fp_n \}$ is such that the ideal $\fp_1^{r_1} \cdot \dots \cdot \fp_n^{r_n} = (x)$ is a principal ideal, then $R_S = R[\tfrac{1}{x}]$.
\end{enumerate}
These are the cases that we will use most. We warn the reader that, in general, $R_S$ is not obtained from $R$ by inverting a multiplicative subset.
\end{example}

Recall from Definition~\reftwo{definition:solid} that a map of rings $\vphi\colon A \to B$ is called a derived localisation if the map
$B\otimes_{A} B \to B$ from the derived tensor product is an equivalence in \(\Der(A)\).
When \(B\) is flat over \(A\), this is equivalent to saying that the multiplication map \(B \otimes^\mathrm{U}_A B \to B\) from the underived tensor product is an isomorphism of ordinary \(A\)-modules.
Recall as well that \(\vphi\) is said to have a perfectly generated fibre if \(I := \fib[A \to B] \in \Der(A)\) can be written as a filtered colimit of perfect \(B\)-torsion complexes, that is, objects in the fibre of \(\Dperf(A) \to \Dperf(B)\).

\begin{lemma}
\label{lemma:derived-localisation}%
Let $R$ be a Dedekind ring and $S$ a set of non-zero prime ideals of $R$. Then \(R_S\) is flat over \(R\) and the map $R \to R_S$ is a derived localisation with a perfectly generated fibre.
\end{lemma}
\begin{proof}
To see that \(R_S\) is flat observe that it can be written as a filtered colimit of projective modules
\[
R_S = \colim_{D \in \Div_S} R_D.
\]
Here \(\Div_S\) is the monoid of effective divisors supported on \(S\), that is, formal sums \(D = \sum_{\fp} a_{\fp} \fp\) with \(a_{\fp} \in \NN\) such that \(a_{\fp} = 0\) for all but finitely many ideals which are contained in \(S\), 
and where for such an effective divisor \(D = \sum a_{\fp} \fp\) we write 
\[
R_{D} = \{x \in K \;|\; \val_{\fp}(x) \geq -a_{\fp}\;\forall \fp\} \subseteq K
\]
for the associated fractional ideal. Since $R_D$ is projective, this shows that \(R_S\) is flat. In addition, since \(\Div_S\) is filtered and the map \(R_{D} \otimes^\mathrm{U}_R R_{D'} \to R_{D + D'}\) induced by the multiplication in \(K\) is an isomorphism, we also deduce from this description that the multiplication map \(R_S \otimes^\mathrm{U}_R R_S \to R_S\) is an isomorphism of \(R\)-modules, so that \(R \to R_S\) is a derived localisation. Finally, to see that \(\fib[R \to R_S] = R/R_S[-1]\) is a filtered colimit of perfect \(S\)-torsion complexes we note that
\[
R_S/R = \colim_{D \in \Div_S} R_{D}/R
\]
where each \(R_{D}/R\) is perfect (being the cokernel of an injective map of projective modules) and \(S_{\infty}\)-torsion: \(R_S \otimes_R[R_D/R] = \cof[R_S \to R_S \otimes_R R_D] = 0\).
\end{proof}

Let now $M$ be a line bundle (that is, a finitely generated projective module of rank 1) over $R$ with an $R$-linear involution,
which we regard as a $\ZZ$-module with involution over $R$ as in Definition~\refthree{definition:modinv}.
For every \(-\infty \leq m \leq \infty\) we may endow $\Dperf(R)$ with the Poincaré structure $\QF^{\geq m}_M$, and $\Dperf(R_S)$ with the Poincaré structure $\QF^{\geq m}_{M_S}$ associated to the localised line bundle
\(M_S:=R_S\otimes^\mathrm{U}_{R}M\).
The extension of scalars is then a Poincaré functor, see Lemma~\refone{lemma:Poincresscalars}, so that $\Dperf_S(R)$ is closed under the duality of $\Dperf(R)$ induced by $M$ and becomes a Poincaré subcategory, with the restricted Poincaré structure.
By abuse of notation, we denote this restricted Poincaré structure again by $\QF^{\geq m}_M \colon \Dperf_S(R){\op} \to \Spa$.

\begin{proposition}
\label{proposition:verdier}%
Let $R$ be a Dedekind ring and $M$ a line-bundle over $R$ with $R$-linear involution. Then
the sequence of Poincaré $\infty$-categories
\[
(\Dperf_S(R),\QF^{\geq m}_M) \lto (\Dperf(R),\QF^{\geq m}_M) \lto (\Dperf(R_S),\QF^{\geq m}_{M_S})
\]
is a Poincaré-Verdier sequence. In particular, it induces a fibre sequence of $\GW$ and $\L$-spectra.
\end{proposition}
\begin{proof}
The last statement follows from the first, since $\GW$ and $\L$ are Verdier-localising functors, which was proven in Corollary~\reftwo{corollary:gwadd} and Corollary~\reftwo{corollary:Ladd}. We now wish to apply Proposition~\reftwo{proposition:left-Kan-extension-specific-structures}. For this we need to check that $M$ is compatible with the localisation $R \to R_S$ in the sense of Definition~\reftwo{definition:module-compatible-with-induction}, which follows from the fact that $R$ and $R_S$ are commutative and $M$ is an $R\otimes R$-module through the multiplication map of $R$, see Example~\reftwo{example:involution-induction-compatible}(1). Furthermore, by Lemma~\refthree{lemma:derived-localisation}, $R_S$ is a flat $R$-module. In addition, the map $\K_0(R) \to \K_0(R_S)$ is surjective: As filtered colimits along surjections are surjections, it suffices to argue this in the case where $S$ is finite. In this case, $R_S$ is itself a Dedekind ring, so it suffices to argue that the map $\Pic(R) \to \Pic(R_S)$ is surjective. This follows from the observation that $\Pic(R)$ and $\Pic(R_S)$ are respectively the quotients of the free abelian groups generated by the prime ideals of $R$ and $R_S$. 
The proposition then follows from Lemma~\refthree{lemma:derived-localisation}.
\end{proof}

We now restrict our attention to the case of the symmetric Poincaré structure \(\QF^{\sym}_M = \QF^{\geq -\infty}_M\) and aim to describe more explicitly the 
shifted Poincaré structure \((\QF^{\sym}_{M})\qshift{1} = \QF^{\sym}_{M[1]}\) restricted to \(\Dperf_S(R)\).
\begin{lemma}
\label{lemma:explicit}%
\
\begin{enumerate}
\item
\label{item:lemma-explicit:1}%
The canonical t-structure on $\Der(R)$ restricts to a t-structure on $\Dperf_S(R)$. Its 
heart \(\Dperf_S(R)^{\heartsuit} = \Torf_S(R)\) is the abelian category of finitely generated \(S_\infty\)-torsion \(R\)-modules. 
\item
\label{item:lemma-explicit:2}%
For $T \in \Dperf_S(R)^\heartsuit$, the space $\Omega^\infty \QF^\sym_{M[1]}(T)$ is discrete and naturally isomorphic to the abelian group of symmetric $M_S/M$-valued bilinear forms on $T$.
\item
\label{item:lemma-explicit:3}%
For $T \in \Dperf_S(R)^\heartsuit$, we have $\Dual_{M[1]}T = \Hom_R(T,M_S/M)$. In particular, $\Dual_{M[1]}$ restricts to a duality on the heart $\Dperf_S(R)^\heartsuit$.
\end{enumerate}
\end{lemma}

\begin{proof}
Since \(R\) and \(R_S\) have global dimension \(\leq 1\), the perfect derived categories \(\Dperf(R)\) and \(\Dperf(R_S)\) inherit the canonical t-structures from the respective unbounded derived categories. By Lemma~\refthree{lemma:derived-localisation}, \(R_S\) is flat over \(R\). Therefore, the localisation functor \(\Dperf(R) \to \Dperf(R_S)\) preserves both connective and truncated objects and hence commutes with truncations and connective covers. As a result, its kernel \(\Dperf_S(R) \subseteq \Dperf(R)\) is closed under truncation and connective covers, and so inherits a t-structure from \(\Dperf(R)\), i.e.\ the inclusion inclusion \(\Dperf_S(R) \subseteq \Dperf(R)\) commutes with truncations and connective covers. The heart of this restricted t-structure is then the abelian category \(\Dperf_S(R) \cap \Der(R)^{\heartsuit} = \Torf_S(R)\) of finitely generated \(S_{\infty}\)-torsion \(R\)-modules, showing~\refthreeitem{item:lemma-explicit:1}.
For \refthreeitem{item:lemma-explicit:2}, we consider the short exact sequence of $R$-modules
\[
0 \to M \to M_S \to M_S/M \to 0
\]
which induces a fibre sequence of Poincaré structures
$\QF^{\sym}_M \to \QF^{\sym}_{M_S} \to \QF^{\sym}_{M_S/M}$
on \(\Dperf(R)\). Now, $\QF^\sym_{M_S}$ vanishes on $\Dperf_S(R)$, so the 
natural transformation \(\QF^{\sym}_{M_S/M} \to \QF^{\sym}_{M[1]}\) restricts to an equivalence on \(\Dperf_S(R)\). We then obtain a natural equivalence
\[
\QF^\sym_{M[1]}(T) \simeq \QF^\sym_{M_S/M}(T) = \hom_R(T \otimes_R T,M_S/M)^{\hC}
\]
for $T \in \Dperf_S(R)$.
Since \(M_S/M\) is discrete, the last mapping spectrum is coconnective as soon as \(T\) is connective. In particular, its underlying space is discrete when \(T,\) belongs to the heart. In this case, we have the equivalence
\[
\Omega^\infty \hom_R(T\otimes_R T,M_S/M)^{\hC} \simeq \Hom_R(T\otimes^\mathrm{U}_R T,M_S/M)^{\Ct}
\]
where the latter term is the abelian group of symmetric \(M_S/M\)-valued forms on \(T\). Finally, to see \refthreeitem{item:lemma-explicit:3}, 
we again use the above short exact sequence of modules to find a fibre sequence
\[
\map_R(X,M_S) \lto \map_R(X,M_S/M) \lto \map_R(X,M[1]) = \Dual_{M[1]}X
\]
for $X \in \Dperf(R)$. We therefore find that for $T \in \Dperf_S(R)$, the first term above vanishes, and we obtain an equivalence $\Dual_{M[1]}T = \map_R(T,M_S/M)$. It then suffices to argue that this mapping spectrum is discrete if $T$ lies in the heart $\Dperf_S(R)^\heartsuit$. To see this, we may assume that $S$ is the set of all prime ideals, in which case $R_S$ is the field of fractions of $R$. Then we find that $M_S/M$ is a divisible and hence injective $R$-module \cite{Kurdachenko-dedekind}, showing that $\Dual_{M[1]}T$ is indeed discrete.
\end{proof}

Using the above, we would like to give an explicit description of the boundary map
\[
\partial_0\colon \L^{\sym}_0(R_S;M_S) \lto \L_{-1}(\Dperf(R)_{S},\QF^{\sym}_M) \cong \L_0(\Dperf(R)_{S},\QF^{\sym}_{M[1]})
\]
associated to the fibre sequence in \(\L\)-theory
arising from the Poincaré-Verdier sequence of Proposition~\refthree{proposition:verdier}. 
By Corollary~\refthree{corollary:d-zero} any class in \(\L^{\sym}_0(R_S; M_S)\) can be represented by a finitely generated projective \(R_S\)-module \(V\) equipped with a unimodular \(M_S\)-valued symmetric form \(b\). %
Recall that an \(R\)-lattice \(P \subseteq V\) is a finitely generated projective \(R\)-submodule of $V$ inducing an isomorphism \(R_S \otimes^\mathrm{U}_R P \to V\). We note that such lattices always exist. Indeed, any such \(V\) is tensored up from a finitely generated projective \(R_{S'}\)-module \(V'\) for \(S'\) a finite set of primes. Then \(R_{S'}\) is a Dedekind ring so that \(V'\) splits as a direct sum of line bundles, and each such line bundle is tensored up from an \(R\)-line bundle since the map \(\Pic(R) \to \Pic(R_S)\) is surjective. Now any \(R\)-lattice \(P \subseteq V\) has a dual \(R\)-lattice \(P^* \subseteq V\) with respect to \(b\), spanned by those vectors \(v \in V\) such that \(b(v,u)\) lies in \(M\) for every \(u \in P\). We will say that an \(R\)-lattice \(P \subseteq V\) is \(R\)-integral if \(P \subseteq P^*\).
We note that any \(R\)-lattice can be made integral by suitably shrinking it, and so any such \((V,b)\) admits an integral \(R\)-lattice. %

We then observe that $T:=P^*/P$ is an object of $\Dperf_S(R)^\heartsuit$ and carries an induced $M_S/M$-valued symmetric form $c$ given by the formula
\[
c([v],[u]) = [b(v,u)] \in M_S/M.
\]
This form is well-defined since $b(u,v) \in M$ whenever $u,v$ are elements of $P^*$ and at least one of them is contained in $P \subseteq P^*$, and defines a Hermitian form for $(\Dperf_S(R),\QF^{\sym}_{M[1]})$ by Lemma \refthree{lemma:explicit}~\refthreeitem{item:lemma-explicit:2}.
With this notation at hand, we have the following:

\begin{proposition}
\label{proposition:boundary-map}%
Let \(V\) be a finitely generated projective \(R_S\)-module equipped with a unimodular symmetric form \(b\colon V \times V \to M_S\), let \(P \subseteq V\) be an integral \(R\)-lattice and $T= P^*/P$ as above. 
Then we have
\[
\partial_0[V,b] = [T,c] \in \L_0(\Dperf_S(R),\QF^{\sym}_{M[1]}).
\]
\end{proposition}
\begin{proof}
We follow the description of the boundary map given in Proposition~\reftwo{proposition:boundary-L}. The first step is to lift \((V,b)\) to a hermitian object in \(\Dperf(R)\), which we do by viewing the restriction \(b_P := b|_{P}\) as an \(M\)-valued hermitian form on \(P\), so that the pair \((P,b_P)\) determines a hermitian lift of \((V,b)\). The bilinear part of \(b\) then identifies the dual lattice \(P^* \subseteq V\) with \(\Dual_{M}P\), such that \((b_P)_{\sharp}\) is given by the inclusion \(P \subseteq P^*\).
As in Proposition~\reftwo{proposition:boundary-L}, we then view \((P,b_P)\) as a surgery datum on \(0\) with respect to \(\QF^{\sym}_{M[1]}\). Carrying out the surgery yields a Poincaré \(\QF^{\sym}_{M[1]}\)-form \(c\) on \(T := \cof[P \to \Dual_M P] = P^*/P\), %
and one has
\[
\partial_0[V,b] = [T,c] \in \L_0(\Dperf_S(R),\QF^{\sym}_{M[1]}).
\]
We would now like to unwind the definitions in the surgery to obtain an explicit description of \(c\) in terms of \(b\). For this, it will be convenient to introduce an auxiliary datum as follows. 
Suppose given a hermitian structure \(\QFD\) on \(\Dperf(R)\) which vanishes on \(\Dperf_S(R)\), and which is equipped with a natural transformation \(\QF^{\sym}_M \Rightarrow \QFD\). Write \(\QFD' = \cof[\QF^{\sym}_M \Rightarrow \QFD]\) and consider the commutative diagram
\[
\begin{tikzcd}[row sep = 10pt, column sep = 5pt]
& \QF^{\sym}_M(T) \ar[rr]\ar[d] && \QF^{\sym}_M(P^*) \ar[rr]\ar[d] && \QF^{\sym}_M(P) \times_{\map(P,P^*)} \map(P^*,P^*)\ar[d] \\  
0 \ar[r,phantom,"\simeq"] & \QFD(T) \ar[rr]\ar[d] && \QFD(P^*) \ar[rr,"\simeq"]\ar[d] && \QFD(P) \times_{\Bil_{\QFD}(P,P)} \Bil_{\QFD}(P^*,P) \ar[d] \\
\QF^{\sym}_{M[1]}(T) \ar[r,phantom,"\simeq"] & \QFD'(T) \ar[rr] && \QFD'(P^*) \ar[rr] && \QFD'(P) \times_{\Bil_{\QFD'}(P,P)} \Bil_{\QFD'}(P^*,P) 
\end{tikzcd}
\]
whose rows and columns are fibre sequences by Corollary \refone{example:usual-sequence}.
The pair \((b_P,\id_{P^*})\) then canonically refines to a point \((b_P,\id_{P^*},\eta)\) in the (infinite loop space of the) top right term above, where \(\eta\) is the identity homotopy from \((b_P)_{\sharp}\) to itself in \(\map(P,P^*)\). On the other hand, since \(\QFD(T) = 0\), the image of \((b_P,\id_{P^*},\eta)\) in the middle right term uniquely lifts to a form \(b_{P^*} \in \Om^{\infty}\QFD(P^*)\). %
Let \([b_{P^*}] \in \Om^{\infty}\QFD'(P^*)\) be the image of \(b_{P^*}\). Then the data of \((b_P,\id_{P^*},\eta)\) determines a null homotopy of the image of \([b_{P^*}]\) in the bottom right fibre product, and hence refines \([b_{P^*}]\) to a \(\QFD'\)-form \(c\) on \(T\), which we may identify with a \(\QF^{\sym}_{M[1]}\)-form on $T$.

When \(\QFD = 0\) this amounts to the usual procedure for computing the resulting form \(c\) on the trace \(T\) of surgery along \((P,b_P) \to 0\). On the other hand, this construction is natural in \(\QFD\), and so any map between two hermitian structures under \(\QF^{\sym}_M\) vanishing on \(\Dperf_S(R)\) determines an identification between the resulting forms \(c\). In other words, to compute the form \(c\) we may choose \(\QFD\) to our convenience. Let us hence choose \(\QFD\) to be \(\QF^{\sym}_{M_S}\), considered as a hermitian structure on \(\Dperf(R)\), so that \(\QFD' = \QF^{\sym}_{M_S/M}\). %
In this case we may identify \(b_{P^*}\) above with the restriction \(b_{P^*} := b|_{P^*}\) of the \(M_S\)-valued form \(b\) to \(P^* \subseteq V\). The \(M_S/M\)-valued form \([b_{P^*}]\) on \(P^*\) is then simply the form which sends \(v,u \in P^*\) to the class \([b(v,u)] \in M_S/M\) of the element \(b(v,u) \in M_S\), and the descended form \(c\) on \(T\) %
is given by the desired explicit formula.
\end{proof}

\subsection{Dévissage}
Let again $R$ be a Dedekind ring,  $S$ a set of non-zero prime ideals of $R$ and $M$ a line bundle over $R$.
In general, the objective of dévissage is to identify the Grothendieck-Witt and \(\L\)-spectra of the Poincaré \(\infty\)-category \((\Dperf_S(R),\QF^{\s}_M)\) in terms of those of the residue fields \(\FF_{\fp} :=R/{\fp}\) for \(\fp \in S\). To establish this, we begin by refining the restriction of scalars functor 
to a Poincaré functor.

We recall that for a ring homomorphism \(f\colon A \to B\),  the extension of scalars functor $f_! \colon \D(A) \to \D(B)$ is left adjoint to the restriction of scalars functor $f^* \colon \D(B) \to \D(A)$, which admits a further right adjoint $f_* \colon \D(A) \to \D(B)$.
If \(B\) is moreover perfect as an \(A\)-module (that is, it admits a finite resolution by finitely generated projective \(A\)-modules), then \(f^*\) restricts to a functor $f^* \colon \Dperf(B) \to \Dperf(A)$ on perfect objects.
Let $M$ and $N$ be invertible $\ZZ$-modules with involution, respectively over $A$ and $B$. Then any map $\Psi \colon (f\otimes f)^*(N) \to M$ induces a hermitian structure on the restriction functor $f^*$
given by the natural transformation 
\[
\QF^{\s}_N(X) = \map_{B\otimes B}(X \otimes X,N)^{\hC} \lrar \map_{A\otimes A}((f\otimes f)^*(X \otimes X),(f\otimes f)^*N)^{\hC} \lrar
\]
\[
\map_{A\otimes A}(f^*(X) \otimes f^*(X),(f\otimes f)^*N)^{\hC} \st{\Psi_*}{\lrar} \map_{A\otimes A}(f^*(X) \otimes f^*(X),M)^{\hC} = \QF^{\s}_M(f^*(X)).
\]

\begin{lemma}
\label{lemma:hermitian_structure_on_restriction}%
In the notation above, the hermitian functor \((f^*,\psi)\colon (\Dperf(B),\QF^{\s}_N) \to (\Dperf(A),\QF^{\s}_M)\) is Poincaré if and only if the map $N \to f_*(M)$ induced by $\Psi$ by restricting to one of the $A$-module structures is an equivalence in $\D(B)$.
\end{lemma}
\begin{proof}
Since \(\Dperf(B)\) is generated by \(B\) via finite colimits and retracts it will suffice to show that the associated natural transformation
\[
f^*\Dual(X) \lto \Dual f^*(X)
\]
evaluates to an equivalence on \(X = B\). Unwinding the definitions, the above map for \(X=B\) identifies with the map $f^*N \to \map_A(f^*B,M) = f^*f_*M$
which is the image under \(f^*\) of \(\gamma\colon N \to f_* M\). Since \(f^*\) is conservative this image is an equivalence if and only if \(\gamma\) is an equivalence.
\end{proof}

Now suppose that \(\fp \subseteq R\) is a non-zero prime ideal in a Dedekind ring, and let 
\(p\colon R \to \FF_{\fp} := R/\fp\) 
be the quotient map. We note that $\fp$ is a rank 1 projective module and hence is $\otimes$-invertible, where the inverse is given by the dual $\fp^{-1} \simeq \map_R(\fp,R)$,
which in turn can also be realized as the fractional ideal 
\[
\fp^{-1} = \{x \in R \;|\; \forall a \in \fp \text{ we have } ax \in R \} \subseteq R,
\]
see \cite{stacks}*{Tag 0AUW}. We find that $\FF_{\fp}$ is perfect as an $R$-module, as it is represented by the chain complex $[\fp \to R]$ with $\fp$ in degree 1 and $R$ in degree $0$. Let us then consider the adjunction
\begin{equation*}
p^*\colon \D(\FF_{\fp}) \adj \D(R)\cocolon p_*
\end{equation*}
and note that $p^*$ preserves compact objects since $p^*(\FF_\fp)$ is compact. It follows that $p_*$ preserves filtered colimits, and since it is exact it in fact preserves all colimits. We recall that $p_\ast$ is given by the formula
\[
p_*(X) = \map_R(p^{\ast}\FF_\fp,X)
\]
regarded as an $\FF_\fp$-module via the functoriality in the first variable. We recall that $M$ is assumed to be a line bundle over $R$.

\begin{lemma}
\label{lemma:purity}%
Let $\fp \subseteq R$ be a non-zero prime ideal in a Dedekind ring and let $p \colon R \to \FF_{\fp}$ be the quotient map. Then there is a canonical equivalence
\[
p_*M \simeq (\fp^{-1}M/M)[-1] .
\]
In addition, any choice of a uniformiser for $\fp$ determines an isomorphism $(\fp^{-1}M/M) \cong M/\fp M = p_!M$.
\end{lemma}
\begin{proof}
Applying the functor $\map_R(-,M)$ to the fibre sequence $\fp\to R\to \FF_\fp$ %
we obtain a fibre sequence
\[
p^*p_*(M) \lto M \lto \fp^{-1} M,
\]
and consequently a distinguished equivalence
\[
p^*p_*(M)[1] \simeq \fp^{-1}M/M
\]
of $R$-modules. The $R$-module structure on $\fp^{-1}M/M$ a priori descends to an $\FF_p$-module structure, and one obtains an equivalence $p_*(M)[1] \simeq \fp^{-1}M/M$ of $\FF_p$-modules, since $p^*$ is fully faithful on discrete modules. 
To see the second part, choose a \(\fp\)-uniformiser \(\pi\), that is, an element of \(K\) whose \(\fp\)-valuation is exactly \(1\). Then \(\pi\) belongs to the local ring \(R_{(\fp)} \subseteq K\), and generates its valuation ideal \(\tilde{\fp} := R_{(\fp)}\fp \subseteq R_{(\fp)}\) there. Since \(M_{(\fp)}\) is flat over \(R\) the induced map $\fp^{-1}M/M \to \tilde{\fp}^{-1}M_{(\fp)}/M_{(\fp)}$ is an isomorphism. Multiplication by \(\pi\) then induces the claimed isomorphism
\[
\tilde{\fp}^{-1}M_{(\fp)}/M_{(\fp)} \cong M_{(\fp)}/\tilde{\fp}M \cong p_!(M). \qedhere
\]
\end{proof}

\begin{corollary}%
\label{corollary:devissage-functor}%
Let $\fp \subseteq R$ be a non-zero prime ideal in a Dedekind ring and let $p \colon R \to \FF_{\fp}$ be the quotient map. For any line bundle $M$ over $R$, the restriction of scalars functor $p^*\colon\Dperf(\FF_{\fp}) \to \Dperf(R)$ canonically refines to a Poincaré functor
\[
(\Dperf(\FF_\fp),\QF^\s_{p_*M}) \lto (\Dperf_S(R),\QF^\s_M).
\]
\end{corollary}
\begin{proof}
We view $M$ and $p_*M$ as modules over $R\otimes R$ and $\FF_\fp \otimes \FF_\fp$ via the multiplication maps, respectively. By Lemma~\refthree{lemma:purity} $p_*M$ is in particular an invertible $\FF_\fp$-module. We may therefore apply
Lemma~\refthree{lemma:hermitian_structure_on_restriction} to the unit of the adjunction $p^*p_*(M) \to M$ to obtain a hermitian structure on the restriction of scalars functor
\[
(\Dperf(\FF_\fp),\QF^\s_{p_*M}) \lto (\Dperf(R),\QF^\s_M)
\]
which is Poincaré again by Lemma~\refthree{lemma:hermitian_structure_on_restriction}. It then suffices to observe that the image of the functor $p^*$ is contained in the subcategory $\Dperf_S(R)$ of $S$-torsion modules.
\end{proof}

\begin{theorem}[Dévissage]
\label{theorem:devissage}%
Let $R$ be a Dedekind ring, $S$ a set of non-zero prime ideals of $R$, and $M$ a line bundle over $R$ with $R$-linear involution. Then
for every \(m \in \ZZ\) the direct sum Poincaré functor 
\begin{equation}
\label{equation:psi}%
\psi_S\colon\displaystyle\mathop{\oplus}_{\fp \in S} (\Dperf(\FF_{\fp}),(\QF^{\s}_{p_*M})\qshift{m}) \lto (\Dperf_S(R),(\QF^{\s}_M)\qshift{m}) 
\end{equation}
induces a canonical equivalence on algebraic \(\K\)-theory, \(\GW\)-theory and \(\L\)-theory spectra.
\end{theorem}

\begin{proof}
First, note that by the fibre sequence of Corollary~\reftwo{corollary:tate-square-L}, it will be enough to prove the theorem for algebraic \(\K\)-theory and \(\L\)-theory. 
Second, both sides of \(\psi_S\) depend on \(S\) in a manner that preserves filtered colimits. More specifically, if we write \(S\) as a filtered colimit \(S = \colim_{S'\subseteq S,|S|<\infty} S'\) of its finite subsets then the direct sum on the left hand side of~\eqrefthree{equation:psi} is the colimit of the corresponding finite direct sums, while on the right hand side the full subcategory \(\Dperf_S(R) \subseteq \Dperf(R)\) is the union of all the full subcategories \(\Dperf(R)_{S'}\) for finite \(S'\subseteq S\). Since both algebraic \(\K\)-theory and \(\L\)-theory commute with filtered colimits we may reduce to the case where \(S\) is finite.

In this case, the left hand side of~\eqrefthree{equation:psi} can also be written as the product of \(\Dperf(\FF_{\fp})\) for varying \(\fp \in S\) (recall that \(\Catp\) is semi-additive, see Proposition~\refone{proposition:catp-pre-add}). In addition, each \(\Dperf(\FF_{\fp})\), being the perfect derived category of a field, supports a t-structure inherited from \(\D(\FF_{\fp})\), and so we can endow the left hand side of~\eqrefthree{equation:psi} with the product of the corresponding t-structures. The heart of this product t-structure is then the direct sum \(\oplus_{\fp \in S} \Vect(\FF_{\fp})\) where \(\Vect(\FF_{\fp})\) is the abelian category of finite dimensional \(\FF_{\fp}\)-vector spaces. We can also identify it with the category of finitely generated modules over the product ring \(\prod_{\fp \in S} \FF_{\fp}\).
For the right hand side of~\eqrefthree{equation:psi}, Lemma~\refthree{lemma:explicit} tells us that \(\Dperf_S(R)\) also carries a t-structure, whose heart is the category \(\Torf_S(R)\) of finitely generated \(S_{\infty}\)-torsion \(R\)-modules. 
Now since~\eqrefthree{equation:psi} is induced by the various restriction functors \(p^*\colon \Dperf(\FF_{\fp}) \to \Dperf(R)\) it preserves connective and truncated objects with respect to the t-structures just discussed. The functor
\begin{equation}
\label{equation:psi-heart}%
\displaystyle\mathop{\oplus}_{\fp \in S} \Vect(\FF_{\fp}) \lto \Torf_S(R)
\end{equation}
induced by~\eqrefthree{equation:psi} on the respective hearts is then fully-faithful (even though~\eqrefthree{equation:psi} itself is not fully-faithful) and can be identified with the inclusion of the full subcategory of finitely generated \(S\)-torsion modules inside all finitely generated \(S_{\infty}\)-torsion modules, i.e.\ with the full subcategory of \emph{semi-simple} objects inside the abelian category \(\Torf_S(R)\). By the main result of Barwick~\cite{barwick-exact} 
the inclusion of hearts induces an equivalence on algebraic \(\K\)-theory spectra on both the domain and codomain of~\eqrefthree{equation:psi}. The desired claim for algebraic \(\K\)-theory is hence equivalent to saying that the inclusion of abelian categories~\eqrefthree{equation:psi-heart} induces an equivalence on algebraic \(\K\)-theory, which in turn follows from Quillen's classical dévissage theorem
 ~\cite{quillen}*{Theorem \S 5.4}.

We will now show that~\eqrefthree{equation:psi} induces an equivalence on \(\L\)-theory. Recall that by Corollary~\refthree{corollary:periodicity-L}, L-theory supports natural equivalences \(\L(\C,\QF\qshift{1}) \simeq \Sig\L(\C,\QF)\). It will thus suffice to prove the claim for \(m=1\) in which case the duality on $\Dperf(\FF_p)$ is given by $\Dual_{p_*(M)[1]} \simeq \Dual_{p_!(M)}$ by Lemma~\refthree{lemma:purity}.
Now since each \(\FF_{\fp}\) has global dimension \(0\) the duality \(\Dual_{p_!M}\) on \(\Dperf(\FF_{\fp})\) maps \(0\)-connective objects to \(0\)-truncated objects and \emph{vice versa}. The same hence holds for the product duality on \(\prod_{\fp \in S} \Dperf(\FF_{\fp})\) with respect to the product t-structure. This also holds for the Poincaré \(\infty\)-category \((\Dperf_S(R),(\QF^{\s}_M)\qshift{1})\) by Lemma~\refthree{lemma:explicit}.
Hence for both sides of~\eqrefthree{equation:psi} we are in the situation of Corollary~\refthree{corollary:dim-zero}, and so to finish the proof it will suffice to show that~\eqrefthree{equation:psi-heart} induces an isomorphism on symmetric and anti-symmetric Witt groups. This follows from the dévissage result of~\cite{devissage}*{Corollary 6.9, Theorem 6.10}.
\end{proof}

The combination of the classical dévissage and localisation theorems of Quillen give rise to the fibre sequence of K-theory spectra
\[
\mathop{\oplus}_{\fp \in S}\K(\FF_{\fp}) \longrightarrow\K(R)\longrightarrow \K(R_S).
\]
From Theorem~\refthree{theorem:devissage} we obtain the corresponding sequences for the symmetric L and GW-spectra.

\begin{corollary}[Localisation-dévissage]
\label{corollary:decomposition-s}%
Under the assumptions of Theorem~\refthree{theorem:devissage}, the restriction and localisation functors yield canonical fibre sequences of spectra
\[
\begin{tikzcd}[row sep=tiny]
\displaystyle\mathop{\oplus}_{\fp \in S}\GW(\FF_{\fp};(\QF^{\s}_{p_*M})\qshift{m}) \ar[r] & \GW(R;(\QF^{\s}_M)\qshift{m}) \ar[r] & \GW(R_S;(\QF^{\s}_{M_S})\qshift{m}) \\
\displaystyle\mathop{\oplus}_{\fp \in S}\L(\FF_{\fp};(\QF^{\s}_{p_*M})\qshift{m}) \ar[r] & \L(R;(\QF^{\s}_M)\qshift{m}) \ar[r] & \L(R_S;(\QF^{\s}_{M_S})\qshift{m})
\end{tikzcd}
\]
for every \(m \in \ZZ\).
\end{corollary}

\begin{remark}
\label{remark:uniformiser-version-L-sequence}%
Making use of Lemma~\refthree{lemma:purity}, after making choices of uniformisers of the ideals $\fp$ in $S$, one obtains the slightly more familiar fibre sequence
\[
\begin{tikzcd}
\displaystyle\mathop{\oplus}_{\fp \in S}\GW(\FF_{\fp};(\QF^{\s}_{p_!M})\qshift{m-1}) \ar[r] & \GW(R;(\QF^{\s}_M)\qshift{m}) \ar[r] & \GW(R_S;(\QF^{\s}_{M_S})\qshift{m}) 
\end{tikzcd}
\]
and likewise for L-theory.
We now briefly remark on the naturality of this sequence. Given a map of Dedekind rings $f\colon R \to R'$, a line bundle $M$ on $R$, and a set of prime ideals $S$ in $R$, we may set $M' = R'\otimes_R M$ and $S'$ the set of primes whose preimage is contained in $S$. There is then the following canonical commutative square
\[
\begin{tikzcd}
	\GW(R;(\QF^{\s}_M)\qshift{m}) \ar[r] \ar[d] & \GW(R_S;(\QF^{\s}_{M_S})\qshift{m}) \ar[d] \\
	\GW(R';(\QF^{\s}_{M'})\qshift{m}) \ar[r] & \GW(R'_{S'};(\QF^{\s}_{{M'}_{\!\! S'}})\qshift{m})
\end{tikzcd}
\]
which induces a map on horizontal fibres: 
\[
\displaystyle\mathop{\oplus}_{\fp \in S} \GW(\FF_{\fp};(\QF^{\s}_{p_!M})\qshift{m-1}) \lto \displaystyle\mathop{\oplus}_{{\fp' \in S'}} \GW(\FF_{\fp'};(\QF^{\s}_{p'_!M'})\qshift{m-1}).
\]
Let us describe this map in the case where $S= \{\fp\}$ consists of a single prime. For any prime $\fp'$ in $S'$, there is then an induced map $\kappa\colon \FF_\fp \to \FF_{\fp'}$ between the corresponding residue fields, and $\kappa_!p_!M \simeq p'_!f_!M = p'_!M'$. However, the map induced on fibres
\[
\GW(\FF_\fp;(\QF^{\s}_{p_!M})\qshift{m-1}) \lto \GW(\FF_{\fp'};(\QF^{\s}_{p'_!{M'}})\qshift{m-1})
\]
is not in general the one induced by the map $\kappa$ and the mentioned equivalence $\kappa_!p_!M \simeq p'_!M'$ unless a uniformiser for $\fp$ is sent to a uniformiser of $\fp'$ under the map $R \to R'$. An example where this indeed fails is the map $\ZZ \to \ZZ[i]$ from the integers to the Gaussian integers with $\fp = (2)$. In this case, $2$ is not a uniformiser of the prime ideal $(1+i)$ in $\ZZ[i]$ which is the single prime ideal containing 2. One can then show that the induced map on K-theory between the residue fields is not the identity. Likewise, using for instance Lemma~\refthree{lemma:boundary-map-local} below, one can also see that the induced map on L-theory between the residue fields is not the identity, but in fact the zero map.
\end{remark}

\begin{remark}
A variant of Corollary~\refthree{corollary:decomposition-s} for \(\L\)-theory of short complexes in non-negative degrees (which, for Dedekind rings, coincides with symmetric \(\L\)-theory in non-negative degrees by Theorem~\refthree{theorem:main-theorem-L-theory} and Corollary~\refthree{corollary:dedekind}), was proven by Ranicki in~\cite{Ranickiyellowbook}*{\S 4.2}. For Grothendieck-Witt theory, Hornbostel and Schlichting prove a dévissage statement and obtain a localisation sequence of the type of Corollary~\refthree{corollary:decomposition-s} under the assumption that $2$ is a unit in $R$, see~\cite{hornbostel},~\cite{hornbostel-schlichting}. Apart from the announcement \cite{Schlichting-integers}*{Theorem 3.2} which provides the above fibre sequence for $\GW$ after passing to connective covers, we are not aware of any previous results in the literature for Grothendieck-Witt spaces, along the lines of Corollary~\refthree{corollary:decomposition-s}, for rings in which \(2\) is not invertible. 
\end{remark}

\begin{remark}
\label{remark:no-quadratic-devissage}%
The dévissage result above is a special feature of the symmetric Poincaré structure: It is the only among the Poincaré structures $\Qgen{m}{M}$ for which this result holds at the spectrum level (not just in a range of degrees). Indeed, to see this it suffices, by Corollary~\refthree{corollary:genuine-is-quadratic}, to argue that dévissage fails for quadratic L-theory. For an explicit example, one can note that the maps
\[
\Omega \L^\qdr(\FF_2) \lto \L^\qdr(\ZZ) \lto \L^\qdr(\ZZ[\tfrac{1}{2}])
\]
cannot be part of a fibre sequence, for instance because it would imply that $\L^\qdr_2(\ZZ) = 0$, which is not the case. Here, we use that $\L^\qdr_2(\ZZ[\tfrac{1}{2}]) \cong \L^\s_2(\ZZ[\tfrac{1}{2}])= 0$, which is of course well-known, but see also Corollary~\refthree{corollary:compute-L-groups-global-field} below, as well as the vanishing of $\L_3^\qdr(\FF_2)$, see Remark~\refthree{remark:odd-vanishing}.

However, dévissage in quadratic L-theory fails only at dyadic primes, that is at those primes which contain the ideal $(2)$. More precisely, if no prime in $S$ is dyadic, then the localisation-dévissage sequence exists also in quadratic L-theory; see Remark~\refthree{remark:devissage-in-quadratic-L} for details. For instance, it can be used to calculate $\L^\qdr(\ZZ[\tfrac{1}{p}])$ for odd primes $p$.
\end{remark}

Shifting the L-theory fibre sequence of Corollary~\refthree{corollary:decomposition-s}, or rather the version discussed in Remark~\refthree{remark:uniformiser-version-L-sequence}, once to the right and choosing uniformisers for all non-zero primes, we obtain a fibre sequence
\[
\L^\s(R) \lto \L^\s(K) \stackrel{\partial}{\lto} \displaystyle\mathop{\oplus}\limits_{\fp} \L^\s(\FF_{\fp})
\]
In the next section, we aim to determine the L-groups of Dedekind rings. In order to do so, we will need to  %
make the effect of the map $\partial$ on $\pi_0$ explicit. Clearly, it suffices to describe the composite of $\partial$ with the projection to $\L^\s(\FF_\fp)$ for each prime $\fp$ of $R$. By naturality of the dévissage theorem and the localisation sequence, to describe this composition we may replace $R$ by its localisation $R_{(\fp)}$ which is a local Dedekind ring and hence a discretely valued ring, as the choice of uniformiser for $\fp$ is (by definition) also a uniformiser for $\fp$, viewed as prime ideal in $R_{(\fp)}$.
Without loss of generality, we may hence assume that $R$ was a discretely valued ring to begin with. Let $\pi$ be a uniformiser of the maximal ideal of $R$, so that every non-zero element in $K$ is uniquely of the form $\pi^i u$ for some unit $u$ in $R$. Clearly, it suffices to describe the map 
\[
\partial_0 \colon \L_0^\s(K) \to \L_0^\s(\FF_\fp)
\]
on generators of the L-group, which are given by the forms $\langle x \rangle=(K,x)$ for units $x$ of $K$, where we have identified canonically $\pi_0(\QF^\s(K))$ with $K$. Indeed, if the characteristic of $K$ is not 2, then every form itself is isomorphic to a diagonal form. If the characteristic is 2, then any unimodular form is the sum of a diagonalisable form and one which admits a Lagrangian, see \cite{milnor-symmetric}*{I \S 3}.
By a change of basis, one finds the relation $\langle x \rangle = \langle xy^2 \rangle$ for any other unit $y$. We may thus suppose without loss of generality that $x$ is either of the form $\pi u$ or of the form $u$, again for $u$ a unit in $R$. By exactness of the localisation-dévissage sequence, we have $\partial_0\langle u \rangle = 0$, so it remains to describe $\partial_0 \langle \pi u \rangle$.

\begin{lemma}
\label{lemma:boundary-map-local}%
Let \(R\) be a local Dedekind ring with maximal ideal \(\fp\), residue field \(\FF_{\fp} = R/\fp\) and fraction field \(K\), and let $\pi$ be a uniformiser for $\fp$, and $u \in R^\times$ a unit. Then we have $\partial_0\langle \pi u \rangle = \langle [u] \rangle$, where $[u]$ denotes the image of $u$ under the map $R^\times \to \FF_\fp^\times$.
\end{lemma}
\begin{proof}
Let \(a \in \fp\) be a generator of \(\fp\), considered as a 1-dimensional form \(\langle a\rangle = (K,a)\). %
Write \(\fp^{-1} = (a^{-1})\) for the inverse fractional ideal of \(\fp\). For this proof, let $\bar{\partial}_0$ denote the map of Proposition~\refthree{proposition:boundary-map}. We will show the following statements.
\begin{enumerate}
\item
\label{item:1-form}%
The class \(\bar{\partial}_0\langle a \rangle \in \L_0(\Dperf_{\fp}(R),\QF^{\sym}_{R[1]})\) is represented by the \(\fp\)-torsion module \(T := \fp^{-1}/R \in \Dperf_{\fp}(R)^{\heartsuit}\), equipped with the symmetric \(K/R\)-valued pairing 
\[
c_a \colon T \times T \to K/R \quad\quad c_a([x],[y]) = [axy] \in K/R.
\]
\item
\label{item:devissage-1}%
Under the dévissage equivalence \(\L^\s_0(\FF_{\fp};\fp^{-1}/R) = \L^\s_0(\FF_{\fp};p_*R[1]) \xrightarrow{\simeq} \L_0(\Dperf_{\fp}(R),\QF^{\sym}_R[1])\), 
the class \(\bar{\partial}_0\langle a\rangle\) corresponds to that of the \(\FF_p\)-vector space \(T' = \fp^{-1}/R\), equipped with the \((\fp^{-1}/R)\)-valued form
\[
c'_a\colon T' \times T' \to \fp^{-1}/R \quad\quad c'_a([x],[y]) = [axy] \in \fp^{-1}/R.
\]
\item
\label{item:devissage-2}%
Using the uniformiser \(\pi\) to define an isomorphism \(p_*R[1] = \fp^{-1}/R \xrightarrow{\pi} R/\fp = \FF_{\fp}\), then, under the resulting equivalence \(\L^\s_0(\FF_{\fp}) \xrightarrow{\simeq} \L_0(\Dperf_{\fp}(R),\QF^{\sym}_{R[1]})\), the class of \(\bar{\partial}_0\langle u\pi\rangle\) corresponds to \(\langle [u] \rangle\). %
\end{enumerate}
Statement~\refthreeitem{item:1-form} is simply the specialization of Proposition~\refthree{proposition:boundary-map} to the present case. Indeed, $R \subseteq K$ is an $R$-lattice whose dual lattice is $\fp^{-1}$, in particular $R$ is an integral lattice and the associated torsion module is precisely $\fp^{-1}/R$ with the described form. For~\refthreeitem{item:devissage-1} we note that the counit map \(p^*p_*R \to R\), through which the dévissage Poincaré functor is constructed (see the proof of Corollary~\refthree{corollary:devissage-functor}), can be identified with the forgetful map \(\map_R(R/\fp,R) = \fib[R \to \fp^{-1}] \to R\). We may consequently factor it as 
\[
\fib[R \to \fp^{-1}] \to \fib[R \to K] \to R ,
\]
which, after suspending, becomes
\( \fp^{-1}/R \to K/R \to R[1] \).
The (suspended) dévissage functor of Corollary~\refthree{corollary:devissage-functor} thus factors as 
\[
(\Dperf(\FF_p),\QF^{\sym}_{p_*R[1]}) \to (\Dperf_{\fp}(R),\QF^{\sym}_{K/R}) \xrightarrow{\simeq} (\Dperf_{\fp}(R),\QF^{\sym}_{R[1]}) ,
\]
where the first functor carries the Poincaré structure induced as in Proposition~\refthree{lemma:hermitian_structure_on_restriction} from the map \(p^*p_*R[1] = \fp^{-1}/R \to K/R\). This Poincaré functor thus sends \((T',c'_a)\) to the Poincaré object \((T,c_a)\), so that~\refthreeitem{item:devissage-1} follows. Finally, \refthreeitem{item:devissage-2} follows from~\refthreeitem{item:devissage-1} since for \(x,y \in R\) we have \(\pi c'_{u\pi}([x/\pi],[y/\pi]) = [uxy] \in \FF_{\fp}\).
\end{proof}

\begin{remark}
\label{remark:boundary-map}%
Lemma~\refthree{lemma:boundary-map-local} identifies the map $\partial_0 \colon \L_0^\s(K) \to \oplus_\fp \L_0^\s(\FF_\fp)$ with the map induced by the maps $\psi^1 \colon \W^\s(K) \to \W^\s(\FF_\fp)$ constructed in \cite{milnor-symmetric}*{Chapter IV \S 1}.
\end{remark}

\subsection{Symmetric and quadratic L-groups of Dedekind rings}
\label{subsection:dedekind-calculation}%

In this section we show that the classical symmetric and quadratic Grothendieck-Witt groups of certain Dedekind rings are finitely generated. By Theorem~\refthree{theorem:fiber-sequence-intro-three} it will suffice to prove the finite generation of the corresponding genuine L-groups, provided the finite generation of the K-groups is known. Moreover, as we show in Corollary~\refthree{corollary:finite-generation-GWm}, by the surgery results established in the earlier sections, it will in fact suffice to show the finite generation of the symmetric and quadratic L-groups.
In this case, we in fact do much more: We give a full calculation of the quadratic and symmetric L-groups of Dedekind rings whose field of fractions is not of characteristic 2. We first treat the symmetric case, where we make essential use of the boundary map
\[
\partial_n \colon \L_n^\s(K) \to \L_n^\s(\FF_\fp)
\]
arising from the localisation-dévissage sequence.

\begin{proposition}
\label{proposition:compute-L-groups}%
Let $R$ be a Dedekind ring whose field of fractions $K$ is not of characteristic 2, and let $\mathcal{I}$ be the (finite) set of dyadic primes of $R$. Then we have
\[
\L_n^\s(R) \cong 
\begin{cases}
\W^\s(R) & \text{ for } n \equiv 0 (4) \\
\displaystyle\mathop{\oplus}\limits_{\fp \in \mathcal{I}} \W^\s(\FF_\fp) & \text{ for } n \equiv 1 (4) \\
0 & \text{ for } n \equiv 2 (4) \\
\coker(\partial_{0}) & \text{ for } n \equiv 3 (4)
\end{cases}
\]
\end{proposition}
\begin{proof}
The case $n\equiv 0 (4)$ follows from combining Corollary~\refthree{corollary:L-zero} and Corollary~\refthree{corollary:dedekind}. Since the symmetric \(\L\)-groups of \(K\) and each residue field \(\FF_{\fp}\) vanish in odd degrees by Corollary~\refthree{corollary:d-zero}, the long exact sequence in \(\L\)-groups furnished by Corollary~\refthree{corollary:decomposition-s} yields for every $k$ an exact sequence
\[
0 \lto \L^{\s}_{2k}(R) \lto \L^{\s}_{2k}(K) \stackrel{\partial_{2k}}{\lto} \oplus_{\fp}\L^{\s}_{2k}(\FF_{\fp}) \lto \L^{\s}_{2k-1}(R) \lto 0.
\]
This shows the case $n \equiv 2 (4)$, as for odd numbers $k$, the group $\L^\s_{2k}(K)$ is isomorphic to the anti-symmetric Witt group of $K$ by Corollary~\refthree{corollary:d-zero}, which vanishes as the characteristic of $K$ is not 2. The remaining cases are obvious from the above exact sequence, making use of the fact that the symmetric L-theory of fields of characteristic 2, like $\FF_{\fp}$ for dyadic primes, is 2-periodic, whereas the symmetric L-theory of fields of odd characteristic vanishes in degrees different from $0 \equiv 4$, see Remark~\refthree{remark:odd-vanishing}.
\end{proof}

\begin{corollary}
\label{corollary:compute-L-groups-global-field}%
Under the assumptions of Proposition~\refthree{proposition:compute-L-groups}, assume in addition that $K$ is a global field and let $d= |\mathcal{I}|$ be the (finite) number of dyadic primes of \(R\). Then we have
\[
\L^{\s}_n(R) = 
\begin{cases}
\W^{\s}(R) & \text{ for } n\equiv 0 (4) \\
(\ZZ/2)^d & \text{ for } n \equiv 1 (4) \\
0 & \text{ for } n \equiv 2 (4) \\
\Pic(R)/2 &\text{ for } n\equiv 3 (4)
\end{cases}
\]
\end{corollary}
\begin{proof}
The case $n \equiv 1 (4)$ follows since the assumption that $K$ is global says that the residue fields $\FF_\fp$ at non-zero primes are finite fields. The claim then follows from the fact that the symmetric Witt group of a finite field of characteristic 2 is given by $\ZZ/2$. For the other non-trivial case, Remark~\refthree{remark:boundary-map} gives the following commutative diagram
\[
\begin{tikzcd}[row sep=small]
\W^\s(K) \ar[r,"\psi^1"] \ar[d,"\cong"] & \displaystyle\mathop{\oplus}\limits_{\fp} \W^\s(\FF_\fp) \ar[d,"\cong"] \\
\L^\s_0(K) \ar[r,"\partial_0"] & \displaystyle\mathop{\oplus}\limits_{\fp} \L^\s_0(\FF_\fp)
\end{tikzcd}
\]
It is then shown in \cite{milnor-symmetric}*{Chapter IV \S 4} that the cokernel of the upper horizontal map is given by $\Pic(R)/2$, provided $K$ is a number field. In \cite{scharlau-quadratic}*{Chapter 6, \S 6, Theorem 6.11} this is extended to hold for a general global field $K$.
\end{proof}

\begin{remark}
\label{remark:one}%
We recall that there is a canonical equivalence $\L^{-\s}(R) \simeq \Sigma^2\L^{\s}(R)$, so that Proposition~\refthree{proposition:compute-L-groups} and Corollary~\refthree{corollary:compute-L-groups-global-field} also determine the $(-1)$-symmetric L-groups.
\end{remark}

When the fraction field of \(R\) is a global field of characteristic \(2\) we have a similar result:
\begin{corollary}
\label{corollary:compute-L-groups-global-Dedekind-ring-characteristic-2}%
Let \(R\) be a Dedekind ring whose field of fractions \(K\) is a global field of characteristic \(2\). Then
\[
\L^{\s}_n(R) = 
\begin{cases} 
\W^{\s}(R) & \text{ for } n \equiv 0 (2) \\
\Pic(R)/2 & \text{ for } n\equiv 1 (2)
\end{cases}
\]
\end{corollary}
\begin{proof}
As in the proof of Proposition~\refthree{proposition:compute-L-groups}, we have an exact sequence
\[
0 \lto \L^{\s}_{2k}(R) \lto \L^{\s}_{2k}(K) \stackrel{\partial_{2k}}{\lto} \oplus_{\fp}\L^{\s}_{2k}(\FF_{\fp}) \lto \L^{\s}_{2k-1}(R) \lto 0.
\]
Since $R$ is an $\FF_2$-algebra, the L-groups $\L^\s_n(R)$ are 2-periodic and since $K$ is a global field, all residue fields $\FF_\fp$ are (finite) fields of characteristic 2. We recall that there is an exact sequence of abelian groups
\[
K^\times \stackrel{\mathrm{div}}{\lto} \mathrm{Div}(R) \lto \Pic(R) \lto 0
\]
where $\mathrm{Div}(R)$ is the free abelian group generated by the prime ideals of $R$, and the group homomorphism $\mathrm{div}$ is determined by the following: For a non-zero element $x$ of $R$, write $(x) = \fp_1^{r_1} \cdot \dots \cdot \fp_n^{r_n}$ with natural numbers $r_i$. Then $\mathrm{div}(x) = \sum_{i=1}^n r_i \cdot \fp_i$. Now consider the diagram
\[
\begin{tikzcd}
\ZZ/2[R\setminus\{0\}] \ar[r] \ar[d,"\langle - \rangle"] & \mathrm{Div}(R)/2 \ar[r] \ar[d,"\cong"] & \Pic(R)/2 \ar[r] \ar[d, dashed] & 0 \\
\W^\s(K) \ar[r,"\partial_0"] & \oplus_{\fp} \W^\s(\FF_\fp) \ar[r] & \L_{1}^\s(R) \ar[r] & 0
\end{tikzcd}
\]
consisting of exact horizontal sequences and the left most top vertical map is induced by the map $\mathrm{div}$ above. Here, the middle vertical isomorphism is induced from the isomorphism $\mathrm{Div}(R)/2 \cong \oplus_{\fp} \ZZ/2$ and the isomorphisms $\W^\s(\FF_\fp) \cong \ZZ/2$. The left square commutes by an explicit check, so that there exists a dashed arrow as indicated. By construction, the dashed map is a surjection, and an injection by the observation that the left most vertical map is surjective: Indeed,  the Witt group $\W^\s(K)$ is generated by the forms $\langle x \rangle$ for $x \in K^\times$ by \cite{milnor-symmetric}*{I \S 3} and for every $x$ in $K^\times$ there is a $y \in K^\times$ such that $xy^2$ is contained in $R \setminus\{0\}$.
\end{proof}

\begin{remark}
\label{remark:local-Dedekind}%
Let $R$ be a local Dedekind ring with fraction field $K$ and residue field $k$. Then the map $\partial_0 \colon \L_0^\s(K) \to \L_0^\s(k)$ is surjective by Lemma~\refthree{lemma:boundary-map-local}, and we have seen earlier that its kernel is $\L_0^\s(R)$. Assuming that the characteristic of $K$ is not $2$, we deduce from Proposition~\refthree{proposition:compute-L-groups} that $\L_n^\s(R)$ vanishes for $n\equiv 2,3\mod(4)$. Furthermore, for $n\equiv 1\mod(4)$, we find that $\L_n^\s(R)$ is either isomorphic to $\W^\s(k)$, if the characteristic of $k$ is 2, or is trivial otherwise. If the characteristic of $K$ is 2, we deduce from the proof of Corollary~\refthree{corollary:compute-L-groups-global-Dedekind-ring-characteristic-2} that $\L_1^\s(R) = 0$.
\end{remark}

We now want to give a formula for the quadratic L-groups of Dedekind rings, similar to Proposition~\refthree{proposition:compute-L-groups} and Corollary~\refthree{corollary:compute-L-groups-global-field}.
By Remark~\refthree{remark:no-quadratic-devissage}, the strategy based on the localisation-dévissage sequence will not work in general for quadratic \(\L\)-groups. %
Instead, we will make use of a general localisation-completion property, Proposition~\refthree{proposition:LC} below, and a rigidity property of quadratic L-theory, Proposition~\refthree{proposition:hensel}. These two properties do not require assuming that the ring in question is Dedekind, or even commutative. Following \S\reftwo{section:appendix-module-examples}, for a subgroup $\mathrm{c}\subset \K_0(R)$ fixed by the involution, we let $\Der^{\mathrm{c}}(R)$ denote the full subcategory of $\Dperf(R)$ spanned by the complexes whose $\K_0$-class lies in $\mathrm{c}$. The subgroup $c$ is sometimes called the \emph{control term} in the literature.

\begin{proposition}
\label{proposition:LC}%
Let $R$ be a ring, $M$ an invertible $\ZZ$-module with involution over $R$, and $S$ the multiplicatively closed subset generated by an integer $\ell \in R$. Assume that $R^\cwedge_\ell$ is the derived $\ell$-completion of $R$, (e.g.\ that
the order of the $\ell^\infty$-torsion in $R$ is bounded, 
for instance that $\ell$ is not a zero divisor). Then the square
\[
\begin{tikzcd}
(\Dperf(R),\Qgen{m}{M}) \ar[r] \ar[d] & (\Dperf(R^\cwedge_\ell),\Qgen{m}{M^\cwedge_\ell}) \ar[d] \\
(\Der^{\mathrm{c}}(R[\tfrac{1}{\ell}]),\Qgen{m}{S^{-1}M}) \ar[r] & (\Der^{\mathrm{c}'}(R^\cwedge_\ell[\tfrac{1}{\ell}]),\Qgen{m}{S^{-1}(M^\cwedge_\ell) })
\end{tikzcd}
\]
is a Poincaré-Verdier square for all $m \in \ZZ \cup \{ \pm \infty \}$, where $c=\im(\K_0(R)\to \K_0(R[\tfrac{1}{\ell}]))$, and $c'=\im(\K_0(R^\cwedge_\ell)\to \K_0(R^\cwedge_\ell[\tfrac{1}{\ell}]))$. In particular it becomes a pullback after applying $\GW$ or $\L$.
\end{proposition}
\begin{proof}
We show that the canonical maps $f\colon R\to R^\cwedge_\ell $ and $\alpha\colon M\to (f\otimes f)^\ast(M^\cwedge_\ell)$ satisfy the conditions of Proposition~\reftwo{proposition:analytic-isomorphism}.
For \reftwoitem{item:eta-is-an-iso}, we note that the composite morphism
\[
R^\cwedge_\ell \otimes_R M \to (R^\cwedge_\ell \otimes R^\cwedge_\ell )\otimes_{R \otimes R} M \to M^\cwedge_\ell
\]
is indeed an equivalence: This is clear for $M=R$, which implies the general case since $M$ is a finitely generated projective $R$-module.
Conditions \reftwoitem{item:S-compatible-M} and \reftwoitem{item:both-left-ore} are immediate from the fact that \(\ell\) is in the image of the unit map \(\ZZ \to R\), and \reftwoitem{item:right-mult-equiv} follows from the fact that the square
\[
\begin{tikzcd}
R \ar[r] \ar[d] & R^\cwedge_\ell \ar[d] \\
R[\tfrac{1}{\ell}] \ar[r] & R^\cwedge_\ell[\tfrac{1}{\ell}]
\end{tikzcd}
\]
is a derived pullback, see for instance \cite{DwyerGreenlees}*{\S 4}, as the assumption on $\ell^\infty$-torsion implies that $R^\cwedge_\ell$ is also a derived completion. 
The final thing to check is that the map $M^\cwedge_\ell[\tfrac{1}{\ell}] \to M[1]$ induces the zero map in $\Ct$-Tate cohomology. This follows from the fact that the domain is a $\QQ$-vector space, and so has trivial $\Ct$-Tate cohomology.
\end{proof}

To make efficient use of the localisation-completion square, we shall also need the following result due to Wall \cite{Wall-hensel}*{Lemma 5}. We include a guide through the proof merely for convenience of the reader, as to avoid confusion about different definitions (and versions) of L-theory. We warn the reader that what is denoted by $L_i^K(R)$ in \cite{Wall-hensel} is what we would denote $\L(\Dfree(R),\QF^\qdr)$, i.e.\ quadratic L-theory based on complexes of (stably) free modules.
\begin{proposition}
\label{proposition:hensel}%
Let $R$ be ring, complete in the $I$-adic topology for an ideal $I$ of $R$. Then the canonical map $\L^\qdr(R) \to \L^\qdr(R/I)$ is an equivalence.
\end{proposition}
\begin{proof}
First, we claim that the functor $\Unimod^{\qdr}(R;\eps) \to \Unimod^{\qdr}(R/I;\eps)$ induces a bijection on isomorphism classes, for $\eps=\pm1$. To see this, we first observe that the functor $\Proj(R) \to \Proj(R/I)$ is full and essentially surjective. Moreover, for any finitely generated projective $R$ module $P$, the map 
\[
\QF^{\qdr}_\eps(P) \to \QF^{\qdr}_{\eps}(P\otimes^\mathrm{U}_R R/I)
\]
is surjective on $\pi_0$, and furthermore an $\eps$-quadratic form is unimodular if and only if its image over $R/I$ is; this is as a consequence of Nakayama's lemma: $I$ is contained in the Jacobson radical as we have assumed that $R$ is $I$-complete.
We deduce that the quotient map \(R \to R/I\) induces a surjective map on isomorphism classes of quadratic forms. %
To see injectivity, we apply \cite{Wall-integers}*{Theorem 2}: amongst other things, it says that given forms $(P,q)$ and $(P',q')$ over $R$, then any isometry between their induced forms over $R/I$ can be lifted to an isometry over $R$. In particular, the map $\Unimod^{\qdr}(R;\eps) \to \Unimod^{\qdr}(R/I;\eps)$ is also injective on isomorphism classes. We deduce that the map $\GW_0^{\qdr}(R;\eps) \to \GW_0^{\qdr}(R/I;\eps)$ is an isomorphism. Since likewise the map $\K_0(R) \to \K_0(R/I)$ is an isomorphism, we deduce that $\L_0^{\qdr}(R;\eps) \to \L_0^{\qdr}(R/I;\eps)$ is an isomorphism as well.
We then consider the diagram
\[
\begin{tikzcd}
\pi_1(\K(R;\eps)_{\hC}) \ar[r] \ar[d, two heads] & \GW_1^{\qdr}(R;\eps) \ar[r] \ar[d, two heads] & \L_1^{\qdr}(R;\eps) \ar[r] \ar[d] & \K_0(R;\eps)_{\Ct} \ar[r] \ar[d,"\cong"] & \GW_0^{\qdr}(R;\eps) \ar[d,"\cong"] \\
\pi_1(\K(R/I;\eps)_{\hC}) \ar[r] & \GW_1^{\qdr}(R/I;\eps) \ar[r] & \L_1^{\qdr}(R/I;\eps) \ar[r] & \K_0(R/I;\eps)_{\Ct} \ar[r] & \GW_0^{\qdr}(R/I;\eps)
\end{tikzcd}
\]
where \cite{Wall-hensel}*{Corollary 1 \& Lemma 1} give that the two left most vertical maps are surjective, and \cite{Wall-hensel}*{Proposition 4} that the induced map on vertical kernels is surjective. This implies that the map $\L^{\qdr}_1(R;\eps) \to \L^{\qdr}_1(R/I;\eps)$ is an isomorphism. From the general periodicity $\L_n^{\qdr}(R;\eps) \cong L_{n+2}^{\qdr}(R;-\eps)$ we deduce the proposition.
\end{proof}

\begin{remark}
We thank Akhil Mathew for making us aware of the following result, see \cite{CMM}*{Remark~5.6} for the details. Namely, let $F$ be a functor from commutative rings to spectra which commutes with filtered colimits, and assume that $F$ satisfies the following property.
For every pair $(R,I)$ where $R$ is a Noetherian commutative ring, complete in the $I$-adic topology for an ideal $I \subseteq R$, the canonical map $F(R) \to F(R/I)$ is an equivalence. Then the map $F(S) \to F(S/J)$ is an equivalence for every henselian pair $(S,J)$. 
Since L-theory commutes with filtered colimits, we deduce from this and Proposition~\refthree{proposition:hensel} that the canonical map $\L^\qdr(S) \to \L^\qdr(S/J)$ is an equivalence for any henselian pair $(S,J)$.
\end{remark}

We now apply the above results in order to compute the quadratic \(\L\)-groups of Dedekind rings.

\begin{proposition}
\label{proposition:compute-quadratic-L-groups}%
Let $R$ be a Dedekind ring whose field of fractions $K$ is not of characteristic 2, and let $\mathcal{I}$ be the (finite) set of dyadic primes of $R$. Then we have
\[
\L^\qdr_n(R) \cong 
\begin{cases}
\W^\qdr(R) & \text{ for } n\equiv 0(4) \\ 
0 & \text{ for } n\equiv 1 (4) \\ 
\displaystyle\mathop{\oplus}\limits_{\fp \in \mathcal{I}} \W^\qdr(\FF_\fp) & \text{ for } n \equiv 2 (4) 
\end{cases}
\]
The isomorphism in degrees $n \equiv 2 (4)$ is induced by the canonical maps $R \to \FF_\fp$ for each dyadic prime $\fp$. For $n\equiv 3(4)$ there is a short exact sequence
\[
0 \lto A \lto \L^\qdr_{n}(R) \lto \L^\s_{n}(R) \lto 0
\]
where $A$ is the total cokernel, that is the cokernel of the map induced on cokernels, of the commutative square
\[
\begin{tikzcd}
\L_0^\qdr(R) \ar[r] \ar[d] & \L_0^\s(R) \ar[d] \\
\L_0^\qdr(R^\cwedge_2) \ar[r] & \L_0^\s(R^\cwedge_2)
\end{tikzcd}
\]

\end{proposition}
\begin{proof}
The canonical map $\W^\qdr(R) \to \L_0^\qdr(R)$ is an isomorphism by Corollary~\refthree{corollary:L-zero}. To see the other cases, we consider the cube 
\[
\begin{tikzcd}[row sep=tiny, column sep=tiny]
 & \L^\s(R) \ar[rr] \ar[dd] && \L^\s(R^\cwedge_2) \ar[dd] \\
\L^\qdr(R) \ar[rr]  \ar[ur] \ar[dd] && \L^\qdr(R^\cwedge_2) \ar[ur] \ar[dd] & \\
 & \L^\s(R[\tfrac{1}{2}]) \ar[rr] && \L^\s(R^\cwedge_2[\tfrac{1}{2}]) \\
\L^\qdr(R[\tfrac{1}{2}]) \ar[ur] \ar[rr] && \L^\qdr(R^\cwedge_2[\tfrac{1}{2}]) \ar[ur]
\end{tikzcd}
\]
which is obtained by mapping the quadratic localisation-completion square appearing in Proposition~\refthree{proposition:LC} to the symmetric one. We note that no control terms are needed since localisations of Dedekind rings induce surjections on $\K_0$.
In this cube, the front and back squares are pullbacks by Proposition~\refthree{proposition:LC}, and the bottom square is a pullback since in all rings that appear 2 is invertible. We deduce that the diagram
\begin{equation}
\label{equation:useful-square}%
\begin{tikzcd}
\L^\qdr(R) \ar[r] \ar[d] & \L^\qdr(R^\cwedge_2) \ar[d] \\
\L^\s(R) \ar[r] & \L^\s(R^\cwedge_2) 
\end{tikzcd}
\end{equation}
is also a pullback.

Now all remaining statements to be proven follow from the long exact Mayer-Vietoris sequence associated to this pullback, using the following:
\begin{enumerate}
\item $\L_n^\qdr(R^\cwedge_2) = 0$ for odd $n$, by Proposition~\refthree{proposition:hensel},
\item $\L_n^\s(R^\cwedge_2) = 0$ for $n\equiv 3 (4)$, because $R^\cwedge_2$ is a product of local Dedekind rings; see Remark~\refthree{remark:local-Dedekind},
\item
\label{item:part-three}%
$\L_n^\s(R) = \L_n^\s(R^\cwedge_2) = 0$ for $n\equiv 2 (4)$ by Proposition~\refthree{proposition:compute-L-groups}, and
\item the map $\L_1^\s(R) \to \L_1^\s(R^\cwedge_2)$ is an isomorphism. This can be seen from the localisation-completion square for symmetric L-theory and \refthreeitem{item:part-three}.
\end{enumerate}
\end{proof}

\begin{corollary}
\label{corollary:compute-quadratic-L-groups-global-field}%
Under the assumptions of Proposition~\refthree{proposition:compute-quadratic-L-groups}, assume in addition that $K$ is a number field and let $d= |\mathcal{I}|$ be the (finite) number of dyadic primes of \(R\). Then we have
\[
\L^\qdr_n(R) \cong 
\begin{cases}
\W^\qdr(R) & \text{ for } n \equiv 0 (4) \\ 
0 & \text{ for } n\equiv 1(4) \\
(\ZZ/2)^d & \text{ for } n\equiv 2(4)
\end{cases}
\]
The invariants in the case $n\equiv 2(4)$ are given by the Arf invariants of the images in the L-theory of $\FF_\fp$ for each dyadic prime $\fp$. Moreover, there is an exact sequence
\[
0 \lto A \lto \L^\qdr_{-1}(R) \lto \Pic(R)/2 \lto 0
\]
where $A$ is as in Proposition~\refthree{proposition:compute-quadratic-L-groups} and is a finite 2-group.
\end{corollary}
\begin{proof}
First we recall from Corollary~\refthree{corollary:compute-L-groups-global-field} that for $n\equiv 3(4)$, we have $\L^\s(R) \cong \Pic(R)/2$, and that $A$ is a quotient of $\L_0^\s(R^\cwedge_2)$. We have $(2) = (\fp_1^{e_1} \cdot \dots \cdot \fp_k^{e_k})$ for some numbers $e_i$, where the $\fp_i$ are the dyadic primes. It follows that there is an isomorphism
\[
\L_0^\s(R^\cwedge_2) \cong \prod\limits_{i=1}^k \L_0^\s(R^\cwedge_{\fp_i}).
\]
It thus suffices to recall that 
\begin{enumerate}
\item the map $\L_0^\s(R^\cwedge_{\fp_i}) \to \L_0^\s(R^\cwedge_{\fp_i}[\tfrac{1}{2}])$ is injective; see the proof of Proposition~\refthree{proposition:compute-L-groups}, and that
\item $\L_0^\s(R^\cwedge_{\fp_i}[\tfrac{1}{2}])$ is a finite 2-group: The fraction field $R^\cwedge_{\fp_i}[\tfrac{1}{2}]$ of $R^\cwedge_{\fp_i}$ is a finite extension of $\QQ_2$, so we may appeal to \cite{Lam}*{Theorem VI 2.29}.
\end{enumerate}
Finally, we note that the residue fields $\FF_\fp$ are finite fields of characteristic 2, so that the Arf invariant provides an isomorphism $\W^\qdr(\FF_\fp) \cong \ZZ/2$.
\end{proof}

\begin{remark}
\label{remark:two}%
As in the symmetric case, we recall that there is a canonical equivalence $\L^{-\qdr}(R) \simeq \Sigma^2\L^{\qdr}(R)$, so that Proposition~\refthree{proposition:compute-quadratic-L-groups} and Corollary~\refthree{corollary:compute-quadratic-L-groups-global-field} also determine the $(-1)$-quadratic L-groups.
\end{remark}

\begin{remark}
\label{remark:A-nontrivial}%
If the number $d$ of dyadic primes of $R$ is at least 2, then $A$ is not trivial: Taking the rank mod 2 induces the right horizontal surjections in the following diagram.
\[
\begin{tikzcd}
\L_0^\qdr(R) \ar[r] \ar[d] & \L_0^\s(R) \ar[d]  \ar[r, two heads] & \ZZ/2 \ar[d] \\
\L_0^\qdr(R^\cwedge_2) \ar[r] & \L_0^\s(R^\cwedge_2) \ar[r, two heads] & (\ZZ/2)^d
\end{tikzcd}
\]
Both horizontal composites are zero, therefore we obtain a commutative diagram
\[
\begin{tikzcd}
\coker(\L_0^\qdr(R) \to \L_0^\s(R)) \ar[r] \ar[d] & \ZZ/2 \ar[d] \\
\coker(\L_0^\qdr(R^\cwedge_2) \to \L_0^s(R^\cwedge_2)) \ar[r] & (\ZZ/2)^d
\end{tikzcd}
\]
whose horizontal arrows are surjective. The induced map on vertical cokernels is a map $A \to (\ZZ/2)^{d-1}$ which is therefore again surjective.
\end{remark}

\begin{example}
\label{example:L-of-integers}%
Let us consider the case $R= \ZZ$. From the pullback diagram \eqrefthree{equation:useful-square}, we obtain an exact sequence
\[
0 \lto \L_0^\qdr(\ZZ) \stackrel{(0,8)}{\lto} \L_0^\qdr(\ZZ^\cwedge_2) \oplus \L_0^\s(\ZZ) \lto \L_0^\s(\ZZ^\cwedge_2) \lto \L_{-1}^\qdr(\ZZ) \lto 0,
\]
where the map $\L_0^\qdr(\ZZ) \to \L_0^\qdr(\ZZ^\cwedge_2)$ is the zero map: By Proposition~\refthree{proposition:hensel}, it suffices to know that the map $\L_0^\qdr(\ZZ) \to \L_0^\qdr(\FF_2)$ is the zero map. For this, one calculates that the Arf invariant of the $E_8$-form (viewed as a form over $\FF_2$) is zero. Furthermore, the map $\L_0^\qdr(\ZZ) \to \L_0^\s(\ZZ)$ is isomorphic to multiplication by $8$, as the $E_8$ form generates $\L_0^\qdr(\ZZ)$. We therefore obtain a short exact sequence
\begin{equation}
\label{equation:useful-sequence}%
0 \lto \ZZ/2 \oplus \ZZ/8 \lto \L_0^\s(\ZZ^\cwedge_2) \lto \L_{-1}^\qdr(\ZZ) \lto 0.
\end{equation}
Furthermore, by localisation-dévissage, there is a short exact sequence
\[
0 \lto \L_0^\s(\ZZ^\cwedge_2) \lto \L_0^\s(\QQ_2) \lto \ZZ/2 \lto 0
\]
and from \cite{Lam}*{VI Theorem 2.29 \& Corollary 2.23}, we know that $\L_0^\s(\QQ_2)$ has 32 elements. We deduce that $\L_0^\s(\ZZ^\cwedge_2)$ has 16 elements, and hence the above injection $\ZZ/2\oplus \ZZ/8 \subseteq \L_0^\s(\ZZ^\cwedge_2)$ is an isomorphism. For completeness, we observe that the exact sequence involving $\L_0^\s(\QQ_2)$ splits, so one obtains the well known isomorphism $\L_0^\s(\QQ_2) \cong (\ZZ/2)^2 \oplus \ZZ/8$ \cite{Lam}*{VI Theorem 2.29}. A concrete splitting is given by the element $\langle -1,2 \rangle$. The only thing that needs checking is that this element has order 2.

From the above and the exact sequence \eqrefthree{equation:useful-sequence}, we find that $\L_{-1}^\qdr(\ZZ) = 0$. In particular, we obtain the well known calculations of the symmetric and quadratic L-groups of $\ZZ$:
\[
\L_n^\s(\ZZ) \cong
\begin{cases}
\ZZ & \text{ for } n\equiv 0 (4) \\
\ZZ/2 & \text{ for } n\equiv 1 (4) \\ 
0 & \text{ for } n \equiv 2 (4) \\ 
0 & \text{ for } n \equiv 3 (4)
\end{cases}
\quad \quad
\L_n^\qdr(\ZZ) \cong 
\begin{cases}
\ZZ & \text{ for } n \equiv 0 (4) \\
0 & \text{ for } n \equiv 1 (4) \\
\ZZ/2 & \text{ for } n \equiv 2 (4) \\
0 & \text{ for } n \equiv 3 (4)
\end{cases}
\]
Together with Theorem~\refthree{theorem:main-theorem-L-theory}, Corollary~\refthree{corollary:dedekind}, and Remark~\refthree{remark:improve-two} this determines $\L^\gs_n(\ZZ)$.
In addition, we find that the map $\L^\gs(\ZZ)[\tfrac{1}{2}] \to \L^\s(\ZZ)[\tfrac{1}{2}]$ is an equivalence. We will make use of this fact in Proposition~\refthree{proposition:l-groups-two-inverted-outside}.
\end{example}

\begin{example}
\label{example:Eisenstein-integers}%
Consider the quadratic extension $K=\QQ[\sqrt{-3}]$ of $\QQ$ and let $R$ be its ring of integers. Concretely, $R$ is the ring of Eisenstein integers $R= \ZZ[\tfrac{1+\sqrt{-3}}{2}]$, which is a euclidean domain and hence a principal ideal domain. The discriminant of $K$ is $(3)$, and as $(2)$ does not divide $(3)$, we deduce that $(2)$ is a prime ideal in $R$ \cite{Neukirch}*{Corollary III.2.12}, and hence is the single dyadic prime. We deduce that $\L_2^\s(R) = \L_3^\s(R) = 0$, as the Picard group of a principal ideal domain vanishes. Furthermore $\L_1^\s(R) \cong \ZZ/2$ and $\L_0^\s(R) \cong \W_0^\s(R) \cong \ZZ/4$ \cite{milnor-symmetric}*{Corollary 4.2}. To calculate the quadratic L-groups we consider the diagram of exact sequences
\[
\begin{tikzcd} 
0 \ar[r] & \L_0^\qdr(\ZZ) \ar[r,"{(8,0)}"] \ar[d,"0"] & \ZZ \oplus \ZZ/2 \ar[r] \ar[d,"{(\pr,\id)}"] & \L_0^\s(\ZZ^\cwedge_2) \ar[r] \ar[d,"\theta"] & 0 \ar[r] \ar[d] & 0 \\
0 \ar[r] & \L_0^\qdr(R) \ar[r] & \ZZ/4 \oplus \ZZ/2 \ar[r] & \L_0^\s(R^\cwedge_2) \ar[r] & A \ar[r] & 0 
\end{tikzcd}
\]
and deduce that $A \cong \coker(\theta)$ and that there is an exact sequence
\[
0 \lto \ZZ/2 \lto \ker(\theta) \lto \L_{0}^\qdr(R) \lto 0.
\]
Now, from the commutative diagram of localisation-dévissage sequences (note that 2 is a uniformiser in both cases)
\[
\begin{tikzcd}
\L^\s(\ZZ^\cwedge_2) \ar[r] \ar[d] & \L^\s(\QQ_2) \ar[r] \ar[d] & \L^\s(\ZZ/(2)) \ar[d] \\
\L^\s(R^\cwedge_2) \ar[r] & \L^\s(K^\cwedge_2) \ar[r] & \L^\s(R/(2))
\end{tikzcd}
\]
we deduce that the kernel and the cokernel of $\theta$ are respectively isomorphic to the kernel and the cokernel of the map
\[
\theta'\colon \L^\s_0(\QQ_2) \lto \L^\s_0(K^\cwedge_2).
\]
It is a general theorem about quadratic extensions of fields that the kernel of $\theta'$ is, as an ideal, generated by the element $\langle 1,3 \rangle$, \cite{Lam}*{VII Theorem 3.5}. Since $-5/3$ is a square in $\QQ_2$, we deduce that $\langle 1,3 \rangle = \langle 1,-5\rangle$. From \cite{Lam}*{VI Remark 2.31}, we then deduce that the kernel of $\theta'$ is spanned by $4\langle 1 \rangle$ and $\langle 1,3 \rangle$ and thus isomorphic to $(\ZZ/2)^2$. We deduce that $\L_0^\qdr(R) \cong \ZZ/2$. From \cite{Lam}*{VII Theorem 3.5}, we also find that 
\[
\coker(\theta') \cong \ker\big(  \L_0^\s(\QQ_2) \stackrel{\cdot \langle 1,3 \rangle}{\lto} \L_0^\s(\QQ_2)\big)
\]
and again from \cite{Lam}*{VI Remark 2.31}, we find that the kernel of $\cdot \langle 1,3 \rangle$ is additively generated by $2\langle 1 \rangle$ and $\langle 1,-2 \rangle$, and deduce an isomorphism
\[
\ker\big(  \L_0^\s(\QQ_2) \stackrel{\cdot \langle 1,3 \rangle}{\lto} \L_0^\s(\QQ_2)\big)  \cong \ZZ/4 \oplus \ZZ/2.
\]
In summary, we obtain the following L-groups for $R$:
\[
\L_n^\s(R) \cong
\begin{cases}
\ZZ/4 & \text{ for } n\equiv 0 (4) \\ 
\ZZ/2 & \text{ for } n\equiv 1 (4) \\ 
0 & \text{ for } n \equiv 2 (4) \\ 
0 & \text{ for } n \equiv 3 (4)
\end{cases}
\quad \quad \quad
\L_n^\qdr(R) \cong 
\begin{cases}
\ZZ/2 & \text{ for } n \equiv 0 (4) \\
0 & \text{ for } n \equiv 1 (4) \\
\ZZ/2 & \text{ for } n \equiv 2 (4) \\
\ZZ/4 \oplus \ZZ/2 & \text{ for } n \equiv 3 (4)
\end{cases}
\]
\end{example}

\begin{remark}
\label{remark:devissage-in-quadratic-L}%
Let $R$ be a Dedekind ring and $S$ a set of primes not containing a dyadic one. We note that in this case, $R$ and $R_S$ have the same 2-adic completions, i.e.\ the canonical map $R^\cwedge_2 \to (R_S)^\cwedge_2$ is an isomorphism. We then consider the following diagram
\[
\begin{tikzcd}
\L^\qdr(R) \ar[r] \ar[d] & \L^\qdr(R_S) \ar[r] \ar[d] & \L^\qdr(R^\cwedge_2) \ar[d] \\
\L^\s(R) \ar[r] & \L^\s(R_S) \ar[r] & \L^\s(R^\cwedge_2) .
\end{tikzcd}
\]
We have seen in the proof of Proposition~\refthree{proposition:compute-quadratic-L-groups} that the big and the right squares are  pullbacks. Therefore, so is the left square. Using that for any non-dyadic prime $\fp$ of $R$, the residue field $\FF_\fp$ is a field of odd characteristic, so that its quadratic and symmetric L-theories agree, we deduce from Corollary~\refthree{corollary:decomposition-s} that there is a fibre sequence
\[
\L^\qdr(R) \lto \L^\qdr(R_S) \lto \displaystyle\mathop{\oplus}\limits_{\fp \in S} \L^\qdr(\FF_{\fp})
\]
so that the failure of dévissage in quadratic L-theory is related only to dyadic primes, as indicated in Remark~\refthree{remark:no-quadratic-devissage}.
\end{remark}

\begin{corollary}
\label{corollary:finite-generation}%
Let $\nO$ be a number ring, that is, a localisation of the rings of integers in a number field away from finitely many primes, and $\epsilon=\pm1$. Then the $\eps$-symmetric L-groups $\L^{\s}_n(\nO;\eps)$ and the $\eps$-quadratic L-groups $\L^\qdr_n(\nO;\eps)$ are finitely generated.
\end{corollary}
\begin{proof}
It follows from Proposition~\refthree{proposition:compute-L-groups} and Corollary~\refthree{corollary:compute-quadratic-L-groups-global-field} and Remarks~\refthree{remark:one} and \refthree{remark:two} that it suffices to show that the symmetric Witt group $\W^{\s}(\nO)$, the quadratic Witt group $\W^\qdr(\nO)$, and the Picard group $\Pic(\nO)$ are finitely generated. The statement for the symmetric Witt group is proven in \cite{milnor-symmetric}*{\S 4, Theorem 4.1}, and in fact $\W^\s(\nO)$ is an extension of a finite group by a free abelian group of rank given by the number of real embeddings of the number field $F$. Now we claim that generally for a Dedekind ring $R$ whose fraction field $K$ is of characteristic different from 2, the canonical map $\W^\qdr(R) \to \W^\s(R)$ is injective, so that $\W^\qdr(R)$ is finitely generated if $\W^\s(R)$ is. This follows from the fact that the map $\W^\qdr(R) \to \W^\qdr(K)$ is injective, see \cite{Knebusch-Scharlau}. As the argument in loc.\ cit.\ is not explicitly written out, let us sketch a direct argument that the map $\W^\qdr(R) \to \W^\s(R)$ is injective: First, assume that a symmetric form $(P,\varphi)$ vanishes in $\W^\s(R)$. Then the same is true for its image in $\W^\s(K)$. By Corollary~\refthree{corollary:d-zero} we deduce that $(P\otimes_R K,\varphi\otimes_R K)$ admits a strict Lagrangian. The argument written in the proof of \cite{Knebusch-Scharlau}*{Lemma 1.4} then shows that  $(P,\varphi)$ indeed itself admits a strict Lagrangian. Now, let $(P,q)$ be a quadratic form whose image in $\W^\s(R)$ vanishes. We deduce that the underlying symmetric bilinear form of $(P,q)$ admits a strict Lagrangian $L$. We then observe that for each $x$ in $L$, we have $2q(x) = b(x,x) = 0$, so that $q_{|L} = 0$ as $R$ is 2-torsion free. It follows that $L$ is a Lagrangian for the quadratic form $(P,q)$ as needed.
Finally, the Picard group of the ring of integers in a number field is finite (in other words, the class number of a ring of integers is finite), and hence the Picard group of a localisation of such a ring receives a surjection from a finite group and is thus itself finite, compare to the proof of
Proposition~\refthree{proposition:verdier}. 
\end{proof}

\begin{remark}
In the above proof, we have again restricted our attention to Dedekind rings whose field of fractions $K$ has characteristic different from 2. If the characteristic of $K$ is 2 we find that:
\begin{enumerate}
\item
\label{item:r-one}%
The map $\W^\s(R) \to \W^\s(K)$ is injective, but
\item
\label{item:r-two}%
the map $\W^\qdr(R) \to \W^\s(R)$ is zero.
\end{enumerate}
Indeed \refthreeitem{item:r-one} follows from the same argument given above, since also for fields $K$ of characteristic 2 a form $(P,b)$ is zero in $\W^\s(K)$ if and only if it admits a strict Lagrangian, see Corollary~\refthree{corollary:d-zero}.
To see \refthreeitem{item:r-two}, it suffices to show that the composite $\W^\qdr(R) \to \W^\s(R) \to \W^\s(K)$ vanishes, as the latter map is injective, see the proof of Corollary~\refthree{corollary:compute-L-groups-global-Dedekind-ring-characteristic-2}. This composite factors through the map $\W^\qdr(K) \to \W^\s(K)$ which is zero as the underlying bilinear form of any quadratic form over a field of characteristic 2 has a symplectic basis and hence admits a Lagrangian.
\end{remark}

\begin{corollary}
\label{corollary:finite-generation-GWm}%
Let $\nO$ be a number ring and $\eps=\pm 1$. Then for all $m,n \in \ZZ$, the groups $\L_n(\nO;\Qgen{m}{\eps})$, and consequently the groups $\GW_n(\nO;\Qgen{m}{\eps})$, are finitely generated. 
\end{corollary}
\begin{proof} We saw in Example~\refthree{example:symm-quad-qf} that the functor $\Qgen{m}{\eps}$ is $m$-quadratic and $(2-m)$-symmetric. Hence, on the one hand, it follows from Corollary~\refthree{corollary:genuine-is-quadratic} that for $n \leq 2m-2$ the map $\L_n^\qdr(\nO;\eps) \to \L_n(\nO;\Qgen{m}{\eps})$ is surjective. By Corollary~\refthree{corollary:finite-generation} the left hand group is finitely generated, so the same is true for $\L_n(\nO;\Qgen{m}{\eps})$.

On the other hand, Corollary~\refthree{corollary:genuine-is-symmetric} implies that the map $\L_n(\nO;\Qgen{m}{\eps}) \to \L_n^\s(\nO;\eps)$ is injective for $n \geq 2m-1$. Again, by Corollary~\refthree{corollary:finite-generation} the target group is finitely generated, it follows that $\L_n(\nO;\Qgen{m}{\eps})$ is so as well. To obtain the consequences for Grothendieck-Witt groups, we recall from Quillen's results that the algebraic K-groups of number rings are finitely generated \cite{quillen-fg}. From the homotopy orbits spectral sequence, it follows that also the homotopy groups of $\K(\nO;\eps)_{\hC}$ are finitely generated, so that the desired result follows from the fibre sequence
\[
\K(\nO;\eps)_{\hC} \lto \GW(\nO;\Qgen{m}{\eps}) \lto \L(\nO;\Qgen{m}{\eps}).
\]
\end{proof}

Combining the above with the comparison theorem of \cite{comparison} gives the following corollary.
\begin{corollary}
\label{corollary:finite-generation-GW}%
Let  $\nO$ be a number ring, and $\eps=\pm1$. Then the classical $\epsilon$-symmetric and $\eps$-quadratic Grothendieck-Witt groups $\GW^{\s}_{\cl,n}(\nO;\eps)$ and $\GW^{\qdr}_{\cl,n}(\nO;\eps)$ are finitely generated for all $n \geq 0$.
\end{corollary}

\begin{remark}
\label{remark:homological-stability}%
The finite generation of the groups $\GW^{\qdr}(\nO;\eps)$ can also be deduced by a homological stability argument similar to the one of Quillen for algebraic $\K$-theory of number rings: By Serre class theory, it suffices to show that the ordinary homology groups of the components of $\Omega^\infty \GW^{\qdr}(\nO;\eps)$ are finitely generated. Since every $\eps$-quadratic form is a direct summand in an $\eps$-hyperbolic form, the group completion theorem identifies any such component with the space 
\[
\begin{cases}
\BO_{\infty,\infty}(\nO)^+ & \text{ for } \eps = 1, \\
\BSp^\qdr_\infty(\nO)^+ & \text{ for } \eps = -1
\end{cases}
\]
where $\rO_{\infty,\infty}(\nO)$ and $\Sp^\qdr_\infty(\nO)$ denote the colimit of the automorphism group of an $n$-fold sum of the $(1)$- and $(-1)$-quadratic hyperbolic form, respectively. Charney \cite{Charney} has proved a homological stability result for those groups, so that it suffices to show that the groups $\rO_{n,n}(\nO)$ and $\Sp^\qdr_n(\nO)$ have finitely generated homology. Let us briefly explain why that is: First we note that both groups are arithmetic. Second, every arithmetic group has a torsion free finite index subgroup \cite{Serre-arithmetic-groups}*{1.3 (4)}, and hence also a \emph{normal} torsion free finite index subgroup. By the Serre spectral sequence for the quotient by this normal subgroup, we find that it suffices to know that torsion free arithmetic groups have finitely generated homology, which follows from the fact they they admit a finite classifying space \cite{Serre-arithmetic-groups}*{1.3 (5)}. We wish to thank Manuel Krannich for a helpful discussion about this and for making us aware of Serre's survey.

Finally, we note that in the symmetric case, it is not generally true that every form embeds into a hyperbolic form (as any such form admits a quadratic refinement), so in order to run a similar argument one first needs to find a symmetric bilinear form $b$ such that every other form embeds into a suitable number of orthogonal copies of $b$, and one needs to prove homological stability for the family of automorphism groups of such orthogonal copies of $b$. To our knowledge, this is not known to hold in the generality of number rings, though it does hold for the integers: In the case of symplectic forms, i.e.\ $\epsilon = -1$, this is again due to Charney, and in the case of symmetric forms, this is shown in Nadig's PhD thesis.
\end{remark}

\section{Grothendieck-Witt groups of Dedekind rings}
\label{section:HLP}%
In this final section we consider the homotopy limit problem for Dedekind domains and finite fields of characteristic \(2\). In the latter case, we extend the solution of the homotopy limit problem from the Grothendieck-Witt space (where it is known to hold by the work of Friedlander) to its Grothendieck-Witt spectrum. We then combine this with the dévissage results of \S\refthree{subsection:homotopy-limit} to solve the homotopy limit problem for Dedekind rings whose fraction field is a global field of characteristic $0$, i.e.\ a number field, proving Theorem~\refthree{theorem:homotopy-limit-intro-three} from the introduction. Finally, we apply these ideas to the particular case of \(\ZZ\) and calculate its \(\pm 1\)-symmetric and genuine \(\pm 1\)-quadratic Grothendieck-Witt groups conditionally on Vandiver's conjecture, and in the range \(n \leq 20000\) unconditionally. 

\subsection{The homotopy limit problem}
\label{subsection:homotopy-limit}%

A prominent question in the hermitian \(\K\)-theory of rings and schemes is when the map from the Grothendieck-Witt space/spectrum to the homotopy fixed points of the associated algebraic \(\K\)-theory space/spectrum is an equivalence. This question, first raised by Thomason in~\cite{thomason}, is commonly known as the \emph{homotopy limit problem}. In the case of fields, the following theorem represents the current state of the art; see \cite{HKO-homotopy-limit, BKSO-fixed-point, bachmannperiodic}. We recall that the virtual mod 2 cohomological dimension $\vcd_2$ of a field $k$ can be defined as the ordinary mod 2 cohomological dimension $\cd_2$ of (the absolute Galois group of) $k[\sqrt{-1}]$. In particular, we have $\vcd_2(k) \leq \cd_2(k)$. Given $\eps=\pm1$ we let $\K(k;\eps)$ denote the K-theory spectrum of $k$ with $\Ct$-action induced by the duality $\Dual=\hom_{k}(-,k(\eps))$.

\begin{theorem}
\label{theorem:homotopy-limit}%
Let \(k\) be a field of characteristic different from \(2\) and such that $\vcd_2(k) < \infty$.
Then the map of spectra
\begin{equation*}
\GW^{\s}(k;\eps) \lto \K(k;\eps)^{\hC}
\end{equation*}
is an equivalence after \(2\)-completion. 
\end{theorem}

\begin{remark}
The cited Theorem~\refthree{theorem:homotopy-limit} was stated in \cite{BKSO-fixed-point} using Schlichting's model for Grothendieck-Witt spectra. Since \(k\) is assumed to have characteristic \(\neq 2\) we may invoke the comparison statement of Proposition~\reftwo{proposition:comp-schlichting-space} and identify Schlichting's construction with ours.
\end{remark}

The characteristic \(0\) case of Theorem~\refthree{theorem:homotopy-limit} was proven in~\cite{HKO-homotopy-limit}, while the positive odd characteristic case is established in~\cite{BKSO-fixed-point}. An alternative proof of this theorem is also provided in recent work of Bachmann and Hopkins \cite{bachmannperiodic}.
Special cases of the above theorem were already known before: the case of the field \(\CC\) of complex numbers, for example, can be reduced to the classical equivalence \(\BO \simeq \BU^{\hC}\) 
see, e.g.,~\cite{berrick-karoubi}*{Lemma 7.3}. In fact, in loc.\ cit.\ the authors prove this also for the \((-1)\)-symmetric variant. The equivalence for \(\CC\) can in turn be used to deduce the same for finite fields $\FF_q$. One can express the Grothendieck-Witt spaces of $\FF_q$ in terms of the Adams operations on $\BO$ and $\BSp$ in a way analogous to the main results of Quillen's famous paper~\cite{quillen-finite-fields} on the algebraic \(\K\)-theory of $\FF_q$.
These results were first established by Friedlander in~\cite{friedlander-computations}, and later expanded and refined in see~\cite{FP-classical} (where also a small mistake was corrected in the case of \(q\) even and quadratic forms: Friedlander computed $\pi_1(\GW_\cl^\qdr(\FF_q))$ to be trivial, but in fact it is isomorphic to $\ZZ/2$). Combined with the positive solution of the homotopy limit problem for \(\CC\) they imply the following. 

\begin{theorem}
\label{theorem:HLP-for-finite-fields}%
For \(\eps = \pm1\) and every prime power \(q\) the natural map
\[
\GW_\cl^{\s}(\FF_q;\eps) \lto \K(\FF_q; \eps)^{\hC}
\]
is an equivalence on connective covers.
\end{theorem}

The results of \cite{friedlander-computations} and ~\cite{FP-classical} on which this approach relies use lengthy  computations in the cohomology of various finite matrix groups. We shall now present an alternative and significantly shorter proof of the Theorem~\refthree{theorem:HLP-for-finite-fields} in the case of \(q\) even, using Theorem~\refthree{theorem:fiber-sequence-intro-three} from the introduction. We recall that $\GW_\cl^{\s}(\FF_q;\eps) \to \GW^s(\FF_q;\eps)$ is an equivalence on connective covers, Corollary~\refthree{corollary:dedekind-gw}, so it suffices to prove the following proposition, covering not only
the Grothendieck-Witt space, but also the corresponding spectrum, and which applies to arbitrary shifts of the symmetric Poincaré structure:

\begin{proposition}
\label{proposition:holim-char-two-s}%
Let $q=2^r$ for some positive integer $r$. Then the map of spectra
\[
\GW(\FF_q;(\QF^{\s})\qshift{m}) \lto \K(\FF_q;(\QF^{\s})\qshift{m})^{\hC}
\]
is an equivalence for every \(m \in \ZZ\).
\end{proposition}
\begin{proof}
Corollary~\reftwo{corollary:tate-square-L} provides, for every ring $R$ and Poincaré structure $\QF$ on $\Dperf(R)$, a pullback square
\[
\begin{tikzcd}
	\GW(R;\QF) \ar[r] \ar[d] & \L(R;\QF) \ar[d] \\
	\K(R;\QF)^{\hC} \ar[r] & \K(R;\QF)^{\tC},
\end{tikzcd}
\]
 It therefore suffices to show that the canonical map 
\begin{equation}
\label{equation:tate-finite-fields}%
\L(\FF_q;(\QF^{\s})\qshift{m}) \lto  \K(\FF_q;(\QF^{\s})\qshift{m})^{\tC}
\end{equation}
is an equivalence for every \(m\). 
Applying the transformation \(\L(-) \to \K(-)^{\tC}\) to the metabolic sequence of Example~\reftwo{example:metabolicfseq} gives a commutative diagram
\[
\begin{tikzcd}
	\L(\FF_q;(\QF^{\s})\qshift{m}) \ar[r,"\simeq"] \ar[d] & \Sig^m\L(\FF_q;\QF^{\s}) \ar[d] \\
	\K(\FF_q;(\QF^{\s})\qshift{m})^{\tC} \ar[r,"\simeq"] & \Sig^m\K(\FF_q;\QF^{\s})^{\tC}
\end{tikzcd}
\]
whose horizontal arrows are equivalences; see the discussion before Corollary~\refthree{corollary:periodicity-L}. It will thus suffice to treat the case $m=0$. Furthermore, L and Tate of K-theory are 2-periodic, see Corollary~\refthree{corollary:periodicity-L},
and it suffices to check that~\eqrefthree{equation:tate-finite-fields} induces an isomorphism on $\pi_0$ and $\pi_1$. By Corollary~\refthree{corollary:surgery-global-dim-for-rings} we have that \(\L_0(\FF_q;\QF^\sym) \cong \W^\s(\FF_q) \cong \ZZ/2\) is the Witt group of symmetric bilinear forms over \(\FF_q\), which is isomorphic to \(\ZZ/2\) generated by the class of the symmetric bilinear form \((\FF_q,b)\) with \(b(1,1)=1\). On the other hand, the same corollary also gives that \(\L_1(\FF_q;\QF^\sym) = 0\). To finish the proof it will hence suffice to show that \(\pi_1\K(\FF_q;\QF^\sym)^{\tC}=0\), that \(\pi_0\K(\FF_q;\QF^\sym)^{\tC}=\ZZ/2\), and that the map $\L_0(\FF_q;\QF^\sym) \to \pi_0\K(\FF_q;\QF^\sym)^{\tC}$ is non-zero.

Now, by Quillen's calculation of the \(\K\)-theory of finite fields~\cite{quillen-finite-fields}, the K-groups \(\K_\ast(\FF_q)\) are odd torsion groups in positive degrees, so that the map $\K(\FF_q) \to \GEM\ZZ$ is a 2-adic equivalence. It follows that the induced map
$\K(\FF_q;\QF^\s)^{\tC} \to \ZZ^{\tC}$ is an equivalence as well, which shows that the Tate-K-groups are as claimed. 
 
To finish the proof it will hence suffice to show that the map \(\L_0(\FF_q;\QF^\sym) \to \pi_0(\K(\FF_q;\QF^\sym)^{\tC})\) sends the generator to the generator. Indeed, in light of the commutative diagram
\[
\begin{tikzcd}
\pi_0\Poinc(\Dperf(\FF_q),\QF^\sym) \ar[d] \ar[r,two heads] & \GW_0(\FF_q;\QF^\sym) \ar[d] \ar[r, two heads] & \L_0(\FF_q;\QF^\sym) \ar[d] \\
\pi_0\core(\Dperf(\FF_q),\QF^\sym)^{\Ct} \ar[r,two heads] & \K_0(\FF_q;\QF^\sym)^{\Ct} \ar[r,two heads] & \widehat{\rH}^0(\Ct,\K_0(\FF_q;\QF^\sym))
\end{tikzcd}
\]
this simply follows from the fact that the composed forgetful functor
\[
\pi_0\Poinc(\Dperf(\FF_q),\QF^\sym) \lto \pi_0\core(\Dperf(\FF_q),\QF^\sym)^{\Ct} \lto \K_0(\FF_q;\QF^\sym)^{\Ct} \lto \K_0(\FF_q) \cong \ZZ
\]
sends \((\FF_q,b)\) to the generator \(1 \in \K_0(\FF_q)\).
\end{proof}

\begin{remark}
An alternative argument can be given making use of multiplicative structures:
In \paperfour, we prove that the map $\L(R;\QF^{\s}) \to \K(R;\QF^\s)^{\tC}$ is a map of $\Einf$-rings if $R$ is a commutative ring. For $R= \FF_q$ with $q$ even, we then know that both homotopy rings are isomorphic to $\FF_2[x^{\pm1}]$, for $|x|=2$. As any ring endomorphism of this ring is an isomorphism, the map we investigate is an equivalence. 
\end{remark}

\begin{remark}
\label{remark:K-of-perfect-field}%
Suppose $k$ is a perfect field of characteristic $2$.
In this case, the map $\FF_2 \to k$ induces an equivalence on 2-complete $\K$-theory and on $\L$-theory.
For $\K$-theory, this follows from an analysis of Adams operations on $\K(k)$, see~\cite{Hiller}*{Theorem 5.4}, and for L-theory it follows from Remark~\refthree{remark:odd-vanishing} that the odd L-groups vanish in both cases.
By 2-periodicity of L-theory, Corollary~\refthree{corollary:periodicity-L}, and Corollary~\refthree{corollary:dedekind} together with Corollary~\refthree{corollary:L-zero}, it therefore suffices to note that the map $\FF_2 \to k$ induces an isomorphism on symmetric Witt groups.
This in turn follows as every element in the symmetric Witt group of a field is a sum of one-dimensional forms $\langle x \rangle$ for $x\in k^\times$. Since $k$ is perfect, the Frobenius is surjective and hence $\langle x \rangle = \langle y^2 \rangle = \langle 1 \rangle$ showing that the rank mod 2 map $\W^\s(k) \to \ZZ/2$ is an isomorphism for any perfect field of characteristic $k$, including $\FF_2$.
Considering the commutative diagram
\[
\begin{tikzcd}
	\L^\s(\FF_2) \ar[r,"\simeq"] \ar[d,"\simeq"] & \L^\s(k) \ar[d] \\
	\K(\FF_2)^{\tC} \ar[r,"\simeq"] & \K(k)^{\tC}
\end{tikzcd}
\]
we deduce that the homotopy limit problem has an affirmative answer for every perfect field of characteristic $2$.
\end{remark}

The results of Berrick et al.\ ~\cite{BKSO-fixed-point} on the homotopy limit problem extend significantly beyond the realm of fields. It is shown, for example, that for any  Noetherian scheme \(X\) of finite Krull dimension over \(\ZZ[\tfrac{1}{2}]\) such that \(\vcd_2(k(x))\) is uniformly bounded across all points \(x \in X\),  the map
\begin{equation*}
\GW(X) \lto \K(X)^{\hC} 
\end{equation*}
is an equivalence after \(2\)-completion. 
Using the results of the previous sections we can now relax the assumption that \(2\) is invertible from the above result. Recall that for a Dedekind ring $R$ with line bundle $M$ with involution $\pm 1$, the canonical map 
\[
\GW_\cl^{\s}(R;M) \lto \GW(R;\QF^{\s}_M)
\]
is an equivalence in non-negative degrees, by Corollary~\refthree{corollary:dedekind-gw}. Combining this with the following result gives Theorem~\refthree{theorem:homotopy-limit-intro-three} from the introduction.

\begin{theorem}[The homotopy limit problem]
\label{theorem:holim-z}%
Let \(R\) be a Dedekind ring whose fraction field is a number field. Then for every \(m \in \ZZ\) and every line bundle $M$ over $R$ with involution $\pm1$, the map
\[
\GW(R;(\QF^{\s}_M)\qshift{m}) \lto \K(R;(\QF^{\s}_M)\qshift{m})^{\hC}
\]
is a 2-adic equivalence. 
\end{theorem}
\begin{proof}
Let \(S\) be the (finite) set of all prime ideals in \(R\) lying over \(2\). We observe that then $R_S = R[\tfrac{1}{2}]$ and similarly that $M_S = M[\tfrac{1}{2}]$ and consider the commutative diagram
\[
\begin{tikzcd}
	\oplus_{\fp \in S}\GW(\FF_{\fp}; (\QF^{\s}_{M_{\fp}})\qshift{m-1}) \ar[d] \ar[r] & \GW(R; (\QF^{\s}_{M})\qshift{m}) \ar[d] \ar[r] & \displaystyle\GW(R[\tfrac{1}{2}]; (\QF^{\s}_{M_{S}})\qshift{m}) \ar[d] \\ 
	\oplus_{\fp \in S}\K(\FF_{\fp}; (\QF^{\s}_{M_{\fp}})\qshift{m-1})^{\hC}  \ar[r] & \K(R;(\QF^{\s}_{M})\qshift{m})^{\hC}  \ar[r] & \K(R[\tfrac{1}{2}]; (\QF^{\s}_{M_{S}})\qshift{m})^{\hC} 
\end{tikzcd}
\]
obtained via the localisation-dévissage sequences of Corollary~\refthree{corollary:decomposition-s}. Note that we have commuted the homotopy fixed points with the \emph{finite} direct sum in the lower left corner. 
The left most vertical map is an equivalence by Proposition~\refthree{proposition:holim-char-two-s}, and the right most vertical map is a 2-adic equivalence by \cite{BKSO-fixed-point}*{Theorem 2.2}: We need to argue that all residue fields of $R[\tfrac{1}{2}]$ have finite mod 2 virtual cohomological dimension. Indeed, the residue fields at non-zero prime ideals are finite fields and hence have cohomological dimension one (the Galois group is $\widehat{\ZZ}$), and the residue field at $0$ is the fraction field which is number field and hence also has finite $\vcd_2$; \cite{Serre-Galois-cohomology}*{\S II.4.4}. 
It then follows that the middle vertical map is a 2-adic equivalence, as desired. 
\end{proof}

\begin{remark}
The conclusion of Theorem~\refthree{theorem:holim-z} thus holds for all Dedekind rings whose field of fractions is a global field of characteristic different from $2$: In the odd characteristic case \cite{BKSO-fixed-point} applies, and the case of characteristic zero is the content of Theorem~\refthree{theorem:holim-z}.
\end{remark}

\begin{remark}
Suppose again that $R$ is a Dedekind ring with global fraction field $K$. Suppose that $K$ has characteristic different from 2 and is not formally real, that is, that $-1$ is a sum of squares. In other words, suppose that $K$ has positive odd characteristic or is a totally imaginary number field. Then the Witt group $\W^\s(K)$ is a 2-primary torsion group of bounded exponent by \cite{scharlau-quadratic}*{Theorem 2.7.9}. As $\W^\s(R)$ is a subgroup of $\W^\s(K)$, see the proof of Proposition~\refthree{proposition:compute-L-groups},  Corollary~\refthree{corollary:compute-L-groups-global-field} implies that $\L^\s(R)$ is (derived) 2-complete. As $\K(R)^{\tC}$ is also 2-complete, the pullback
\[
\begin{tikzcd}
\GW(R;\QF^\s) \ar[r] \ar[d] & \L(R;\QF^\s) \ar[d] \\
\K(R;\QF^\s)^{\hC} \ar[r] & \K(R;\QF^\s)^{\tC}
\end{tikzcd}
\]
together with Theorem~\refthree{theorem:holim-z} implies that the map of Theorem~\refthree{theorem:holim-z} is in fact an equivalence before 2-completion. Conversely, if $K$ admits a real embedding,  then $\L^\s(R)$ is not 2-complete: We have seen in Corollary~\refthree{corollary:finite-generation} that all homotopy groups are finitely generated, so $\L^\s(R)$ is 2-complete if and only if all symmetric L-groups of $R$ are 2-complete. However, as observed in the proof of Corollary~\refthree{corollary:finite-generation}, $\W^\s(R)$ has rank equal to the number of real embeddings of $K$, and is thus not 2-complete. It hence follows that the map under investigation in Theorem~\refthree{theorem:holim-z} is not an integral equivalence if $K$ admits a real embedding. See also \cite{BKSO-fixed-point}*{Theorem 2.4 \& Proposition 4.7}. In fact, in our situation, the same result is true for $\W^\s(R;M)$ for any line bundle $M$ on $R$: The map $\W^\s(R;M) \to \W^\s(K)$ is an isomorphism after inverting 2, and the map $\W^\s(K) \to \W^\s(\RR)$ induced from a real embedding of $K$ is surjective. Hence the composite is non-zero and consequently $\ZZ$ is a direct summand inside $\W^\s(R;M)$. Hence $\L^\s(R;M)$ is not 2-complete.

As a side remark, we note that in the case where $K$ admits a real embedding, $\L^\s(R)$ contains $\L^s(\RR)$ as a retract, and $\L^\s(\RR)$ is not 2-complete. To see that $\L^\s(\RR)$ is indeed a retract, consider the following composite 
\[
\L^\s(\ZZ) \lto \L^\s(R) \lto \L^\s(\RR)
\]
where the two maps are induced by the canonical map $\ZZ \to R$ and the map $R \to K \subseteq \RR$ induced by a real embedding of $K$. This composite admits a splitting, as was observed in \cite{HLN}*{Theorem A}.
\end{remark}

\begin{remark}
The proof of Theorem~\refthree{theorem:holim-z} reveals that the assumptions are not optimal. Assume for instance that $R$ is the ring of integers in a non-archimedean local field $K$ of mixed characteristic $(0,2)$ and let $k$ be the residue field of the local ring $R$. For instance, assume that $R$ is a dyadic completion of the ring of integers in a number field.
Since $R$ is local the line bundle $M$ is trivial. We again consider the diagram consisting of horizontal fibre sequences
\[
\begin{tikzcd}
	\GW(k;(\QF^{\s})\qshift{m-1}) \ar[r] \ar[d] & \GW(R; (\QF^{\s})\qshift{m}) \ar[r] \ar[d] & \GW(K; (\QF^{\s})\qshift{m}) \ar[d] \\
	\K(k;(\QF^{\s})\qshift{m-1})^{\hC}  \ar[r] & \K(R; (\QF^{\s})\qshift{m})^{\hC}  \ar[r] & \K(K; (\QF^{\s})\qshift{m})^{\hC} 
\end{tikzcd}
\]
First, we note that $\cd_2(K) = 2$ \cite{Serre-Galois-cohomology}*{\S 4.3}, so the right vertical map is a 2-adic equivalence. We deduce that the middle vertical map is a 2-adic equivalence if and only if the left vertical map is a 2-adic equivalence. Thus if we assume that $k$ is a finite field, the middle vertical map is a 2-adic equivalence. In fact, in this case, $K$ is a finite extension of $\QQ_2^\cwedge$, and as observed earlier, $\L^\s(K)$ is 2-complete, in fact 2-power torsion \cite{Lam}*{Theorem VI 2.29}. It follows that the middle vertical map is in fact an equivalence.
\end{remark}

The proof of Theorem~\refthree{theorem:holim-z} allows us to also deduce the following result, which, in case of the integers was conjectured by Berrick and Karoubi \cite{berrick-karoubi}.
\begin{proposition}
\label{proposition:BK-conjecture}%
Let $R$ be a Dedekind ring whose fraction field 
is a global field of characteristic zero. Then the map 
\[
\GW^\s(R;\eps) \lto \GW^\s(R[\tfrac{1}{2}];\eps)
\]
is a 2-local equivalence on connected covers and injective in $\pi_0$.
\end{proposition}
\begin{proof}
The fibre of the map in question is given by a sum of terms of the kind $\GW(\FF_{\mathfrak{p}};(\QF^{\s}_{\eps})\qshift{-1})$, with $\FF_{\mathfrak{p}}$ a finite field of characteristic 2 by Corollary~\refthree{corollary:decomposition-s}. It therefore suffices to show that each of these terms is $2$-locally $(-1)$-truncated and has trivial $\pi_0$. To see this, we note that the map
\[
\GW(\FF_{\mathfrak{p}};(\QF^{\s}_{\eps})\qshift{-1}) \lto \K(\FF_{\mathfrak{p}};(\QF^{\s}_{\eps})\qshift{-1})^{\hC}
\]
is an equivalence by Proposition~\refthree{proposition:holim-char-two-s}. Let us denote by $\ZZ(-1)$ the complex $\ZZ$ in degree 0 with the sign action of $\Ct$. The map $\K(\FF_{\mathfrak{p}};(\QF^{\s}_{\eps})\qshift{-1}) \to \ZZ(-1)$ is a $\Ct$-equivariant map whose fibre has finite and odd torsion homotopy groups. It follows upon applying $(-)^{\hC}$ that this map is 1-connective and a $2$-local equivalence. The proposition follows.
\end{proof}

\begin{remark}
Using Remark~\refthree{remark:K-of-perfect-field}, one obtains the following variant of Proposition~\refthree{proposition:BK-conjecture}. Namely, let $R$ be a Dedekind ring of characteristic zero such that all residue fields of dyadic primes are perfect. Then the map 
\[
\GW^\s(R;\eps)/2 \lto \GW^\s(R[\tfrac{1}{2}];\eps)/2
\]
is $(-1)$-truncated. For this, we simply need to know that for a perfect field $k$ of characteristic 2, $K(k)/2$ is 0-truncated, i.e.\ the higher $\K$-groups $\K_n(k)$, for $n\geq 1$, are uniquely 2-divisible. In general, this of course does not imply that these groups vanish after localisation at $2$, but it is the case for instance for an algebraic closure of $\FF_2$. Whenever the residue fields satisfy the property that $K(k)_{(2)} \simeq \ZZ_{(2)}$, the same argument as in the proof of Proposition~\refthree{proposition:BK-conjecture} applies.

Finally, we note that every perfect field $k$ of positive characteristic $p$ is the residue field of a characteristic 0 Dedekind ring, namely the (complete) discrete valuation ring $W(k)$ of $p$-typical Witt vectors on $k$.
\end{remark}

We continue by noting the following obstruction to a positive solution of the homotopy limit problem for classical Grothendieck-Witt-theory of a discrete ring $R$, see also \cite{BKSO-fixed-point}*{Remark 4.9}.
In particular, Proposition~\refthree{proposition:obstruction-to-HLP} implies that the map $\GW^\s_\cl(R)/2 \to \K(R)^{\hC}/2$ cannot be an equivalence in non-negative degrees unless the comparison map $\L^\gs(R) \to \L^\s(R)$ is so as well. Recall from Example~\refthree{example:dimension-bound-sharp} that there are rings for which this is not the case.

\begin{proposition}
\label{proposition:obstruction-to-HLP}%
Suppose that the map $\GW_\cl^\s(R;M)/2 \to \K(R)^{\hC}/2$ is $n$-truncated for some natural number $n$, where we view $\GW_{\cl}^\s(R;M)/2$ as a (connective) spectrum. Then the map $\L^\gs(R;M) \to \L^\s(R;M)$ is $(n-1)$-truncated.
\end{proposition}
\begin{proof}
By Proposition~\refthree{proposition:l-groups-two-inverted-outside} below, the map $\L^\gs(R;M) \to \L^\s(R;M)$ is an equivalence after inverting $2$. Therefore, once we show that the map $\L^\gs(R;M)/2 \to \L^\s(R;M)/2$ is $n$-truncated, it follows that the map $\L^\gs(R;M) \to \L^\s(R;M)$ is $(n-1)$-truncated as claimed. We first observe that the map $\GW^\gs(R;M)/2 \to \K(R)^{\hC}/2$ is also $n$-truncated, because the map $\GW_{\cl}^\s(R;M)/2 \to \GW^{\gs}(R;M)/2$ is $0$-truncated by \cite{comparison}.
We then consider the pullback diagram
\[
\begin{tikzcd}
	\GW^\gs(R;M)/2 \ar[r] \ar[d] & \L^\gs(R;M)/2 \ar[d] \\
	\K(R)^{\hC}/2 \ar[r] & \K(R)^{\tC}/2 
\end{tikzcd}
\]
and conclude that the map $\L^\gs(R;M)/2 \to \K(R)^{\tC}/2$ is $n$-truncated as well. Then we recall that there are canonical shift maps
\[
\cdots \lto \Sigma^4 \L^{\gs}(R;M) \stackrel{\sigma}{\lto} \L^\gs(R;M) \stackrel{\sigma}{\lto} \Sigma^{-4}\L^\gs(R;M) \lto \cdots
\]
whose filtered colimit is given by $\L^\s(R;M)$. It therefore suffices to show that $\L^\gs(R;M)/2 \to \Sigma^{-4}\L^\gs(R;M)/2$ is $n$-truncated. For this we consider the diagram
\[
\begin{tikzcd}
	\L^\gs(R;M)/2 \ar[r,"\sigma"] \ar[d] & \Sigma^{-4}\L^{\gs}(R;M)/2 \ar[d] \\
	\K(R)^{\tC}/2 \ar[r,"\sigma'"] & \Sigma^{-4}\K(R)^{\tC}/2
\end{tikzcd}
\]
and note that the lower horizontal map is an equivalence. Furthermore, the left vertical map is $n$-truncated and the right vertical map is $(n-4)$-truncated. We deduce that the upper horizontal map is also $n$-truncated.
\end{proof}

We finish this section with the promised calculation of 2-inverted genuine L-theory. At this point, we will invoke multiplicative structures on L-theory which we develop in detail in \paperfour.
\begin{proposition}
\label{proposition:l-groups-two-inverted-outside}%
Let $R$ be a ring with invertible $\ZZ$-module with involution $M$, and let $m \in \ZZ \cup \{  \pm \infty \}$. Then the natural map 
\[
\L(R;\Qgen{m}{M})[\tfrac{1}{2}] \lto \L(R;\QF^\s_M)[\tfrac{1}{2}]
\]
is an equivalence.
\end{proposition}
\begin{proof}
We first observe that the canonical map $\L^\gs(\ZZ) \to \L^\s(\ZZ)$ is an equivalence after inverting $2$, see Example~\refthree{example:L-of-integers}. 
Moreover, the shift maps appearing in the proof of Proposition~\refthree{proposition:obstruction-to-HLP} are in fact given by multiplication with an element $x \in \L_4^\gs(\ZZ)$, namely the Poincaré object $\ZZ[-2]$ with its standard genuine symmetric Poincaré structure of signature 1. Thus
 we find
\[
\L(R;\QF^\s_M) \simeq \L(R;\Qgen{m}{M})[x^{-1}] \simeq \L(R;\Qgen{m}{M}) \otimes_{\L^\gs(\ZZ)} \L^\s(\ZZ),
\]
and the result follows.
\end{proof}

\subsection{Grothendieck-Witt groups of the integers}
\label{subsection:integers}%

In this section, we will specialise the results established earlier in the paper to the ring of integers \(\ZZ\), and calculate its classical \(\eps\)-symmetric and \(\eps\)-quadratic Grothendieck-Witt groups. We will exploit Corollary~\refthree{corollary:dedekind} and instead calculate the non-negative Grothendieck-Witt groups of $\GW^{\s}(\ZZ;\eps)=\GW(\ZZ;\QF^\sym_{\eps})$ for the non-genuine symmetric Poincaré structure. Our calculation crucially relies on the knowledge of the algebraic \(\K\)-groups of \(\ZZ\), resulting from the resolution of the Quillen-Lichtenbaum conjecture completed in the work of Rost and Voevodsky. We will use \cite{weibel-kbook} as our primary reference for this material. The other external input is the calculation of Berrick-Karoubi~\cite{berrick-karoubi} of the Grothendieck-Witt groups of \(\ZZ[\tfrac{1}{2}]\).
Before we start the computation, we give a brief account of the $4$ types of classical Grothendieck-Witt groups that we are considering.

\vspace{10pt}
\noindent\textbf{\((1)\)-Symmetric}:
We recall that $\GW_\cl^\s(\ZZ)$ denotes the homotopy theoretic group completion of the maximal subgroupoid of the category of non-degenerate symmetric bilinear forms over $\ZZ$. By \cite{serre}*{Théorème 1}, there is an isomorphism $\pi_0\GW_\cl^\s(\ZZ)\cong \ZZ \oplus \ZZ$ where the summands are generated by the classes of the forms \(\langle 1 \rangle\) and \(\langle -1 \rangle\) on a free module of rank 1 $\ZZ$; they send $(x,y)$ to $xy$ and $-xy$, respectively. We write 
\(\rO_{\langle n, n \rangle}(\ZZ) = \Aut((\langle 1 \rangle \perp \langle -1 \rangle)^{\perp n}) \subseteq \GL_{2n}(\ZZ)\) and \(\rO_{\langle \infty, \infty \rangle}(\ZZ) = \colim_n \rO_{\langle n, n \rangle}(\ZZ)\).  
Then the commutator subgroup of \(\rO_{\langle \infty, \infty \rangle}(\ZZ)\) is perfect by e.g.~\cite{rw-group-completion}*{Proposition 3.1}. Moreover,
it is a direct consequence of \cite{serre}*{Théorème 4} that any non-degenerate symmetric bilinear form over \(\ZZ\) is an orthogonal summand in \((\langle 1 \rangle \perp \langle -1 \rangle)^{\perp n}\) for some \(n \geq 0\). Hence, the group completion theorem yields a homotopy equivalence of spaces
\[
\tau_{>0}\GW_\cl^{\s}(\ZZ) \simeq \BO_{\langle \infty, \infty \rangle}(\ZZ)^+,
\]
see \cite{mcduff-segal}, or \cite{rw-group-completion}*{Corollary 1.2}.

\vspace{10pt}
\noindent\textbf{\((-1)\)-Symmetric:} 
Similarly $\GW_\cl^{-\s}(\ZZ)$ is the homotopy theoretic group completion of the maximal subgroupoid of the category of non-degenerate symplectic bilinear forms over $\ZZ$.
We let $\Hms$ be the standard symplectic bilinear form on $\ZZ^2$. As every symplectic form over $\ZZ$ is isomorphic to a finite orthogonal sum of copies of $\Hms$, we find $\pi_0\GW^{-\s}_{\cl}(\ZZ) \cong \ZZ$, generated by $\Hms$. We write \(\Sp_{2n}(\ZZ) = \Aut((\Hms)^{\perp n}) \subseteq \GL_{2n}(\ZZ)\) and 
\(\Sp_{ \infty}(\ZZ) = \colim_n \Sp_{2n}(\ZZ)\). The group \(\Sp_{ \infty}(\ZZ)\) is again perfect, see e.g.~\cite{rw-group-completion}*{Proposition 3.1}, and the group completion theorem yields a homotopy equivalence of spaces
\[
\tau_{>0}\GW_\cl^{-\s}(\ZZ) \simeq \BSp_\infty (\ZZ)^+.
\]

\vspace{10pt}
\noindent\textbf{\((1)\)-Quadratic:} 
Now $\GW_\cl^{\qdr}(\ZZ)$ is the homotopy theoretic group completion of the maximal subgroupoid of the category of non-degenerate quadratic forms over $\ZZ$. Let \(\HqZ\) be the standard hyperbolic quadratic form and \(E_8\) the classical \(8\)-dimensional quadratic form associated to the Dynkin diagram of the same name. By~\cite{serre}*{Théorème 5}, every quadratic form \((P,q)\) satisfies \(P \oplus \HqZ \cong \HqZ^n \oplus E_8^m\) for some $n$ and $m$ and $\pi_0\GW_{\cl}^{\qdr}(\ZZ) \cong \ZZ \oplus \ZZ$ with generators $\HqZ$ and $E_8$. We write \(\rO_{n,n}(\ZZ) = \Aut((\HqZ)^{\oplus n}) \subseteq \GL_{2n}(\ZZ)\) and \(\rO_{\infty, \infty}(\ZZ) = \colim_n \rO_{n,n}(\ZZ)\). As above the group \(\rO_{\infty, \infty}(\ZZ)\) has perfect commutator subgroup and since any quadratic form over \(\ZZ\) is a direct summand of \((\HqZ)^{\perp n}\) for some \(n \geq 0\) there is a homotopy equivalence of spaces
\[
\tau_{>0}\GW_{\cl}^{\qdr}(\ZZ) \simeq \BO_{\infty,\infty} (\ZZ)^+.
\]

\vspace{10pt}
\noindent\textbf{\((-1)\)-Quadratic}: 
Finally, $\GW_{\cl}^{-\qdr}(\ZZ)$ is similarly built from $(-1)$-quadratic forms over $\ZZ$. Such a form is determined by its rank (which is an even number) and its Arf invariant, see~\cite{browder}*{\S III.1}. 
Let
\[
\mathrm{H}^0_{-\qdr} = \left(\ZZ^2, \begin{pmatrix} 0 & 1 \\ -1 & 0\end{pmatrix}, x y \right) \quad\text{and}\quad \mathrm{H}^1_{-\qdr} = \left(\ZZ^2, \begin{pmatrix} 0 & 1 \\ -1 & 0\end{pmatrix}, x^2 + xy + y^2 \right),
\]
be the standard hyperbolic \((-1)\)-quadratic forms with Arf invariant \(0\) and \(1\), respectively.
Then every \((-1)\)-quadratic form with Arf invariant \(0\) is isomorphic to a direct sum of copies of \(\mathrm{H}^0_{-\qdr}\), and every \((-1)\)-quadratic form with Arf invariant \(1\) is isomorphic to a direct sum of copies of \(\mathrm{H}^0_{-\qdr}\) plus one copy of \(\mathrm{H}^1_{-\qdr}\). Thus, $\pi_0\GW_{\cl}^{-\qdr} \cong \ZZ \oplus \ZZ/2$. We define \(\Sp^{\qdr}_{2n}(\ZZ) = \Aut((\mathrm{H}^0_{-\qdr})^{\perp n}) \subseteq \Sp_{2n}(\ZZ)\) to be the group of matrices preserving both the bilinear form and its quadratic refinement and set \(\Sp^{\qdr}_{ \infty}(\ZZ) = \colim_n \Sp^{\qdr}_{2n}(\ZZ)\). As above, the group completion theorem yields a homotopy equivalence of spaces
\[
\tau_{>0}\GW_{\cl}^{-\qdr}(\ZZ) \simeq \BSp^{\qdr}_\infty (\ZZ)^+.
\]

\subsubsection*{The Grothendieck-Witt groups of $\ZZ$}
We now proceed to calculate the \(\epsilon\)-symmetric Grothendieck-Witt groups of \(\ZZ\). Recall that the Bernoulli numbers \(\{B_n\}_{n \geq 0}\) are rational numbers determined by the equation
\[
\frac{x}{e^x -1} = \sum_{n=0}^{\infty}\frac{B_n}{n!}x^n.
\]
We write \(c_n\) for the numerator of \(|\tfrac{B_{2n}}{4n}|\) which is an odd number while the denominator of \( |\tfrac{B_{2n}}{4n}|\) will be denoted by \(w_{2n}\). 
For each \(k  \geq 0\) we then have equations
\[
|\K_{8k+2}(\ZZ)| = 2\cdot c_{2k+1} \, \, \text{ and }  \, \, |\K_{8k+6}(\ZZ)| =  c_{2k+2},
\]
and isomorphisms
\[
\K_{8k+3}(\ZZ) \cong  \ZZ/2 w_{4k+2} \, \, \text{ and }  \, \, \K_{8k+7}(\ZZ) \cong  \ZZ/w_{4k+4},
\]
see~\cite{weibel-kbook}*{Theorem 10.1}. 
\begin{remark}
\label{remark:unconditional}%
For an abelian group \(A\), we write \(A_{\odd}\) for the odd torsion subgroup of \(A\). For \(m \leq 5000\) the group \(\K_{4m-2}(\ZZ)_{\odd}\) is known to be cyclic of order \(c_{m}\), see~\cite{weibel-kbook}*{Example 10.3.2}. This holds for all \(m\) if the Kummer-Vandiver conjecture is true, see~\cite{weibel-kbook}*{Theorem 10.2}.
\end{remark}

We now arrive at the main computation of the \(\eps\)-symmetric Grothendieck-Witt groups of \(\ZZ\) in degrees \(n \geq 1\).

\begin{theorem}
\label{theorem:gwsz}%
The classical \(\eps\)-symmetric Grothendieck-Witt groups \(\ZZ\) are given in degrees \(n \geq 1\) by the following table:
\begin{table}[!ht]
\begin{tabular}{|c||c|c|}
\hline
\(n=\) & \(\GW^{\s}_{\cl,n}(\ZZ)\) & \(\GW^{-\s}_{\cl,n}(\ZZ)\) \\ \hline\hline
\(8k\)        & \(\ZZ \oplus \ZZ/2\)                                                           & \(0\)                                                                                  \\ \hline
\(8k+1\)        & \((\ZZ/2)^3\)                                                        & \(0\)                                                                                                  \\ \hline
\(8k+2\)        & \((\ZZ/2)^2 \oplus \K_{8k+2}(\ZZ)_{\odd}\)                                                                    & \(\ZZ \oplus \K_{8k+2}(\ZZ)_{\odd}\)                                                                                                \\ \hline
\(8k+3\)        & \(\ZZ/w_{4k+2}\)                                                            & \(\ZZ/2w_{4k+2}\)                                                                                             \\ \hline
\(8k+4\)        & \(\ZZ\)                                                                                    & \(\ZZ/2\)                                                                                              \\ \hline
\(8k+5\)        & \(0\)                                                                                      & \(\ZZ/2\)                                                                                              \\ \hline
\(8k+6\)        & \(\K_{8k+6}(\ZZ)_{\odd}\)                                                                                      & \(\ZZ\oplus \K_{8k+6}(\ZZ)_{\odd}\)                                                                                                \\ \hline
\(8k+7\)        & \(\ZZ/w_{4k+4}\)                                                 & \(\ZZ/w_{4k+4}\)                                                              \\ \hline
\end{tabular}
\end{table}
\end{theorem}

\begin{remark}
\label{remark:j-homomorphism}%
The number \(w_{2n}\) is equal to the cardinality of the image of the \(J\)-homomorphism \(\pi_{4n-1}(O) \to \pi_{4n-1}(\SS)\) in the stable stem. By~\cite{quillen-letter}*{pg.\ 186} the unit map \(\pi_{4n-1}(\SS) \to \K_{4n-1}(\ZZ)\) is injective on this image. Since the unit map for $\K(\ZZ)$ factors through the unit map for $\GW^{\s}_{\cl}(\ZZ)$, it follows that the groups $\GW^{\s}_{\cl,8k+3}(\ZZ)$ and $\GW^{\s}_{\cl,8k+7}(\ZZ)$ consist precisely of image of $J$-classes.
\end{remark}

\begin{proof}[Proof of Theorem~\refthree{theorem:gwsz}]
Since the groups in question are finitely generated, it suffices to prove that the theorem holds after localisation at 2 and after inverting 2. First, we argue 2-locally. Proposition~\refthree{proposition:BK-conjecture} and Corollary~\refthree{corollary:dedekind-gw} imply that for $\eps = \pm 1$, the canonical map
\[
\GW^{\s}_{\cl}(\ZZ;\eps) \lto \GW^{\s}_\cl(\ZZ[\tfrac{1}{2}];\eps)
\]
is a 2-local equivalence in degrees $\geq 1$. 
One can then compare with \cite{berrick-karoubi}*{Theorem B}\footnote{Note that what we denote $\GW$ is denoted $\mathcal{L}$ in loc.\ cit.\ and that the homotopy groups of $\mathcal{L}(R)$ are denoted by $L_i(R)$.}, where the 2-local $\GW$-groups of $\ZZ[\tfrac{1}{2}]$ are determined as displayed.
In order to compare their values for $8k+3$ and $8k+7$ with ours, note that by work of von Staudt the largest power of \(2\) which divides \(w_{2n}\) is the same as the largest power of \(2\) which divides \(8n\). 

For the 2-inverted case, we observe that the fibre sequence
\[
\K(\ZZ;\eps)_{\hC} \lto \GW^{\s}(\ZZ;\eps) \lto \L^{\s}(\ZZ;\eps)
\]
from Theorem~\refthree{theorem:fiber-sequence-intro-three} splits after inverting 2, see e.g.\ Corollary~\reftwo{corollary:GW-split-one-half}, so that there is an equivalence of spectra
\[
\GW^{\s}(\ZZ;\eps)[\tfrac{1}{2}] \simeq (\K(\ZZ;\eps)_{\hC})[\tfrac{1}{2}] \oplus \L^{\s}(\ZZ;\eps)[\tfrac{1}{2}]
\]
Furthermore, we note that there is an isomorphism $\pi_n(\K(\ZZ;\eps)_{\hC}[\tfrac{1}{2}]) \cong (\K_n(\ZZ;\eps)[\tfrac{1}{2}])_{\Ct}$.
It then follows from Lemma~\refthree{lemma:kz-inv} below that
\[
\GW^{\s}_n(\ZZ;\eps)[\tfrac{1}{2}] 
\cong 
\begin{cases}
\L^{\s}_n(\ZZ;\eps)[\tfrac{1}{2}] & \text{ for } n \equiv 0,1 \text{ mod } 4 \\ 
\K_n(\ZZ;\eps)[\tfrac{1}{2}] \oplus \L^{\s}_n(\ZZ;\eps)[\tfrac{1}{2}] & \text{ for } n \equiv 2,3 \text{ mod } 4 
\end{cases}
\]
This matches with the values in the above table after tensoring with \(\ZZ[\tfrac{1}{2}]\) and so the desired result follows.
\end{proof}

We thank Søren Galatius for telling us about the following lemma, see also \cite{FGV}*{\S 2}. Let $F$ be a number field and $\nO$ a ring of $S$-integers in $F$.
\begin{lemma}
\label{lemma:kz-inv}%
The $\Ct$-actions induced by the Poincaré structures \(\QF^{s}\) and \(\QF^{s}_-\) on \(\Dperf(\nO)\) induce multiplication by \((-1)^{n}\) on the groups \(\K_{2n-1}(\nO)[\tfrac{1}{2}]\) and \(\K_{2n-2}(\nO)[\tfrac{1}{2}]\) for each \(n \geq 2\).
\end{lemma}
\begin{proof}
We first note that the dualities associated to \(\QF^{\s}\) and \(\QF^{\s}_-\) have the same underlying equivalences \(\Dperf(\nO) \to \Dperf(\nO)\op\) so that the induced $\Ct$-action on homotopy groups is the same in both cases. Hence it will suffice to prove the claim for \(\QF^{\s}\). Since the K-groups of $\nO$ are finitely generated, it suffices to prove the claim on the $\ell$-completed K-groups $\K_n(\nO)^\cwedge_\ell$ for all odd primes $\ell$. We then use that the map $\K_n(\nO)^\cwedge_\ell \to \K_n(\nO[\tfrac{1}{\ell}])^\cwedge_\ell \to \pi_n(\K^{\et}(\nO[\tfrac{1}{\ell}])^\cwedge_\ell)$ are isomorphisms for $n \geq 2$; for the first, this follows from Quillen's localisation sequence, and for the second see e.g.\ \cite{CM}*{Theorem~1.2} for a more general statement. In particular, the étale descent spectral sequence, see e.g.\ \cite{CM}*{Theorem~1.3} is a convergent spectral sequence
\[
E_2^{s,n} = \mathrm H^s_{\et}(\spec(\nO[\tfrac{1}{\ell}]);\ZZ^\cwedge_\ell(n/2)) \Longrightarrow \pi_{n-s}(\K(\nO)^\cwedge_\ell)
\]
for $n-s \geq 2$. Here $\ZZ^\cwedge_\ell(n/2)$ is the homotopy sheaf $\pi_{n}(\K(-))$ on the étale site of $\spec(\nO[\tfrac{1}{\ell}])$, which by results of Gabber and Suslin is (as an abelian group) given by $\pi_n(\mathrm{ku}^\cwedge_\ell)$ and hence vanishes for odd values of $n$. We note that the descent spectral sequence is natural with respect to $\psi^{-1}$, i.e.\ the dualisation action, which acts on $\ZZ^\cwedge_\ell(n)$ by $(-1)^{n}$. Now, the $\ell$-adic étale cohomological dimension of $\spec(\nO[\tfrac{1}{\ell}])$ is $2$, so the spectral sequence above is concentrated in the columns $0,1$ and $2$. In addition
\[
\mathrm H^0_\et(\spec(\nO[\tfrac{1}{\ell}]);\ZZ_\ell^\cwedge(n)) = 0 \quad \text{ for } n > 0.
\]
In particular, the spectral sequence collapses at $E_2$ and gives isomorphisms 
\[
\pi_{2n-2}(\K(\nO)) \cong \mathrm H^2_\et(\spec(\nO[\tfrac{1}{\ell}]);\ZZ^\cwedge_\ell(n)) \quad \text{ and } \quad \pi_{2n-1}(\K(\nO)) \cong \mathrm H^1_\et(\spec(\nO[\tfrac{1}{\ell}]);\ZZ^\cwedge_\ell(n))
\]
for $n \geq 2$, from which the lemma follows.
\end{proof}

\begin{remark}
\label{remark:Schlichting-away-from-two}%
A calculation of the Grothendieck-Witt groups of the integers has also been announced by Schlichting. See \cite{Schlichting-integers} and the erratum \cite{Schlichting-wrong}.
\end{remark}

In low degrees the groups can be worked out explicitly.

\begin{proposition}
\label{proposition:gwz-low-deg}%
The first 24 non-negative Grothendieck-Witt groups of \(\ZZ\) are given by the table~\refthree{table:gw} below.

\begin{table}[!ht]
\label{table:gw}%
\caption{The first 24 Grothendieck-Witt groups of \(\ZZ\)}
\begin{tabular}{|c||c|c||c|c||c|}
\hline
\(k\) & \(\GW^{\s}_k(\ZZ)\) & \(k\)    & \(\GW^{s}_k(\ZZ)\)           & \(k\)    & \(\GW^{s}_k(\ZZ)\)           \\ \hline\hline
\(0\) & \(\ZZ \oplus \ZZ\)                                                                       & \(8\)  & \(\ZZ \oplus \ZZ/2\)      & \(16\) & \(\ZZ \oplus \ZZ/2\)       \\ \hline
\(1\) & \((\ZZ/2)^3\)                                                                              & \(9\)  & \((\ZZ/2)^3\) & \(17\) & \((\ZZ/2)^3\) \\ \hline
\(2\) & \((\ZZ/2)^2\)                                                                             & \(10\) & \((\ZZ/2)^2\) & \(18\) & \((\ZZ/2)^2\) \\ \hline
\(3\) & \(\ZZ/24\)                                                                                 & \(11\) & \(\ZZ/504\)   & \(19\) & \(\ZZ/264\)  \\ \hline
\(4\) & \(\ZZ\)                                                                                    & \(12\) & \(\ZZ\)       & \(20\) & \(\ZZ\)       \\ \hline
\(5\) & \(0\)                                                                                      & \(13\) & \(0\)         & \(21\) & \(0\)         \\ \hline
\(6\) & \(0\)                                                                                      & \(14\) & \(0\)         & \(22\) & \(\ZZ/691\)   \\ \hline
\(7\) & \(\ZZ/240\)                                                                                & \(15\) & \(\ZZ/480\)   & \(23\) & \(\ZZ/65520\)  \\ \hline
\end{tabular}
\end{table}

\begin{table}[!ht]
\begin{tabular}{|c||c|c||c|c||c|}
\hline
\(k\) & \(\GW^{-\s}_k(\ZZ)\) & \(k\)    & \(\GW^{-\s}_k(\ZZ)\)           & \(k\)    & \(\GW^{-\s}_k(\ZZ)\)           \\ \hline\hline
\(0\) & \(\ZZ\)                                                                       & \(8\)  & \(0\)      & \(16\) & \(0\)       \\ \hline
\(1\) & \(0\)                                                                              & \(9\)  & \(0\) & \(17\) & \(0\) \\ \hline
\(2\) & \(\ZZ\)                                                                             & \(10\) & \(\ZZ\) & \(18\) & \(\ZZ\) \\ \hline
\(3\) & \(\ZZ/48\)                                                                                 & \(11\) & \(\ZZ/1008\)   & \(19\) & \(\ZZ/528\)  \\ \hline
\(4\) & \(\ZZ/2\)                                                                                    & \(12\) & \(\ZZ/2\)       & \(20\) & \(\ZZ/2\)       \\ \hline
\(5\) & \(\ZZ/2\)                                                                                      & \(13\) & \(\ZZ/2\)         & \(21\) & \(\ZZ/2\)         \\ \hline
\(6\) & \(\ZZ\)                                                                                      & \(14\) & \(\ZZ\)         & \(22\) & \(\ZZ \oplus \ZZ/691\)   \\ \hline
\(7\) & \(\ZZ/240\)                                                                                & \(15\) & \(\ZZ/480\)   & \(23\) & \(\ZZ/65520\)  \\ \hline
\end{tabular}
\end{table}
\end{proposition}

\begin{proof}
The only information not already present in the table of Theorem~\refthree{theorem:gwsz} is the structure of the odd torsion in \(\K_{n}(\ZZ)\) for \(n=2,6\) mod \(8\). This can be read off from the list of \(\K\)-groups~\cite{weibel-kbook}*{Example 10.3} - the only non-trivial one in this range is 
\(\K_{22}(\ZZ) = \ZZ/691\). 
\end{proof}

We now turn to the computation of the classical \(\eps\)-quadratic Grothendieck-Witt groups of \(\ZZ\). 
Recall that for \(\epsilon =\pm 1\) there is a Poincaré functor $(\Dperf(\ZZ),\QF^{\gq}_\eps) \to (\Dperf(\ZZ),\QF^{\gs}_\eps)$,
which by the fibre sequence of Corollary~\reftwo{corollary:tate-square-L} induces a cartesian square of spectra
\[
\begin{tikzcd}
\GW^{\gq}(\ZZ;\eps) \ar[d] \ar[r] & \L^{\gq}(\ZZ;\eps)\ar[d] \\ \GW^{\gs}(\ZZ;\eps) \ar[r] & \L^{\gs}(\ZZ;\eps).
\end{tikzcd}
\]
The non-negative homotopy groups of the bottom left hand spectrum were computed in Theorem~\refthree{theorem:gwsz} above. To understand the spectrum \(\GW^{ \gq}(\ZZ;\eps)\) we will calculate the homotopy groups of the cofibre of the right hand vertical map, which is equivalent to the cofibre of the left hand vertical map. We begin with the case \(\epsilon =1\).
Write \(C\) for the cofibre of the map \(\L^{\gq}(\ZZ) \to \L^{\gs}(\ZZ)\) and \(C_i\) for the homotopy group \(\pi_i (C)\).  
\begin{lemma}
\label{lemma:q-to-s-cofibre-orthogonal}%
The groups \(C_i\) are given by 
\begin{enumerate}
\item \(C_1 \cong \ZZ/2\),
\item \(C_0 \cong \ZZ/8\) 
\item
\label{item:Cminus-one-Z-two}%
\(C_{-1} \cong \ZZ/2\), 
\item \(C_i =0\) for all other values of \(i\). 
\end{enumerate}
\end{lemma}
\begin{proof}
Let us consider the commutative diagram
\begin{equation*}
\begin{tikzcd}[row sep = small]
	& \L^{ \gq}(\ZZ) \ar[dd] \ar[dr,"{\simeq_{\geq 2}}"] & \\ \L^{\qdr} (\ZZ) \ar[ur,"{\simeq_{\leq 1}}"] \ar[dr,"{\simeq_{\leq -3}}"']  & & \L^{ \s}(\ZZ) \\ & \L^{ \gs}(\ZZ) \ar[ur,"{\simeq_{\geq -2}}"'], & 
\end{tikzcd} 
\end{equation*}
where the subscript on the symbol \(\simeq\) indicates the range of dimensions \(i\) in which the map induces an isomorphism on \(\pi_i\). These ranges are obtained from Corollaries~\refthree{corollary:genuine-is-quadratic} and \refthree{corollary:genuine-is-symmetric}, in the second case using that by Example~\refthree{example:symm-quad-qf} the Poincaré structures $\QF^{\gq}=\QF^{\gev}$ and $\QF^{\gs}=\Qgen{(-1)}{}$ are $1$-symmetric and $3$-symmetric, respectively.
Using in addition that $\L^\gs_{-2}(\ZZ) =0$,
it follows that \(C_i\) is at most non-trivial in the range \(-1 \leq i \leq 1\) as claimed. 
We then find that \(C_{-1} \cong \L^{\gq}_{-2}(\ZZ) \cong \L^{\qdr}_{-2}(\ZZ) \cong \ZZ/2\). The remaining two groups sit
in the exact sequence
\[
0 \lto \L^{\gs}_1(\ZZ) \lto C_1 \lto \L^{\gq}_0(\ZZ) \lto \L^{\gs}_0(\ZZ) \lto C_0 \to 0.
\]
Since the map \(\L^{\gq}_0(\ZZ) \to \L^{\gs}_0(\ZZ)\) identifies with the multiplication by \(8\) map on \(\ZZ\) it follows that \(C_0 \cong \ZZ/8\) and \(C_1 \cong \L^{\gs}_1(\ZZ) \cong \ZZ/2\); see~\cite{Ranickiyellowbook}*{Prop 4.3.1}.
\end{proof}

\begin{remark}
Let us denote by $\L^n(R)$ the cofibre of the symmetrisation map $\L^\qdr(R) \to \L^\s(R)$, called \emph{normal} or \emph{hyperquadratic} L-theory in Ranicki's work \cite{RanickiTSO, Ranickiblue}. We then have $C \simeq \tau_{[-1,1]}\L^n(\ZZ)$.
\end{remark}

\begin{theorem}
\label{theorem:gwqz}%
The classical quadratic Grothendieck-Witt groups of $\ZZ$ are given by 
\begin{enumerate}
\item \(\GW^{\gq}_0(\ZZ) \cong \ZZ \oplus \ZZ\),
\item 
\label{item:GWqoneZ}%
\(\GW^{\gq}_1(\ZZ) \cong \ZZ/2 \oplus \ZZ/2\),
\item 
\label{item:GWqnZ}%
\(\GW^{\gq}_n (\ZZ) \cong \GW^{\gs}_n(\ZZ)\) for \(n \geq 2\). 
\end{enumerate}
\end{theorem}
\begin{proof}
The group \(\GW^{\gq}_0(\ZZ)\) is well known to be freely generated by the standard hyperbolic form and the positive definite even form \(E_8\) (see the discussion at the beginning of the section).  
For \refthreeitem{item:GWqoneZ} consider the exact sequence
\[
C_2 \to \GW^{\gq}_1(\ZZ) \to \GW^{\gs}_1(\ZZ) \to C_1 \to \GW^{\gq}_0(\ZZ) \to \GW^{\gs}_0(\ZZ).
\]
The map \(\GW^{\gq}_0(\ZZ) \to \GW^{\gs}_0(\ZZ)\) is injective and the image has index 8. It follows that \(\GW^{\gs}_1(\ZZ) \cong (\ZZ/2)^3\) maps surjectively onto \(C_1 \cong \ZZ/2\) and since \(C_2 = 0\) by Lemma~\refthree{lemma:q-to-s-cofibre-orthogonal} we get that \(\GW^{\gq}_1(\ZZ) \cong (\ZZ/2)^2\).  Finally, \refthreeitem{item:GWqnZ} is implied by Lemma~\refthree{lemma:q-to-s-cofibre-orthogonal}\refthreeitem{item:Cminus-one-Z-two}.
\end{proof}

We now turn to the case \(\epsilon = -1\).  

\begin{lemma}
\label{lemma:q-to-s-cofibre-symplectic}%
Let \(D\) be the cofibre of the map \(\L^{-\gq}(\ZZ) \to \L^{-\gs}(\ZZ)\). Then $D \simeq \Sigma^2C$.
\end{lemma}
\begin{proof}
By Proposition~\refthree{proposition:periodcat} and Remark~\refthree{remark:genuineeq}, we have canonical equivalences $\L^{-\gq}(\ZZ) \simeq \Sigma^2\L^{\gq}(\ZZ)$ and $\L^{-\gs}(\ZZ) \simeq \Sigma^2\L^{\gs}(\ZZ)$. Under these equivalences, the symmetrisation map in the definition of $D$ corresponds to the double suspension of the one in the definition of $C$.
\end{proof}

\begin{lemma}
\label{lemma:K-1Z-h}%
There are group isomorphisms
\begin{enumerate}
\item
\label{item:K-1Z-h-1}%
\(\pi_1 \K(\ZZ;\QF^{\gq}_-)_{\hC} \cong \ZZ/4\),
\item
\label{item:K-1Z-h-2}%
\(\pi_2 \K(\ZZ;\QF^{\gq}_-)_{\hC} =0.\)
\end{enumerate}
\end{lemma}
\begin{proof}
Since the involution on \(\K(\ZZ;\QF^{\gq}_-)\) only depends on the underlying duality the canonical map 
\[
\pi_n \K(\ZZ;\QF^{\gq}_-)_{\hC} \to \pi_n \K(\ZZ;\QF^{\s}_-)_{\hC}
\]
is an isomorphism, and we shall henceforth replace \(\QF^{\gq}_-\) with \(\QF^{\s}_-\).

We first compute \(\pi_1 \K(\ZZ;\QF^{\s}_-)_{\hC}\). Consider the homotopy orbit spectral sequence 
\[
E^2_{s,t} = \mathrm{H}_s(\Ct;\pi_t \K(\ZZ;\QF^{\s}_-))  \implies \pi_{s+t}\K(\ZZ;\QF^{\s}_-)_{\hC}.
\]
Since \(\mathrm{H}_2(\Ct;\pi_0 \K(\ZZ;\QF^{\s}_-)) = 0\) the generator of \(\mathrm{H}_0(\Ct;\pi_1 \K(\ZZ;\QF^{\s}_-)) \cong \ZZ/2\) is 
not an eventual boundary. The group \(\pi_1 \K(\ZZ;\QF^{\s}_-)_{\hC}\) also gets a contribution from \(\mathrm{H}_1(\Ct;\pi_0 \K(\ZZ;\QF^{\s}_-))\) which has order \(2\), so in total it must have order \(4\). There is furthermore an exact sequence 
\[
\L^{-\s}_2(\ZZ) \to \pi_1 \K(\ZZ;\QF^{\s}_-)_{\hC} \to \GW_1^{-\s}(\ZZ),
\]
where the left hand group is isomorphic to \(\L^{\s}_0(\ZZ) \cong \ZZ\) and the right hand group is trivial by Table~\refthree{table:gw}. It follows that the middle group is cyclic and is hence isomorphic to \(\ZZ/4\).

We will now compute $\pi_2 \K(\ZZ;\QF^{\s}_-)_{\hC}$. For this, it will be useful to embed \(\ZZ\) in the field \(\RR\) of real numbers, and consider the \emph{topological} variants of \(\K\)-theory and \(\GW\)-theory for \(\RR\), equipped with its usual topology. For this we follow the approach of~\cite{schlichting-derived}*{\S 10} and define these in terms of the simplicial ring \(\RR^{\Del^{\bullet}} \in \Fun(\Del\op,\Ring)\), whose \(n\)-simplices are the set \(\RR^{\Del^n}\) of continuous maps of topological spaces \(|\Del^n| \to \RR\), considered as a ring via pointwise operations. One then defines the topological variants of \(\K\)-theory, \(\GW\)-theory and \(\L\)-theory by
\[
\K^{\topo}(\RR) := |\K(\RR^{\Del^{\bullet}})| = \displaystyle\mathop{\colim}_{n \in \Del\op}\K(\RR^{\Del^n}) \in \Spa
\]
\[
\GW^{\topo}(\RR;\QF^{\s}_\eps) := |\GW(\RR^{\Del^{\bullet}};\QF^{\s}_\eps)| = \displaystyle\mathop{\colim}_{n \in \Del\op}\GW(\RR^{\Del^n};\QF^{\s}_\eps) \in \Spa
\]
and
\[
\L^{\topo}(\RR;\QF^{\s}_\eps) := |\L(\RR^{\Del^{\bullet}};\QF^{\s}_\eps)| = \displaystyle\mathop{\colim}_{n \in \Del\op}\L(\RR^{\Del^n};\QF^{\s}_\eps) \in \Spa  .
\]
The construction above furnishes a natural map of spectra \(\K^{\topo}(\RR) \to \ko = \ZZ \times \BGL^{{\topo}}_{\infty}(\RR)\) which is an equivalence by~\cite{schlichting-derived}*{Proposition 10.2}. Similarly, by the same proposition 
\(\GW^{\topo}_0(\RR;\QF^{\s}_\eps) \cong \GW_0(\RR;\QF^{\s}_\eps)\) and \(\tau_{\geq 1}\GW^{\topo}(\RR;\QF^{\s}_\eps)\) is naturally equivalent to \(\BO^{{\topo}}_{\infty,\infty}(\RR)\) when \(\eps = 1\) and to \(\BSp^{{\topo}}_{\infty}(\RR)\) when \(\eps=-1\). The superscript top indicates that we topologise the groups as sequential colimits of Lie groups. In addition, by~\cite{schlichting-derived}*{Remark 10.4} the natural map \(\L(\RR;\QF^{\s}_\eps) \to \L^{\topo}(\RR;\QF^{\s}_\eps)\) is an equivalence.

We now claim that the map \(\K(\ZZ) \to \K^{\topo}(\RR) \simeq \ko\) induces isomorphisms on \(\pi_i\) for \(i \leq 2\). The groups in question are in fact isomorphic, furthermore the composite $\SS \to \K(\ZZ) \to \ko$ is an isomorphism in the claimed range, so the result follows.

Let us write \(\K^{\topo}(\RR;\QF^{\s}_-)\) for the spectrum \(\K^{\topo}(\RR)\) considered together with the \(C_2\)-action induced by the duality associated to \(\QF^{\s}_-\). Since taking homotopy orbits preserves connectivity we get from the above that the map \(\pi_i \K(\ZZ;\QF^{\s}_-)_{\hC} \to \pi_i \K^{\topo}(\RR;\QF^{\s}_-)_{\hC}\) is an isomorphism for \(i \leq 2\). To finish the proof it will hence suffice to show that \(\pi_2\K^{\topo}(\RR;\QF^{\s}_-)_{\hC}\) vanishes.
Since geometric realisations preserve fibre sequences of spectra, the latter group sits in an exact sequence
\[
\L^{\topo}_3(\RR;\QF^{\s}_-) \lto \pi_2 \K^{\topo}(\RR;\QF^{\s}_-)_{\hC} \lto \GW^{\topo}_2(\RR;\QF^{\s}_-) .
\]
Since \(\RR\) is a field we have that \(\L^{\topo}_3(\RR;\QF^{\s}_-) \cong \L_3(\RR;\QF^{\s}_-) \cong 0\) and since \(\GW^{\topo}_2(\RR;\QF^{\s}_-) \cong \pi_1\Sp_\infty^\topo(\RR) \cong \ZZ\) it follows that the group \(\pi_2 \K^{\topo}(\RR;\QF^{\s}_-)_{\hC}\) is free. But from the homotopy orbit spectral sequence we see that it has order at most 4, and so we conclude that it is trivial.
\end{proof}

\begin{remark}
By Karoubi periodicity (as formulated e.g.\ in Corollary~\reftwo{corollary:karoubi-skew-periodicity}), we know that $\K(\ZZ;\QF^{s}_-) \simeq \SS^{2-2\sigma}\otimes \K(\ZZ;\QF^{\s})$ as spectrum with $\Ct$-action. Furthermore, the map $\K(\ZZ) \to \ko$ is $\Ct$-equivariant with respect to the $\Ct$-action induced by $\QF^{\s}$ on $\K(\ZZ)$ and the trivial action on $\ko$. The above lemma is then a statement about low dimensional homotopy groups of $(\SS^{2-2\sigma}\otimes\ko)_{\hC}$. These can also be computed using the cofibre sequence $\Ct_+ \to S^0 \to S^\sigma$ and some elaborations thereof.
\end{remark}

\begin{theorem}
\label{theorem:gw-gqz}%
There are isomorphisms
\begin{enumerate}
\item 
\label{item:gw-gqz-1}%
\(\GW^{-\gq}_0(\ZZ) \cong \ZZ \oplus \ZZ/2\),
\item
\label{item:gw-gqz-2}%
\(\GW^{-\gq}_1(\ZZ) \cong \ZZ/4\),
\item
\label{item:gw-gqz-3}%
\(\GW^{-\gq}_2(\ZZ) \cong \ZZ\)
\item
\label{item:gw-gqz-4}%
\(\GW^{-\gq}_3(\ZZ) \cong \ZZ/24\)
\item
\label{item:gw-gqz-5}%
\(\GW^{-\gq}_i (\ZZ) \cong \GW^{-\gs}_i (\ZZ)\) for \(i \geq 4\).
\end{enumerate}
\end{theorem}
We remark that statements \refthreeitem{item:gw-gqz-2} and \refthreeitem{item:gw-gqz-3} have been shown previously by Krannich and Kupers using geometric methods, see \cite{KK}.
\begin{proof}
Part~\refthreeitem{item:gw-gqz-5} follows immediately from Lemma~\refthree{lemma:q-to-s-cofibre-symplectic} and Lemma~\refthree{lemma:q-to-s-cofibre-orthogonal}\refthreeitem{item:Cminus-one-Z-two}. Part~\refthreeitem{item:gw-gqz-1} is well known, see the discussion at the beginning of the section. For Part~\refthreeitem{item:gw-gqz-3} it suffices to note that by Lemma~\refthree{lemma:K-1Z-h} the map \(\GW^{-\gq}_2(\ZZ) \to \L^{-\gq}_2(\ZZ) \cong \L^{\gq}_0(\ZZ)\cong \ZZ\) is injective with finite cokernel.

Now to show \refthreeitem{item:gw-gqz-2} consider the following commutative diagram with exact rows:
\[
\begin{tikzcd} 
	\L^{-\gq}_2(\ZZ) \ar[d] \ar[r] & \pi_1 \K(\ZZ;\QF^{\gq}_-)_{\hC} \ar[d,"\cong"] \ar[r] & \GW^{-\gq}_1(\ZZ) \ar[d] \ar[r]& \L^{-\gq}_1(\ZZ) \ar[d] \\ 
	\L^{-\gs}_2(\ZZ)  \ar[r] & \pi_1 \K(\ZZ;\QF^{\gs}_-)_{\hC} \ar[r] & \GW^{-\gs}_1(\ZZ) \ar[r]& \L^{-\gs}_1(\ZZ) 
\end{tikzcd}
\]
Since \(\GW^{-\gs}_1(\ZZ) = 0\) by Table~\eqrefthree{table:gw} the bottom left hand map must be surjective. As in the proof of Lemma~\refthree{lemma:q-to-s-cofibre-symplectic}, the map \(\L^{-\gq}_2(\ZZ) \to \L^{-\gs}_2(\ZZ)\) identifies with the map \(\L^{\gq}_0(\ZZ) \to \L^{\gs}_0(\ZZ)\) and hence with the inclusion \(8\ZZ \hrar \ZZ\). Since \(\pi_1 \K(\ZZ,\QF^{\gq}_-)_{\hC} \cong \ZZ/4\) the upper left hand map must be \(0\). In addition \(\L^{-\gq}_1(\ZZ)\cong \L^{-\qdr}_1(\ZZ)\cong 0\) and so the upper middle map gives an isomorphism \(\GW^{-\gq}_1(\ZZ) \cong \ZZ/4\) by Lemma~\refthree{lemma:K-1Z-h}\refthreeitem{item:K-1Z-h-1}.

Finally, to prove \refthreeitem{item:gw-gqz-4} consider the commutative diagram
\[
\begin{tikzcd} 
	\L^{-\gq}_4(\ZZ) \ar[d] \ar[r] & \pi_3 \K(\ZZ;\QF^{\gq}_-)_{\hC} \ar[d,"\cong"] \ar[r] & \GW^{-\gq}_3(\ZZ) \ar[d] \ar[r]& \L^{-\gq}_3(\ZZ) \ar[d] \\ 
	\L^{-\gs}_4(\ZZ)  \ar[r] & \pi_3 \K(\ZZ;\QF^{\gs}_-)_{\hC} \ar[r] & \GW^{-\gs}_3(\ZZ) \ar[r,two heads]& \L^{-\gs}_3(\ZZ) 
\end{tikzcd}
\]
where the bottom right map is surjective by Lemma~\refthree{lemma:K-1Z-h}\refthreeitem{item:K-1Z-h-2}. 
Then \(\L^{-\gq}_3(\ZZ) \cong \L^{\gq}_1(\ZZ) = 0\) and \(\L^{-\gq}_4(\ZZ) \cong \L^{-\gs}_4(\ZZ)\cong \L^{\gs}_2(\ZZ) \cong 0\), which implies that the top middle horizontal map in the above diagram is an isomorphism and the bottom middle horizontal map is injective with cokernel \(\L^{-\gs}_3(\ZZ) \cong \L^{-\s}_3(\ZZ) \cong \ZZ/2\); see Proposition~\refthree{proposition:compute-L-groups}. Since \(\GW^{-\gs}_3(\ZZ) \cong \ZZ/48\) by Table~\eqrefthree{table:gw} this implies that \(\GW^{-\gq}_3(\ZZ) \cong \ZZ/24\), as claimed.
\end{proof}

\begin{remark}
\label{remark:BK-conjecture-number-rings}%
By Proposition~\refthree{proposition:BK-conjecture} and Corollary~\refthree{corollary:dedekind-gw}, we know that for the ring of integers $\nO$ in a number field $F$, the canonical map 
\[
\GW^{\s}_\cl(\nO;\eps) \lto \GW^{\s}_\cl(\nO[\tfrac{1}{2}];\eps)
\]
is a 2-local equivalence in degrees $\geq 1$.
In principle, one can then use the results of \cite{hermitian-dedekind} to calculate the 2-local Grothendieck-Witt groups of $\nO$. As before, the $p$-local homotopy for odd $p$ is controlled by the isomorphisms
\[
\GW^{\s}_{\cl,n}(\nO;\eps)[\tfrac{1}{2}] \cong \K_n(\nO;\eps)[\tfrac{1}{2}]_{\Ct} \oplus \L^{\s}_n(\nO;\eps)[\tfrac{1}{2}].
\]
One can then use Lemma~\refthree{lemma:kz-inv} which determines the $\Ct$-action on the $2$-inverted $\K$-groups of $\nO$. Moreover, by Corollary~\refthree{corollary:compute-L-groups-global-field}, the 2-inverted L-groups are only non-zero for $n = 4k$ and in this case are a free $\ZZ[\tfrac{1}{2}]$-module of rank equal to the number of real embeddings of the number field $F$, see \cite{milnor-symmetric}*{Chapter IV.4}.
\end{remark}

\end{document}